\documentclass[12pt]{amsart}
\usepackage{amscd}
\usepackage{mathrsfs}
\newtheorem{theorem}{Theorem}[section]
\newtheorem{lemma}[theorem]{Lemma}
\newtheorem{corollary}[theorem]{Corollary}
\theoremstyle{definition}

\newtheorem{example}[theorem]{Example}
\newtheorem{remark}[theorem]{Remark}
\numberwithin{equation}{section}

\def\H{\mathscr H}  
\def\K{\mathscr K} 
 
\def\D{\mathscr D}
\def\C{\mathscr C} 
\def\G{\mathscr G}
\def\R{\mathscr R}

\def\E{\mathsf E}
\def\B{\mathsf B}

\def\Z{\mathsf Z}
\def\Ex{\mathsf{Ext}}
\def\ExP{\mathsf{Ext}_0}
\def\ExM{\mathsf{Ext}_{\text{\rm M}}}
\def\ExN{\widetilde{\mathsf{Ext}}}

\def\RE{\mathbb R}

\def\Ran{\mathscr R}

\def\uno{\mathsf 1}
\def\zero{\mathsf 0}
\def\fh{\mathfrak h}

\def\t{\tilde}

\def\e{\text{\rm e}}
\def\ccdot{\!\cdot\!}

\def\max{\text{\rm max}}
\def\min{\text{\rm min}}
\def\diff{\nabla\!\cdot\! a\nabla}

\def\be{\begin{equation}}
\def\ee{\end{equation}}
\def\ba{\begin{align}}
\def\ea{\end{align}}

\newsymbol\dotplus 1275
\newsymbol\subsetneq 2328

\begin{document}

\title[Markovian  Extensions of Elliptic Operators]{Markovian  Extensions of 
Symmetric Second Order Elliptic Differential Operators}
\author{Andrea Posilicano}
\address{DiSAT - Sezione di Matematica,  Universit\`a
dell'Insubria, I-22100 Como, Italy}

\email{posilicano@uninsubria.it}

%\keywords{Markovian Operators, Self-Adjoint Extensions, Kre\u \i n's Resolvent Formula,
%Elliptic Differential Operators}
%\thanks{{\it Mathematics Subject Classification (2000).} 47B25
%(primary), 
%47B38, 35J25 (secondary)}

\begin{abstract}
Let $\Omega\subset\RE^n$ be bounded with a smooth boundary $\Gamma$ and let $S$ be the symmetric operator in $L^2(\Omega)$ given by the minimal realization of a second order elliptic differential operator. We give a complete classification of the 
Markovian self-adjoint extensions of $S$ by providing an explicit one-to-one correspondence between such extensions and the class of Dirichlet forms in $L^2(\Gamma)$ which are additively decomposable by the bilinear form of the Dirichlet-to-Neumann operator plus a Markovian form.  By such a result two further equivalent classifications are provided: the first one is expressed in terms of an additive decomposition of the bilinear forms associated to the extensions, the second one uses the additive decomposition of the resolvents provided by Kre\u\i n's formula.  The Markovian part of the decomposition allows to characterize 
the operator domain of the corresponding extension in terms of Wentzell-type boundary conditions. Some properties of the extensions, and of the corresponding  Dirichlet forms, semigroups and heat kernels, like locality, regularity, irreducibility, recurrence, transience, ultracontractivity and Gaussian bounds are 
also discussed.  
\end{abstract}

{\maketitle }

\section{Introduction}
A negative self-adjoint
operator $A$ on the real Hilbert space $L^2(X)$ is said to be Markovian if the semi-group $e^{tA}$,
$t\ge 0$, is positivity-preserving and is a contraction in
$L^2(X)\cap L^\infty(X)$. In 1959, in the seminal paper
\cite{[BD]}, Beurling and Deny discovered the connection between the
Markov property for symmetric semi-groups and the contractivity
property for Dirichlet spaces. Later, in 1970, Fukushima (see \cite{[Fu70]}) found the connection between regular Dirichlet forms and
symmetric Hunt Markov  processes, thus opening the way to the deep
interplay between Dirichlet spaces and probability theory,  and
providing the analogue, in a  Hilbert $L^2(X)$ space setting, of the
well known connection between Feller Markov processes and Feller
(i.e. strongly continuous, positivity preserving and contracting)
semi-groups in $C_b(X)$, the Banach space of bounded continuous
functions on  $X$ (see e.g. \cite{[Silv1]}, \cite {[Fu]} and
\cite{[FOT]} for a thorough introduction to Dirichlet forms and
symmetric Markov processes). Beside the probabilistic side, the
Markovian property helps the study of the deep connections between analytic properties of the
semi-groups and their generators as  logarithmic Sobolev inequalities, ultracontractivity and 
heat kernel
estimates (see e.g. \cite{[Dv]}, \cite{[Gr]}, \cite{[Wang]} and references therein). Therefore it
is worthwhile to find conditions guaranteeing the Markovian nature of
a given self-adjoint operator. In particular,  when $X=\Omega\subset
\RE^n$ is a bounded domain with a  smooth boundary $\Gamma$, we are interested in characterizing Markovian
self-adjoint extensions of the minimal realization of a given
symmetric second order (by positive maximum  principle, for Markov
generators the order can not be higher) elliptic operator in terms of
boundary conditions.\par The connections between boundary conditions
and Markov property have a long history. Here a brief abstract.\par In
1957, in the paper \cite{[Fel]}, Feller classified all Markovian
self-adjoint realizations of symmetric (generalized) second order
differential  operators in $L^2(a,b)$, $(a,b)\subset\RE$, in terms of
boundary conditions.  Such conditions (see \cite{[Fel]}, Theorem 10.2)
are explicitly expressed in terms of certain inequalities on the
coefficients of the $2\times 2$ real symmetric matrix $B$ describing the boundary 
conditions at $\{a\}\cup\{b\}$. It is easy to
check (see Section 3 below) that such inequalities coincide, in such 
simple 2-dimensional Dirichlet spaces setting, with the necessary and
sufficient conditions guaranteeing that the bilinear form associated with $B$ is a Dirichlet form on $\RE^2$. Thus Feller's results can be
re-phrased in terms of a correspondence between Markovian
self-adjoint extensions of a symmetric second order differential
operators on an interval and Dirichet forms on its boundary.\par In
1959, in the paper \cite{[we]}, Wentzell, aiming at extending Feller's
results to higher dimensions, sought the most general boundary
conditions which  restrict a given elliptic second order differential
operator in a bounded domain $\Omega\subset\RE^n$ to a generator of a
Feller Markov  process, and hence of a Feller semigroup, in
$C_b(\Omega)$. Wentzell's results, in the realm of Feller
semi-groups, have been extended and clarified in a lot of successive
papers by many  authors (see e.g. \cite{[U1]}, \cite{[U2]},
\cite{[BCP]}, \cite{[Ta]} and references therein).  Since Wentzell's
framework is a Banach space one, his results can not  be directly
re-phrased in a Dirichlet space language, which requires an  Hilbert
space setting. However looking back at Feller's results, and noticing
that the boundary operator entering in Wentzell's conditions (see
\cite{[we]}, formula (3)) appears to have (in the
$L^2(\Omega)$-symmetric case) an associated bilinear form resembling
the ones furnished by the Beurling-Deny decomposition for regular
Dirichlet forms, the
suggestion is clear: there should be a correspondence
between Markovian self-adjoint extensions of a symmetric second order
differential  operators in $L^2(\Omega)$ and Dirichet forms in
$L^2(\Gamma)$.  The self-adjoint  operators associated with  the
Dirichlet forms on the boundary should  then realize Wentzell-type
boundary conditions. \par  In 1969, Fukushima (see the paper
\cite{[Fu69]}), given the resolvent density $R^D_\lambda(x,y)$
corresponding to absorbing barrier Brownian motion on a 
bounded domain $\Omega$ (i.e. the resolvent kernel of the Dirichlet
Laplacian on $\Omega$), considered the family of all conservative symmetric Markovian
resolvent densities on $\Omega$ of the kind 
\begin{equation}\label{FR}R_\lambda(x,y)=R^D_\lambda(x,y)+H_\lambda(x,y)\,,
\end{equation}  
where  $H_\lambda(x,y)$ is a positive function of $\lambda>0$, $\lambda$-harmonic in $x$ for each $\lambda$ and $y$, such that for any compact $K\subset\Omega$, $\sup_{x\in K,y\in\Omega}H_\lambda(x,y)$ is finite.
% and such that, for any compact $k\subset D$, $\sup_{(x,y)\in k\times D}H_\lambda(x,y)<\infty$. 
By Dirichlet spaces and potential theory
analysis, Fukushima found a correspondence between such a
family and a class of Dirichlet spaces on the
Martin boundary of  $\Omega$ (Martin boundary coincides with the topological one in the case $\Gamma$ is Lipschitz). The generalized  Laplacians
corresponding to such family of resolvents turn out to be
characterizable in terms of boundary conditions involving the notion
of (generalized) normal derivative in Doob's sense (see formula 6.8 in
\cite{[Fu69]}). Fukushima's results have been extended to more general
elliptic operators by Kunita (see \cite{[kun]}) and successively, by
Silverstein (see \cite{[Silv1]}, \cite{[Silv]}, also see \cite{[ES1]} and \cite{[ES2]}) and LeJan (see
\cite{[LJ]}), to general Markovian operators. In particular
Silverstein found a characterization, again in terms of Dirichlet
spaces on the boundary,  of the Markovian resolvents $R_\lambda\ge
R^0_\lambda$ dominating a given one $R^0_\lambda$. For recent
developments of boundary theory  of Dirichlet forms and symmetric
Markov processes, we refer to the book \cite{[CF]} by Chen and Fukushima.\par
Here our approach is different from the ones described above:
we build on the theory of self-adjoint extensions as initiated  by
Kre\u\i n {[K]}, Vi\u sik \cite{[V]}, Birman \cite{[B]} and Grubb
\cite{[G68]}. In particular Grubb characterized all self-adjoint
extensions  of a given symmetric elliptic differential operator on a
domain $\Omega$ with a smooth boundary in terms of (non-local)
boundary conditions. Thus Wentzell-type boundary conditions, and their
generalization due to the Dirichlet space  approach initiated by
Fukushima, should be part of Grubb's results. In recent years there
has been a renovated interest for the connection between theory of
self-adjoint extensions and boundary conditions for partial
differential operators due to its re-formulation in terms of Kre\u\i
n's resolvent formula (see \cite{[P08]}, \cite{[MM]}, \cite{[Ry]}, \cite{[BMNW]},
\cite{[P10]}, \cite{[G08]}, \cite{[BGW]}, \cite{[MMM]},  
%\cite{[GM2]}, \cite{[GM1]},
\cite{[GM3]} and references therein). Let us notice
that \eqref{FR} (taking into account the characterization of the
harmonic part $H_\lambda$ given by formulae (4.5)-(4.7) in
\cite{[Fu69]}) has the same structure as Kre\u\i n's formula for the
resolvent of a self-adjoint extensions of the minimal Laplacian. This
indicates that Fukushima's results can be re-phrased in terms of such a 
resolvent formula (see Theorem \ref{finaleres} and Remark \ref{4.25}).  
\par The content of this paper is the
following. Section 2 is of preliminary nature. Here at first we recall
the theory of self-adjoint extensions of a given symmetric operator $S$,
following the simple approach presented in \cite{[P08]} (building on
previous paper \cite{[P01]}), to which we refer for proofs and
relations with other equivalent methods. Then we provide the bilinear
forms associated  with  the self-adjoint extensions and recall the
connection between Markovian generators and Dirichlet forms, thus
reducing the problem of the search of Markovian extensions to the
characterization of self-adjoint extensions having associated bilinear
forms which are  Dirichlet forms.    The section is concluded recalling the correspondence between  regular Dirichlet forms and Hunt Markov processes and the connections between path properties of such processes and analytical properties of the corresponding Dirichlet forms like 
conservativeness, transience, recurrence and irreducibility.
    \par In Section 3, to enhance reader's intuition, we consider the toy example given by
$\frac{d^2}{dx^2}$ on the real interval $(0,\ell)$, re-obtaining, by
straightforward considerations about Dirichlet forms on $\RE^2$,
Feller's results (as given in \cite{[Fel]}, Theorem 10.2). This simple
example is instructive since it permits to introduce, in a simpler
finite dimensional setting, many of  the results that will be then obtained in the successive section.
\par In Section 4  we extend the construction given in the
previous one to  the case $S=A_{\min}$, where $A_{\min}$ denotes the minimal realization of an elliptic second order differential
operator on a bounded domain $\Omega$ with a smooth boundary $\Gamma$. By a result due 
to Fukushima and Watanabe (which we recall in Theorem \ref{maximum}) the maximal element (with respect to the semi-order induced by the
associated bilinear forms) of the 
set $\ExM(A_{\min})$ of Markovian self-adjoint extensions of $A_{\min}$ is the Neumann realization $A_{N}$. As an immediate consequence of such a result any Markovian self-adjoint extension of $A_{\min}$ satisfies a logarithmic Sobolev inequality; hence its semigroup is ultracontractive and Gaussian heat kernel estimates hold  (see Corollaries  \ref{ultra} and \ref{gaussian}). By Fukushima-Watanabe theorem  we are led to  consider 
the set $\widetilde{\Ex}(A_{\min})\supseteq  {\ExM}(A_{\min})$ of
self-adjoint extensions sandwiched 
 between the Dirichlet realization $A_{D}$ (the minimun element of  ${\ExM}(A_{\min})$) and the
Neumann one $A_{N}$. 
In Theorem \ref{sub} we give a simple recipe to define, in terms of bilinear forms $f_{\Pi,B}$ on $L^{2}(\Gamma)$, bilinear forms corresponding to 
extensions in $\widetilde{\Ex}(A_{\min})$. Such extensions belong to $\ExM(A_{\min})$ whenever the corresponding forms $f_{\Pi,B}$ are Dirichlet forms. 
By combining  Theorem \ref{sub} with its converse (see Theorem \ref{converse}), we finally obtain Theorem \ref{finale} which provides a one-to-one correspondence between 
$\ExM(A_{\min})$ and the class of Dirichlet forms of $L^{2}(\Gamma)$ which admit a decomposition in terms of the Dirichlet form corresponding to (minus) the Dirichlet-to-Neumann 
operator on $\Gamma$ plus a Markovian form \footnote{Our definition of Markovian form is 
stronger than the usual one; the two definitions coincide in the case of a Dirichlet form (see Remark \ref{coincide} below)}. Such a result is the analogue, in our framework, of the correspondence established by Fukushima in \cite{[Fu69]}. In the case one could prove that the Markovian component is always closable (we conjecture that this is the case) one should obtain an even simpler correspondence (see Remark \ref{closed}).  In any case by Theorem \ref{sub} this latter correspondence holds in one direction, thus providing simple sufficient conditions leading to Markovian extensions. Theorem \ref{finale} has a simpler,  equivalent version   in terms of bilinear forms (see Theorem \ref{finale2}): 
\par\noindent
{\it Let $F$ be a closed bilinear form on $L^{2}(\Omega)$. Then $F=F_{A}$, $A\in\ExM(A_{\min})$, if and only if $\D(F)\subseteq H^{1}(\Omega)$ and there exists a Markovian form $f_{b}$ on  $L^{2}(\Gamma)$ such that 
$F(u,v)=F_N(u,v)+f_{b}(\gamma_{0}u,\gamma_{0}u)$.} 
\par\noindent
Here $F_{N}$ is the bilinear form associated to the Neumann realization $A_{N}$ and $\gamma_{0}$ is the trace (evaluation) map at $\Gamma$. Moreover, by Kre\u\i n's formula, a version of Theorem \ref{finale2} in terms of resolvents can also be given  (see Theorem \ref{finaleres}). This provides our version of Fukushima's \eqref{FR} (see Remark \ref{4.25}).\par
In section 5 we look for the boundary conditions associated to a Markovian extension. By Theorem \ref{tww} (also see Corollary \ref{cww}) these are defined in terms of the Markovian 
form $f_{b}$ appearing in the decomposition provided by Theorem \ref{finale}. In the case such a Markovian component is a regular Dirichlet form, then, by Beurling-Deny decomposition, these boundary conditions resemble 
the ones obtained in Wentzell's seminal paper \cite{[we]} (see Remark \ref{rww}). Some examples are provided.\par
We conclude with a remark about our regularity assumptions.  The hypothesis on the smoothnes of the boundary of $\Omega$ can be relaxed to $C^{1,1}$ (i.e. to boundaries which locally are the graph of a differentiable function having Lipschitz derivatives): by using the results contained in \cite{[P10]} and \cite{[G08]} all the statements here presented hold in this more general setting, the proof being essentially the same. By using the results contained in \cite{[GM3]}, \cite{[GMN]} and \cite{[BM]},  we expect that our results can be further generalized to hold on domains with a Lipschitz boundary. 
Lastly let us notice that the Fukushima-Watanabe theorem has been recently extended by Robinson and Sikora (see \cite{[RS]}, Theorem 1.1) to the case of an elliptic operator on an arbitrary open bounded set $\Omega$; the final goal should be to generalize the results here presented to the case of such sets, replacing the topological boundary with the Martin one.
 %; see Remark \ref{RS} for some immediate consequences of such a result, in particular the one regarding Gaussian heat kernel estimates on sets with the extension property.  
%\end{section}
\section{Preliminaries} 
\subsection{Notations}
\begin{itemize}
\item 
$\H$, $\fh$ denote Hilbert spaces with scalar products
  $\langle\cdot,\cdot\rangle$,  $(\cdot,\cdot)$ and corresponding
  norms $\|\cdot\|$, $|\cdot|$;
\item 
Given a linear operator $L$ we denote by
$\D(L)$, $\K(L)$, $\R(L)$, $\G(L)$, $\rho(L)$ its
domain, kernel, range, graph and resolvent set respectively; 
\item
For any $z\in\rho(L)$ we pose $R^L_z:=(-L+z)^{-1}$;
\item 
$L|{\mathscr V}$ denotes the restriction of $L$ to the subspace
  ${\mathscr V}\subset\D(L)$ and we pose $L{\mathscr V}:=
  \R(L|{\mathscr V})$;
\item
$\Ex(S)$ denotes the (potentially empty) set of all self-adjoint extensions of the 
symmetric operator $S$ and $\ExM(S)\subseteq \Ex(S)$ denotes the subset of Markovian self-adjoint extensions;
\item 
Given an orthogonal projector $\Pi:\fh\to\fh$, we use the same symbol
$\Pi$ to denote both the injection $\Pi|\R(\Pi):\R(\Pi)\to\fh$  and
the surjection $(\Pi|\R(\Pi))^*:\fh\to\R(\Pi)$;
\item 
$\E(\fh)$ denotes the set of couples $(\Pi,\Theta)$, where 
$\Pi$ is an orthogonal projection in $\fh$ and $\Theta$ 
is a self-adjoint operator in the Hilbert space $\R(\Pi)$;
\item
Given $v\in\fh$, $|v|=1$, the orthogonal projector $v\otimes v$ is defined by $[v\otimes v] (u):=(v,u)\,v$;
\item
Given a sesquilinear (bilinear in the real case) 
form $F$,  we pose $F(\phi):=F(\phi,\phi)$ for the corresponding
quadratic form;
\item 
$F_A$ denotes the symmetric sesquilinear (bilinear in the real case) form associated  with  the self-adjoint operator $-A$ in the Hilbert space $\H$;
\item 
$f_{\Pi,\Theta}$, $(\Pi,\Theta)\in\E(\fh)$, denotes the
  symmetric sesquilinear (bilinear in the real case) form associated  with  the self-adjoint operator $\Theta$ in the Hilbert space $\R(\Pi)$ and $f_{\Theta}\equiv f_{\uno,\Theta}$;
\item
$\B(X,Y)$ denotes the set of linear bounded operator on 
$X$ to $Y$ and $\B(X)\equiv\B(X,X)$;
\item
$1_B$ denotes the characteristic function of the set $B$;
\item 
$u\wedge v:=\min\{u,v\}$, $u\vee v:=\max\{u,v\}$, $u_\#:=(0\vee u)\wedge 1$. 
\end{itemize}
%\end{subsection}
\subsection{Self-adjoint Extensions}
Let  
\be
S:\D(S)\subseteq\H\to\H
\ee
be a  symmetric linear operator such that $-S\ge \lambda_0>0$, i.e.  
\be
\forall\,\phi\in\D(S)\,,\quad\langle -S\phi,\phi\rangle\ge \lambda_0\|\phi\|^2\,.
\ee 
By Friedrichs' theorem $S$ 
has a self-adjoint extension 
\be 
A_0:\D(A_0)\subseteq\H\to\H
\ee 
such that $-A_0\ge\lambda_0$. $A_0$ is the unique self-adjoint 
extension such that $\D(A_0)\subseteq \D(F_0)$, where $\D(F_0)$ denotes the domain of $F_0$, the closure of the positive sesquilinear form 
$F^{\circ}(\phi,\psi):=\langle -S\phi,\psi\rangle$.\par 
From now on we suppose that $S$ is not
essentially self-adjoint. Thus, by von Neumann's theory of
self-adjoint extensions, there exists an unitary operator 
$U_{0}:\K(S^*-i)\to\K(S^*+i)$, such that 
\be
\H_{0}=\D(S^{**})\oplus\G(U_{0})\,,
\ee
\be
A_0(\phi+\psi+U_{0}\psi)=S\phi+i\psi-iU_{0}\psi\,,
\ee
where $\H_0$ denotes the Hilbert space given by $\D(A_0)$ equipped with the scalar product (giving rise to the graph norm)
\be
\langle\phi,\psi\rangle_{0}:=\langle A_0\phi,A_0\psi\rangle+\langle\phi,\psi\rangle\,.
\ee
Therefore the closure of $S$, i.e. $S^{**}$, 
is the restriction of $A_0$ to the kernel of the orthogonal projection from $\H_{0}$ onto $\K(S^*-i)$. Thus, since 
this gives some advantages in practical applications, we will look for the
self-adjoint extensions of $S$ by considering the equivalent 
problem of the search of the self-adjoint extensions of the
restriction of $A_0$ to the kernel $\K(\tau)$, which we suppose coinciding with 
$\D(S^{**})$, 
of a surjective bounded linear operator
\be
\tau:\H_{0}\to\fh \,.
\ee
%Here 
%$\fh$ is an auxiliary Hilbert space with scalar product 
%$(\cdot,\cdot)$ and corresponding norm $|\cdot|$. 
Typically $A_0$ is an elliptic differential operator, $\tau$ 
is some trace (restriction) operator along a null subset $N$ 
and $\fh$ is some Hilbert space of functions on $N$.\par 
For notational convenience let us introduce, for any $z\in\rho(A_0)$, the following bounded linear operators:
\be
R^0_z:\H\to\H_0\,,\quad R^0_z:=(-A_0+z)^{-1}
\ee
and 
\be
G_z:\fh\to\H\,,\quad G_z:=\left(\tau R^0_{\bar z}\right)^*\,.
\ee
By the results provided in \cite{[P08]} (building on previous results
obtained in \cite{[P01]} and \cite{[P03]}) one has the following
\begin{theorem}\label{estensioni} 
(i) The set $\Ex(S)$ is parametrized by 
$\E(\fh)$. \par\noindent (ii) Let $A_{(\Pi,\Theta)}$ be the self-adjoint  extension corresponding to $(\Pi,\Theta)\in \E(\fh)$.  Then 
\begin{equation}\label{A}
A_{(\Pi,\Theta)}:\D(A_{(\Pi,\Theta)})\subseteq\H\to\H\,,\quad
A_{(\Pi,\Theta)}\phi:=A_0\phi_0\,,
\end{equation}
\begin{align}\label{D}
\D(A_{(\Pi,\Theta)})
:=\left\{\phi=\phi_0+G_0
\xi_\phi\,,\ \phi_0\in \D(A_0)\,,\, \xi_\phi\in
\D(\Theta)\,,\,\Pi\tau
\phi_0=\Theta\xi_\phi
\right\}.
\end{align}
(iii) The resolvent $R_z^{(\Pi,\Theta)}:=(-A_{(\Pi,\Theta)}+z)^{-1}$ of $A_{(\Pi,\Theta)}$ is given, for any 
$z\in\rho(A_{0})\cap\rho(A_{(\Pi,\Theta)})$, by the Kre\u\i n's type formula
\begin{align}\label{R}
R_z^{(\Pi,\Theta)}=R^0_z+G_z
\Pi(\Theta+z\Pi G^*_0 G_z\Pi)^{-1}\Pi G^*_{\bar z}\,.
\end{align}
\end{theorem}
\begin{remark}
The operator $G_z$ is injective (by surjectivity of $\tau$) 
and  by Lemma 2.1 in \cite{[P03]}, given the surjectivity hypothesis
$\R(\tau)=\fh$ the density one $\overline{\K(\tau)}=\H$ is 
equivalent to
\be\label{1.1}
\forall z\in\rho(A_0)\,,\quad 
\R(G_z)\cap\D(A_0)=\{0\}\,.
\ee
So the decomposition 
appearing in  $\D(A_{(\Pi,\Theta)})$ is unique. 
Moreover by first resolvent identity one obtains  
%(see \cite{[P01]}, Lemma 2.1)
\begin{equation}\label{1.2}
\forall\,w,z\in\rho(A_0)\,,\quad 
G_w-G_{z}=(z-w)\,R_w^0G_{z}\,.
\end{equation}
So  
\be
\R(G_w-G_z)\subseteq
\D(A_0)
\ee
and
\begin{equation}\label{1.2.1}
z\,G^*_0 G_z=z\,G^*_zG_0=z\,G^*_0(\uno-z\,R^0_z)G_0=\tau(G_0-G_z)\,.
\end{equation}
%Therefore the Kre\u\i n's type formula for the resolvent 
%s well defined. 
\end{remark}
\begin{remark}\label{aggiunto} Notice that the knowledge of 
the adjoint $S^*$ is not required. However it can be
readily calculated: by \cite{[P04]}, Theorem 3.1, one has 
\be
S^*:\D(S^*)\subseteq\H\to\H\,,\qquad S^*\phi=A_0\phi_0\,,
\ee
\be
\D(S^*)=\{\phi\in\H\,:\,\phi=\phi_0+G_0\xi_{\phi},\  \phi_0\in
\D(A_0),\ \xi_{\phi}\in\fh\}
\ee
and the Green-type formula
\begin{equation}\label{Green0}
\langle\phi,S^*\psi\rangle-\langle S^*\phi,\psi\rangle
=(\hat\tau\phi,\hat\rho\psi)
-(\hat\rho\phi,\hat\tau\psi)
\end{equation}
holds true. Here  the operators $\hat\tau$ and $\hat\rho$ 
are defined by
\be
\hat\tau:\D(S^*)\to\fh\,,\qquad\hat\tau\phi :=\tau 
\phi_0
\ee
and 
\be
\hat\rho:\D(S^*)\to\fh\,,\qquad\hat\rho\phi :=\xi_\phi\,.
\ee
Also notice that $G_z\xi$ solves the boundary value type 
problem
\begin{align}\label{Gz}
\begin{cases}
S^*G_z\xi=zG_z\xi\,,&\\
\hat\rho G_z\xi=\xi\,.&
\end{cases}
\end{align}
Hence, by Theorem \ref{estensioni}, 
\begin{equation}\label{DA1}
A_{(\Pi,\Theta)}=S^{*}|\{\phi\in\D(S^{*})\,:\,\hat\rho\phi\in\D(\Theta)\,,\ \Pi\hat\tau\phi=\Theta\hat\rho\phi\}\,.
\end{equation}
%\begin{equation}\label{DA2}
%\D(A_{(\Pi,\Theta)})=\{\phi\in\D(S^{*})\,:\,\hat\rho\phi\in\D(\Theta)\,,\ \Pi\hat\tau
%\phi=\Theta\hat\rho\phi\}\,.
%\end{equation}
 By \cite{[P03]}, Theorem 3.1, $(\fh,\hat\tau,\hat\rho)$ is a boundary
triple for $S^*$, with corresponding Weyl 
function 
\begin{equation}\label{WF}
M_z=z\,G^*_0G_z\equiv z\,G^*_zG_0\,.
\end{equation}
We refer to \cite{[DM]}, \cite{[MM]} and references therein for boundary triplets theory.
\end{remark}
%\begin{remark}\label{traslazione}
%Suppose that we start from the symmetric operator $S_\alpha:=S-\alpha$, $\alpha\ge 0$.
%Then, applying Theorem \ref{estensioni} to this case we obtain $\Ex(S_\alpha)=\{A^{(\alpha)}_{(\Pi,\Theta)}\,,\ (\Pi,\Theta)\in\E(\fh)\}$, where
% \begin{equation}
%A_{(\Pi,\Theta)}^{(\alpha)}:\D(A^{(\alpha)}_{(\Pi,\Theta)})\subseteq\H\to\H\,,\quad
%A_{(\Pi,\Theta)}^{(\alpha)}\phi:=(A_0-\alpha)\phi_\alpha\,,
%\end{equation}
%\begin{align}
%\D(A_{(\Pi,\Theta)}^{(\alpha)}):=\left\{\phi=\phi_\alpha+G_\alpha
%\xi_\phi\,,\ \phi_\alpha\in \D(A_0)\,,\, \xi_\phi\in
%\D(\Theta)\,,\,\Pi\tau
%\phi_\alpha=\Theta\xi_\phi
%\right\}.
%\end{align}
%Since, by \eqref{1.2}, $G_\alpha=G_0-\alpha R^0_0 G_\alpha$, one has
%$$
%\phi_\alpha+G_\alpha\xi=(\phi_\alpha-\alpha R^0_0 G_\alpha\xi)+G_0\xi\,,
%$$
%so that, posing $\phi_0=\phi_\alpha-\alpha R^0_0 G_\alpha\xi$ and $M_\alpha=\alpha G^*_0G_\alpha$
%$$
%\tau\phi_\alpha=
%\tau\phi_0+M_\alpha\xi\,.
%$$
%This shows that 
%$$
%\D(A_{(\Pi,\Theta)}^{(\alpha)})=\D(A_{(\Pi,\Theta_\alpha)})\,,\quad
%\Theta_\alpha:=\Theta-\Pi M_\alpha\Pi
%$$
%and
%\begin{align}
%A_{(\Pi,\Theta)}^{(\alpha)}\phi=&(A_0-\alpha)(\phi_0+\alpha R^0_0 G_\alpha\xi_\phi)\\
%=&
%A_{(\Pi,\Theta_\alpha)}\phi+\alpha A_0R^0_0G_\alpha\xi_\phi-\alpha\phi_\alpha\\
%=&(A_{(\Pi,\Theta_\alpha)}-\alpha)\phi\,.
%\end{align}
%\end{remark}
%\end{subsection}
\subsection{Sesquilinear Forms}
Here we determine the symmetric sesquilinear form $F_A$ of $A\in\Ex(S)$.  From now on we assume the following additional hypothesis on $S$: 
\begin{equation}\label{hyp}
\K(S^*)\cap\D(F_0)=\{0\}\,,
\end{equation}
where $F_0$ 
is the sesquilinear form associated  with  $-A_0$, i.e. is the
sesquilinear form of the Friedrichs extension of $-S$. By Remark \ref{aggiunto} one has $\K(S^*)=\R(G_0)$, and so, by \eqref{1.1},  
(\ref{hyp}) is equivalent, since $\D(A_0)\subseteq\D(F_0)$, to
\begin{equation}\label{hypp}
\R(G_0)\cap\D(F_0)=\{0\}\,.
\end{equation}
Notice that hypotheses (\ref{hyp}), or better its equivalent (\ref{hypp}), 
ensures that the decomposition appearing in 
$\D(F_{(\Pi,\Theta)})$ below is unique and so 
$F_{(\Pi,\Theta)}$ is well-defined. In the case $S$ is the minimal realization of a 2nd order elliptic differential operator on  a bounded domain, \eqref{hyp} always holds true (see Remark \ref{regularity} below). \par Next theorem is our version 
of Theorem 1.2 in \cite{[G70]} (also see Theorem 1 in \cite{[MMM92]}). There the more general 
case of (not necessarily self-adjoint) 
coercive extensions was considered. Our 
simpler framework allows for a straightforward proof, 
which we provide for reader's convenience.
\begin{theorem} \label{qf} 
Let
\be
F_0:\D(F_0)\times \D(F_0)\subseteq \H
\times \H \to\RE
\ee
be the positive sesquilinear form corresponding to $-A_0$ and suppose
that \eqref{hypp} holds true. Let
$A_{(\Pi,\Theta)}\in\Ex(S)$  and let 
\be
f_{\Pi,\Theta}:\D(f_{\Pi,\Theta})\times\D(f_{\Pi,\Theta})\subseteq\R(\Pi)\times\R(\Pi)\to\RE
\ee
be the sesquilinear form 
corresponding to $(\Pi,\Theta)\in\E(\fh)$. 
Then  
the sesquilinear form $F_{(\Pi,\Theta)}$ corresponding to 
$-A_{(\Pi,\Theta)}$ is given by 
\be
F_{(\Pi,\Theta)}:\D(F_{(\Pi,\Theta)})\times
\D(F_{(\Pi,\Theta)})\subseteq \H\times\H \to\RE\,,
\ee
\be
F_{(\Pi,\Theta)}(\phi,\psi)=F_0(\phi_0,\psi_0)+f_{\Pi,\Theta}(\xi_\phi,\xi_\psi)\,,
\ee
\be
\D(F_{(\Pi,\Theta)}):=\{\phi=\phi_0+G_0\xi_\phi
\,,\ \phi_0\in \D(F_0)\,,\ \xi_\phi\in\D(f_{\Pi,\Theta})\}\,.
\ee
\end{theorem}
\begin{proof}
Let 
\be
L:\D(L)\subseteq \H\to \H
\ee 
be the linear operator associated  with  $F_{(\Pi,\Theta)}$, 
i.e. 
\be
\D(L):=\{\phi\in\D(F_{(\Pi,\Theta)}): 
\exists\varphi\in\H\ \text{\rm s.t.}\
\forall \psi\in\D(F_{(\Pi,\Theta)})\,,\ F_{(\Pi,\Theta)}
(\phi,\psi)=\langle\varphi,\psi\rangle \}\,,
\ee
\be
L\phi:=\varphi\,.
\ee
Since $\D(F_0)\subseteq\D(F_{(\Pi,\Theta)})$, 
$\D(F_{(\Pi,\Theta)})$ is dense and so $L$ is well-defined.\par
By the definition of $\D(F_{(\Pi,\Theta)})$ 
and by taking, in the definition of $\D(L)$, 
at first $\xi_\psi=0$ and then $\psi=G_0\xi$, 
one gets that $\phi=\phi_0+G_0\xi_\phi\in\D(L)$ 
if and only if there exists $\varphi$ such that
\be
\forall\psi_0\in\D(F_0)\,,\quad 
F_0(\phi_0,\psi_0)=
\langle\varphi,\psi_0\rangle
\ee
and
\be
\forall\xi\in\D(f_{\Pi,\Theta})\,,\quad f_{\Pi,\Theta}(\xi_\phi,\xi)
=\langle\varphi,G_0\xi\rangle\,.
\ee
Thus $\phi_0\in\D(A_0)$, $L\phi=-A_0
\phi_0$, and 
\be
\langle\varphi,G_0\xi\rangle
=\langle -A_0\phi_0,G_0\xi\rangle=
(\tau \phi_0,\xi)=(\Pi\tau\phi_0,\xi)\,.
\ee
Therefore
\be
\forall\,\xi\in\D(f_{\Pi,\Theta})\,,\quad f_{\Pi,\Theta}(\xi_\phi,\xi)=
(\Pi\tau\phi_0,\xi)\,,
\ee
i.e.
\be
\xi_\phi\in D(\Theta)\,,\quad \Theta\xi_\phi=\Pi\tau\phi_0\,.
\ee
In conclusion $\D(L)=\D(A_{(\Pi,\Theta)})$ and
\be
L\phi=-A_0\phi_0=
-A_{(\Pi,\Theta)}\phi\,.
\ee
\end{proof}
\begin{remark}\label{traslazioneforma}
By \eqref{1.2} and \eqref{WF}, by $(F_0+\lambda)(R^0_\lambda\phi,\psi)=\langle \phi,\psi\rangle$, $\lambda\ge 0$, and by Theorem \ref{qf}, one has, for all $\alpha\ge 0$,
\begin{align}
&\left(F_{(\Pi,\Theta)}+\alpha\right)(G_\alpha\xi_1,G_\alpha\xi_2)\\
=&
F_0(G_\alpha\xi_1-G_0\xi_1, G_\alpha\xi_2-G_0\xi_2)+
f_{\Pi,\Theta}(\xi_1,\xi_2)+\alpha(G_\alpha\xi_1,G_\alpha\xi_2)\\
=&-\alpha F_0(R^0_0G_\alpha\xi_1,G_\alpha\xi_2-G_0\xi_2)+
f_{\Pi,\Theta}(\xi_1,\xi_2)+\alpha(G_\alpha\xi_1,G_\alpha\xi_2)\\
=&
M_\alpha(\xi_1,\xi_2)+f_{\Pi,\Theta}(\xi_1,\xi_2)\,.
\end{align}
We make use of this relation in the proof of Theorem \ref{converse} below. 
%denoting by $F^{(\alpha)}_{(\Pi,\Theta)}$ the quadratic form associated with $-A^{(\alpha)}_{(\Pi,\Theta)}$, i.e. 
%$$
%F^{(\alpha)}_{(\Pi,\Theta)}(\phi,\psi)=(F_0+\alpha)(\phi_\alpha,\psi_\alpha)+
%f_{\Pi,\Theta}(\xi_\phi,\xi_\psi)\,,
%$$
%$$
%\D(F^{(\alpha)}_{(\Pi,\Theta)})=\{\phi=\psi_\alpha+G_\alpha\xi_\psi\,\ \psi_\alpha\in\D(F_0)\,,\ \xi_\psi\in\D(f_{\Pi,\Theta})\}\,,
%$$
%one has
%$$
%F^{(\alpha)}_{(\Pi,\Theta)}=F_{(\Pi,\Theta_\alpha)}+\alpha\,.
%$$
%In particular one obtains
%$$
%f_{\Pi,\Theta}(\xi_1,\xi_2)=F^{(\alpha)}_{(\Pi,\Theta)}(G_\alpha\xi_1,G_\alpha\xi_2)=
%\left(F_{(\Pi,\Theta_\alpha)}+\alpha\right)(G_\alpha\xi_1,G_\alpha\xi_2)\,,
%$$
%so that
%\begin{equation}\label{trasforma}
%f_{\Pi,\Theta}(\xi_1,\xi_2)+(M_\alpha\xi_1,\xi_2)=\left(F_{(\Pi,\Theta)}+\alpha\right)(G_\alpha\xi_1,G_\alpha\xi_2)\,.
%\end{equation}
\end{remark}
\begin{remark}\label{core}
Suppose that $f_{\Pi,\Theta}$ is lower bounded, so that 
$F_{(\Pi,\Theta)}$ is lower bounded. Then it is immediate to check that if
$\C(f_{\Pi,\Theta})$ is a core of $f_{\Pi,\Theta}$ and $\C(F_0)$ is a
core of $F_0$ then 
\be
\C(F_{(\Pi,\Theta)}):=\{\phi\in \H\,:\, \phi=\phi_0+G_0
\xi_\phi\,,\ 
\phi_0\in \C(F_0)\,,\ \xi_\phi\in\C(f_{\Pi,\Theta})\}
\ee
is a core of $F_{(\Pi,\Theta)}$\,.  
\end{remark}
%\end{subsection}
\subsection{Maximal and minimal extensions}
Let us now define 
\be
\ExP(S):=\{A\in\Ex (S)\,:\, -A\ge 0\}\,,
\ee
By Theorem \ref{qf} one immediately gets the following well known result going back to Birman (see \cite{[B]}):
\be
\ExP(S)=\{A_{(\Pi,\Theta)}\,:\, \Theta\ge 0\}\,,
\ee
Now we define, on the set $
\ExP(S)$, the semi-order $\preceq$ by (see e.g. \cite{[Fu]}, Section 3.3) 
\be
A_1\preceq A_2\,\iff\,\begin{cases} &\D(F_{A_1})\subseteq\D(F_{A_2})\,,\\  &\forall\,\phi\in\D(F_{A_1})\,,\ F_{A_1}(\phi)\ge F_{A_2}(\phi)\,.
\end{cases}
\ee
By Theorem \ref{qf}, given  
$A_{(\Pi_1,\Theta_1)}$ and $A_{(\Pi_2,\Theta_2)}$ in $\ExP(S)$, one has 
\be
A_{(\Pi_1,\Theta_1)}\preceq A_{(\Pi_2,\Theta_2)}\,\iff\,(\Pi_1,\Theta_1)\preceq (\Pi_2,\Theta_2)
\ee
i.e.
\be
A_{(\Pi_1,\Theta_1)}\preceq A_{(\Pi_2,\Theta_2)}\,\iff\,\begin{cases} 
& \D(f_{\Pi_1,\Theta_1})\subseteq \D(f_{\Pi_2,\Theta_2})\,,\\  
&\forall\xi\in\D(f_{\Pi_1,\Theta_1})\,,\ f_{\Pi_1,\Theta_1}(\xi)\ge
f_{\Pi_2,\Theta_2}(\xi)\,.
\end{cases}
\ee
Thus the Friedrichs extension 
$A_0$ (corresponding to $\Pi=0$) is the minimal element of 
$\ExP(S)$ and $A_K:=A_{(\uno,0)}$ is the maximal one. 
The extension $A_K$, discovered by von Neuman 
in \cite{[vN]},  is named {\it Kre\u\i n's
  extension},  after Kre\u\i n's seminal paper \cite{[k]} characterizing the extremal elements of $\ExP(S)$.
% containing 
%the first proof about determination of the maximum element .
%in that $A_K$ is   .
%\end{subsection}
\vskip 10pt
\noindent
{\bf Warning.} From now on all the Hilbert spaces we consider 
are {\it  real} Hilbert spaces.
\vskip 10pt
\subsection{Dirichlet Forms}
Let us now suppose that $\H=L^2(X,{\mathscr B},m)$ ($\equiv L^2(X)$ for short), where $X$ is a locally compact separable metric space and 
$m$ is a positive Radon measure on the Borel $\sigma$-algebra ${\mathscr B}$ of $X$ such that supp$(m)=X$, i.e. $m$ is finite on compact sets and is strictly positive on non-empty open sets. \par 
A negative self-adjoint operator 
$A$ is said to be {\it Markovian} if  
\be
0\le u\le 1\,,\ m\text{-a.e.}\quad\Longrightarrow\quad 
\forall t> 0\,,\quad 0\le e^{tA}u\le 1\,,\ m\text{-a.e.}\,.
\ee
By
\begin{equation}\label{Res-SG}
e^{tA}u=\lim_{\lambda\uparrow\infty}e^{-\lambda t}\sum_{n=0}^{\infty}\frac{(\lambda t)^{n}}{n!}\,(\lambda R^{A}_{\lambda})^{n}u\,,\quad R^{A}_{\lambda}u=\int_{0}^{\infty}e^{-\lambda t}e^{tA}u\,dt\,,
\end{equation}
this is equivalent to 
\be
0\le u\le 1\,,\ m\text{-a.e.}\quad\Longrightarrow\quad 
\forall \lambda>0\,,\quad 0\le\lambda R_\lambda^A u\le 1\,,\ m\text{-a.e.}\,.
\ee
By \cite{[Dv]}, Theorem 1.4.1, if $A$ is Markovian then $L^{1}(X)\cap L^{\infty}(X)$ is invariant under $e^{tA}$ and $e^{tA}|L^{1}(X)\cap L^{2}(X)$ can be extended  to a strongly continuous 
contractive semi-group on $L^{p}(X)$ for all $p\in [1,\infty)$.\par 
A function $\Phi:\RE\to\RE$ is said to be a {\it normal contraction} if 
\be
\Phi(0)=0\quad\text{and}\quad |\Phi(t)-\Phi(s)|\le |t-s|\,.
\ee
We will be mainly concerned with the normal contraction given by the {\it unit contraction} 
\be
\Phi(u)=u_{\#}:=(0\vee u)\wedge 1\,.
\ee
A positive symmetric bilinear form 
\be
F:\D(F)\times\D(F)\subseteq L^2(X)\times L^2(X)\to \RE 
\ee
is 
said to be a {\it Markovian form} 
if 
%for any normal contraction $\Phi$ one has 
\begin{equation}\label{NC}
u\in\D(F)\,\Longrightarrow\,
\begin{cases} &u_{\#}\in\D(F)\,,\\  
&F(u_{\#})\le F(u)\,.
\end{cases}
\end{equation}
%where
%\be
%[\Phi(u)](x):=\Phi(u(x))\,.
%\ee
%for each $\epsilon>0$ there exists a 
%function $\phi_\epsilon:\RE\to[-\epsilon, 1+\epsilon]$, such that
%$\begin{equation}\label{NC}
%\phi_\epsilon(t)=t\,,\ 0\le t\le 1\,,\quad 0\le \phi_\epsilon(t)-
%\phi_\epsilon(s)\le t-s\,,\ s\le t\,,
%\end{equation}
%and
%$$
%u\in\D(F)\,\Longrightarrow\,
%\begin{cases} &\phi_\epsilon(u)\in\D(F)\,,\\  
%&F(\phi_\epsilon(u))\le F(u)\,.
%\end{cases}
%$$
A closed Markovian form $F$ is said to be a {\it Dirichlet form}  and the Hilbert space $\H(F)$ given by the set $\D(F)$ equipped with the 
scalar product 
\be
\langle u,v\rangle_F:=F(u,v)+\langle u,v\rangle_{L^2(X)}
\ee
is the corresponding {\it Dirichlet space}. 
\begin{remark}\label{coincide} Notice that our definition of Markovian form is stronger than the usual one (as given e.g. in \cite{[Fu]},  Section 1.1). By \cite{[Fu]}, Theorem 1.4.1, the two definitions coincide whenever the form is closed. Moreover, if $F$ is a Dirichlet form then \eqref{NC} is equivalent to 
\begin{equation}
\text{$\forall\Phi$ normal contraction,} \qquad
u\in\D(F)\,\Longrightarrow\,
\begin{cases} &\Phi(u)\in\D(F)\,,\\  
&F(\Phi(u))\le F(u)\,.
\end{cases}
\end{equation}
\end{remark}
\vskip 10pt
\noindent
{\bf Warning.} Notice that, according to our definitions, Markovian and Dirichlet forms are not necessarily 
densely defined in $L^2(X)$. In particular, forms densely defined in $\R(\Pi)\not=L^{2}(X)$ of the kind $f_{\Pi,\Theta}$, $(\Pi,\Theta)\in\E(L^{2}(X))$, will be often regarded as forms on $L^2(X)$ 
with a not dense domain.
\vskip 10pt
\noindent
%If in the above definitions one does not suppose that forms are densely defined, 
%then one gets the definition of  a {\it Dirichlet form in wide sense}. 
%Many of the analytical results  about Dirichlet forms (e.g. next theorems \ref{closure}, \ref{canc}, \ref{Res-Form}) extend to the case of not dense domain (see e.g. \cite{[Ou]}, Section 2.6).
By \cite{[Fu]}, Theorem 2.1.1., one has
\begin{theorem}\label{closure} The closure of a closable Markovian form is a 
Dirichlet form.
\end{theorem}\noindent
%Moreover, by \cite{[Fu]}, Theorem 1.4.1, for closed forms the special normal contraction $\Phi(u)=u_\#$ suffices in \eqref{NC}:
%\begin{theorem}\label{canc} A closed, symmetric bilinear form $F$ is a Dirichlet form if and only if
%\be
%u\in\D(F)\,\Longrightarrow\,
%\begin{cases} &u_\#\in\D(F)\,,\\  
%&F(u_\#)\le F(u)\,.
%\end{cases}
%\ee
%\end{theorem} \vskip 10pt
\noindent 
By Theorem 10 in \cite{[An]} one has
\begin{theorem}\label{ancona}
Let $F$ be a Dirichlet form and  let $\Phi$ be any normal contraction. Then 
the map $u\mapsto\Phi(u) $ 
is continuous on the Dirichlet space 
%$\left(\D(F),\langle\cdot,\cdot\rangle_F\right)$ 
$\H(F)$ to itself.
\end{theorem}
\vskip 10pt
\noindent 
The connection between  Markovian operators and Dirichlet forms is  given by the {\it Beurling-Deny criterion} (see e.g. \cite{[Fu]}, Theorem 1.4.1, \cite{[Dv]}, Theorem 1.3.3):
\begin{theorem}\label{dirichlet}
$A$ is Markovian $\iff$ $F_A$ is a densely defined Dirichlet form.
\end{theorem}
\vskip 10pt
\noindent
Thus, posing
\be
\ExM(S):=\{A\in\Ex(S)\,:\, \text{\rm $A$ is Markovian}\}\,,
\ee
one has, by Theorems \ref{estensioni} and \ref{qf},
\begin{equation}\label{ExM}
\ExM(S)=\{\text{$A_{(\Pi,\Theta)}\,:\,F_{(\Pi,\Theta)}$ is a Dirichlet form}\}\,.
\end{equation}
%\end{subsection}
\subsection{Yosida approximations}
Let $F_A$ be the bilinear form associated  with  the positive self-adjoint 
operator $-A$ 
%Then (see e.g. \cite{[Fu]}, equation (1.3.7))
%$$
%\forall u\in L^2(X)\,,\,\forall v\in \D(F_A)\,,\qquad 
%(F_A+\lambda)(R^A_\lambda u,v)=\langle u,v\rangle_{L^2(X)}\,.
%$$
and let us consider the bounded bilinear symmetric form $F_A^\lambda$ associated with the Yosida approximation of $A$, i.e.
\be
F_A^\lambda:L^2(X)\times L^2(X)\to\RE\,,\quad
%\ee
%\be
F_A^\lambda(u,v):=\langle u,\lambda(\uno  -\lambda R^A_\lambda) v\rangle_{L^2(X)}\,.
\ee
Then,
by \cite{[Fu]}, Lemma 1.3.4, formulae (1.4.7)-(1.4.9) and Theorem 1.4.2, one has
the following
\begin{theorem}\label{Res-Form} For any $u\in L^2(X)$, the function 
$\lambda\mapsto F_A^\lambda(u)$ is non-decreasing,
\be
\D(F_A)=\{u\in L^2(X)\,:\, \lim_{\lambda\uparrow\infty}\,F_A^\lambda(u)<\infty\}\,,
\ee
\be
F_A(u,v)=\lim_{\lambda\uparrow\infty}\,F_A^\lambda(u,v)\,.
\ee
Moreover, if $A$ is Markovian, i.e. if $F_A$ is a Dirichlet form, then $F^\lambda_A$ is a Dirichlet form and 
\be
F^\lambda_A=\breve F^\lambda_A+\check F^\lambda_A\,,
\ee
where the bounded Dirichlet forms $\breve F^\lambda_A$ and $\check F^\lambda_A$ are defined by  
\be
\breve F_A^\lambda(u,v):=\frac{\lambda}{2}\int_{X\times X}(u(x)-u(y))(v(x)-v(y))\,d\sigma_A^\lambda(x,y)
\ee
\be
\check F_A^\lambda(u,v):=\lambda\int_Xu(x)v(x)s_A^\lambda(x)\,dm(x)\,.
\ee
Here $0\le s_A^\lambda\le 1$ and $\sigma_A^\lambda$ is a positive, symmetric, Radon measure  on $X\times X$. 
\end{theorem}
%Notice that, in the case $m(X)<\infty$,   by Lemma \ref{Res-Form} one has, for all $u\in L^\infty(X,m)$,
%$$
%F^\lambda_A(1,u^2)=
%\lambda\int_Xu^2(x)(1-s^A_\lambda(x))\,dm(x)\,.
%$$
\begin{remark}\label{DM0}
If $F_A$ is a Dirichlet form, by Theorem 1.4.2 in \cite{[Fu]},   
\be
u\in L^\infty(X,m)\cap\D(F_A)\quad\Longrightarrow\quad u^2\in\D(F_A)
\ee
and 
\be
u\in\D(F_A)\quad\Longrightarrow\quad \begin{cases}&u_{n}\in\D(F_A)\\ 
&F_A(u)=\underset{n\uparrow\infty}\lim\,F_A(u_{n})\,,\end{cases}
\ee
where $u_{n}:=((-n)\vee u)\wedge n$. Thus if $m(X)<\infty$ then 
\begin{equation}\label{DM1}
F_A(u)=\lim_{n\uparrow\infty}\lim_{\lambda\uparrow\infty }\left(
\breve F^{\lambda}_{A}(u_{n})
+F^\lambda_A(1,u_{n}^2)\right)
\end{equation}
and, if moreover $1\in\D(F_A)$,
\begin{equation}\label{DM2}
F_A(u)=\lim_{n\uparrow\infty}\left(\lim_{\lambda\uparrow\infty }\breve F^{\lambda}_{A}(u_{n})+F_A(1,u_{n}^2)\right)\,.
\end{equation}
 \end{remark}
%\end{subsection}
\subsection{Capacity and quasi continuity}
Let $F$ be a densely defined Dirichlet form. For any open set 
${\mathcal O}\subseteq X$ we define its {\it $F$-capacity} by (see \cite{[Fu]}, section 3.1) 
\be
\text{cap}_{F}({\mathcal O}):=\inf\{\langle u,u\rangle_{F}\,,\ u\in\D(F)\,,\ u\ge 1\ \text{$m$-a.e. on ${\mathcal O}$}\}\,.
\ee
Here, as usual, one poses $\inf(\emptyset)=+\infty$. For an arbitrary set ${\mathcal B}\subseteq X$ one then defines
\be
\text{cap}_{F}({\mathcal B}):=\inf\{\text{cap}_{F}({\mathcal O})\,,\ {\mathcal O}\ \text{open},\ {\mathcal O}\supset {\mathcal B}\}\,.
\ee
Such definitions provide a Choquet capacity (see \cite{[Fu]}, Theorem 3.1.1).\par
A statement is said to hold {\it quasi everywhere} (q.e. for short) if there exists a set ${\mathcal N}$ of zero capacity such that the statement is true outside ${\mathcal N}$. Notice that 
\be
\text{cap}_{F}({\mathcal N})=0\quad\Rightarrow\quad 
m({\mathcal N})=0\,.
\ee
Also notice that if $X=\RE^n$ and $F(u,v)=\langle\nabla u,
\nabla v\rangle_{L^2(\RE^{n})}$ with 
\be
\D(F)=H^1(\RE^n)=\{u\in L^{2}(\RE^{n})\,:\,\|\nabla u\|\in L^{2}(\RE^{n})\}\,,
\ee
then $\text{cap}_{F}$ is the usual Newtonian capacity.\par
Given an extended real valued function $u$ on $X$, we call it {\it quasi continuous} if  
for any $\epsilon > 0$ there exists an open set ${\mathcal O}_{\epsilon}\subset X$ such that $\text{cap}_{F}({\mathcal O}_{\epsilon}) 
< \epsilon$ and the restriction of $u$ to $X\backslash {\mathcal O}_{\epsilon}$ is finite and continuous.\par
Given a function $u$, $\t u$ is said to be a 
{\it quasi continuous modification} of $u$  if $\t u$ is quasi continuous and $\t u=u$ $m$-a.e.. By \cite{[Fu]}, Theorem 3.1.3, one has 
\begin{theorem} 
Any $u\in\D(F)$ has a quasi continuos modification which is unique up to a set of zero $F$-capacity.
\end{theorem}
%\end{subsection}

 \subsection{The Beurling-Deny Decomposition} 
 A densely defined Dirichlet form 
$F$ is said to be {\it regular on $X$} if $\D(F)\cap C_c(X)$ is both $\H(F)$-dense in $\D(F)$ and  $L^{\infty}(X)$-dense in $C_c(X)$ (here $C_c(X)$ denotes the set of continuos function with compact support). For regular Dirichlet forms Beurling-Deny decomposition holds (see \cite{[BD]}, \cite{[Fu]}, Theorem 2.2.1, \cite{[CF]}, Theorem 4.3.3):
\begin{theorem} Any regular Dirichlet form $F$ on $L^2(X)$ 
admits, for any $u,v\in \D(F)$ the decomposition
\be
F(u,v)=F^{(c)}(u,v)+F^{(j)}(u,v)+F^{(k)}(u,v)\,.
\ee
Here $F^{(c)}$ is a Markovian form which satisfies the strongly local property, i.e. $F^{(c)}(u,v)=0$ whenever $u$ has a compact support and $v$ is constant on a neighborhood of the support of $u$; 
\be
F^{(j)}(u,v)=\int_{X\times X}(\t u(x)-\t u(y))(\t v(x)-\t v(y))\,dJ(x,y)\,,
\ee
\be
F^{(k)}(u,v)=\int_X\t u(x)\t v(x)\,d\kappa(x)\,,
\ee
where $\t u$ and $\t v$ denote quasi continuos versions of $u$ and $v$, $J$ is a symmetric positive Radon measure on $X\times X$ off the diagonal and $\kappa$ is a positive Radon measure on $X$.
\end{theorem} 
\vskip 5pt
\noindent
In the case $X$ has a differential structure, i.e. $X=\Omega\subseteq \RE^n$ more can be said about $F^{(c)}$ (see \cite{[Silv]}, Theorem 16.1, \cite{[FOT]}, Theorem  3.2.3):
\begin{theorem} Let $F$ a regular Dirichlet form on $L^{2}(\Omega)$. 
Then for any $u,v \in C_c^\infty(\Omega)\cap\D(F)$ one has
\be
F^{(c)}(u,v)=\sum_{1\le i,j\le n}\int_\Omega \frac{\partial u}{\partial x_{i}}\,\frac{\partial v}{\partial x_{j}}\,
d\nu_{ij}\,,
\ee
where the $\nu_{ij}$'s are positive Radon measures on $\Omega$ such that, for any $\xi\in\RE^n$ and for any compact $K\subset\Omega$,
\be
\sum_{1\le i,j\le n}\nu_{ij}(K)\xi_i\xi_j\ge 0\,,\qquad
\nu_{ij}(K)=\nu_{ji}(K)\,.
\ee
\end{theorem}
%\end{subsection}
\subsection{Logarithmic Sobolev inequalities, ultracontractivity and heat kernel estimates}
Let $F_{A}$ be a densely defined Dirichlet Form on $L^{2}(X)$. Let us denote by $\kappa_{A}(t,\cdot,\cdot)$ the integral kernel of  $e^{tA}$. Here we briefly recall the connections between certain functional inequalities 
involving $F_{A}$, bondedness of $e^{tA}$ from $L^{2}(X)$ to $L^{\infty}(X)\cap L^{2}(X)$ and estimates on $\kappa_{A}$ (see e.g. \cite{[Dv]}, 
\cite{[Gr]}, \cite{[Wang]} for proofs, more details and further results) . 
\par
%In this subsection for notational convenience we pose $\|\cdot\|_{p}\equiv\|\cdot\|_{{L^{p}(X)}}$.\par
We say that $F_{A}$ satisfies a {\it logarithmic Sobolev inequality} (with function $\beta$) if there exist a continuous monotone decreasing function $\beta$ such that for all $\epsilon>0$ and for all positive $u\in\D(F_{A})\cap L^{1}(X)\cap L^{\infty}(X)$ there holds
\begin{equation}
%\label{LSI}
\int_{X}\!\!\! u^{2}\log u\,dm\le \epsilon F_{A}(u)+\beta(\epsilon)\|u\|^{2}
_{L^2(X)}
+\|u\|^{2}_{L^2(X)}\log\|u\|_{L^2(X)}.
\end{equation}
If $F_{A}$ satisfies a logarithmic Sobolev inequality with function $\beta$ such that $m(t):=\frac1t\smallint_{0}^{t}\beta(\epsilon)\,d\epsilon$ is finite for any $t>0$ then (see \cite{[Dv]}, Corollary 2.2.8) $e^{tA}$ is {\it ultracontractive} (with function $m$), i.e. 
\be
\forall t>0\,,\ \forall u\in L^{2}(X)\,,\quad \|e^{tA}u\|_{L^{\infty}(X)}\le e^{m(t)}\|u\|_{L^2(X)}\,.
\ee
Conversely (see \cite{[Dv]}, Theorem 2.2.3) if $e^{tA}$ is ultracontractive with a continuous monotone decreasing 
function $m$ then $F_{A}$ satisfies a logarithmic Sobolev inequality with function $\beta=m$.
Moreover (see \cite{[Dv]}, Lemma  2.1.2) ultracontractivity with function $m$ implies the heat kernel estimate
\be
\forall t>0\,,\ \text{for $m$-a.e. $x$ and $y$,} \quad \kappa_{A}(t,x,y)\le e^{2m(t/2)}\,.
\ee
Conversely the estimate $\kappa_{A}(t,x,y)\le e^{m(t)}$ implies ultracontractivity (and hence a logarithmic Sobolev inequality) with function 
$m/2$. \par 
If $m(X)<\infty$ and $e^{tA}$ is ultracontractive then (see \cite{[Dv]}, Theorem 2.1.4)
\begin{equation}\label{trace} 
\forall t>0\,,\quad \text{trace}(e^{tA})<+\infty
\end{equation} 
and  (see \cite{[Dv]}, Theorem 2.1.5) $e^{tA}$ is compact on $L^{p}(X)$ for any 
$p\in [1,\infty]$. Moreover any eigenfunction $v_{n}$ of $A$ is in $L^{\infty}(X)$ and
\be
\kappa_{A}(t,x,y)=\sum_{n=1}^{\infty}e^{\lambda_{n}t}v_{n}(x)v_{n}(y)\,,
\ee
where $\lambda_{n}$ is the eigenvalue corresponding to $v_{n}$ and the series converges uniformly on $[t_{\circ},\infty)\times X\times X$ for any $t_{\circ}>0$.
%\end{subsection}
\subsection{Dirichlet Forms and Hunt Processes} Here we briefly recall the one-to-one 
correspondence between regular Dirichlet forms and Markov processes. We refer to 
\cite{[Fu70]}, \cite{[Fu]}, \cite{[FOT]} and \cite{[CF]} for more details and proofs.\par
 Let the Dirichlet form $F_{A}$ be regular on $X$  and let $e^{tA}$, $t\ge 0$, be the semi-group on $L^{2}(X)$ generated by the 
corresponding Markovian operator $A$. Then there exists an (unique in law) $X_{\partial}$-valued, {\it Hunt Markov process} $\Z_{A}=(\{\Z_{t}\}_{t\ge 0}, \{{\mathsf P}_{\! x}\}_{x\in X})$ such that 
\begin{equation}\label{Hunt}
\forall u\in L^{2}(X)\,,\quad e^{tA}u(x)={\mathsf E}_{x}(u(\Z_{t}))\,,\ \text{for q.e. $x\in X$},
\end{equation}
where ${\mathsf E}_{x}$ denotes expectation with respect to the probability measure ${\mathsf P}_{\! x}$. 
Here $X_{\partial}:=X\cup\{\partial\}$ ($\partial$ is the 
``cemetery''), $\Z_{t}(\omega)=\partial$ for every $t\ge
\zeta(\omega)$, where the lifetime $\zeta$ is defined by 
$\zeta(\omega):=\inf\{t\ge 0: \Z_{t}(\omega)=\partial\}$, and $u(\partial):=0$. \par
The trajectories of  $\Z_{A}$ are almost surely right-continuos with
left limits.  
If $F_A$ has the {\it local property}, i.e. if $F_{A}(u,v)=0$ for any
couple $u,v\in\D(F_{A})$ with compact disjoint supports, 
then  $\Z_{A}$ is a {\it Diffusion}, i.e. its trajectories are almost surely continuous. Notice that 
by the Beurling-Deny decomposition $\Z_{A}$ is a diffusion if and only if $F_{A}^{(j)}=0$.
 %\end{subsection}
\subsection{Dirichlet Forms and paths behavior} 
Some analytical
  properties of the Dirichlet form $F_A$ and the corresponding
  resolvents $R^A_\lambda$ and semigroups $e^{tA}$ 
correlate with paths
  behavior of $\Z_A$. Here we recall the main results following 
\cite{[FOT]} and \cite{[CF]} to which we refer for more 
details and proofs.
For simplicity from now on in this subsection 
we suppose that $m(X)<\infty$.\par 
The Markovian operator $A$  is said to be {\it conservative} if
$\lambda R^{A}_{\lambda}1=1$ $m$-a.e. for all
$\lambda>0$.  By \eqref{Res-SG} this is
equivalent to $e^{tA}1=1$ $m$-a.e. for all $t>0$. 
\par
By \eqref{Hunt} and \cite{[FOT]}, exercise 4.5.1, 
one has
\begin{equation}\label{conserviff}
{A}\ \text{is conservative}\ \iff\ 
\forall x\in X\,,\quad {\mathsf P}_{\! x}(\zeta=+\infty)=1\,.
\end{equation}
\par
Denoting by $\t R_\lambda^A$, $\lambda>0$, the extension 
to $L^1(X)$ of the resolvent and posing, 
for $m$-a.e. $x\in X$, and for all positive $u\in L^1(X)$,  
\be
\t R^Au(x):=\lim_{\lambda\downarrow 0}\t R^A_{\lambda}u(x)\,,
\ee
$A$ is said to be {\it transient} if 
\be
\forall u\in L^1(X)\,,\ u\ge 0\ m\text{-a.e.}\,,\quad 
m(x\in X\,:\,\t R^Au(x)=+\infty)=0\,.
\ee
and is said to be {\it recurrent} if 
\be
\forall u\in L^1(X)\,,\ u\ge 0\ m\text{-a.e.}\,,\quad 
m(x\in X\,:\,0<\t R^Au(x)<+\infty)=0\,.
\ee
By Theorem \ref{Res-Form}, if $A$ is conservative then $1\in\D(F_A)$ and $F_A(1)=0$. By \cite{[FOT]}, Theorem 1.6.3, this implies that $A$ is recurrent. Since, by \cite{[CF]}, Theorem 2.1.10, recurrence implies conservativeness, in conclusion we get (here the hypothesis $m(X)<\infty$ is essential) 
\begin{equation}\label{conserviff2}
A\ \text{conservative}\iff A\ \text{recurrent}\iff 
\text{$1\in\D(F_A)$ and $F_A(1)=0$.}
\end{equation}
%By \cite{[FOT]}, Theorem 1.5.1, ${A}$ is transient if and only if
%there exists  a $m$-a.e. strictly positive $g\in L^{\infty}(X)$ such that 
%\begin{equation}\label{transiff}
%\forall u\in\D(F_A)\,,\ \int_{X}|u(x)|g(x)\,dm(x)\le F_A(u)\,.
%\end{equation}
%Therefore by \eqref{conserviff2}, by \eqref{transiff} and by the definition of the semi-order $\prec$, 
%one has  that  if $F_{1}\preceqcF_{2}$ for some $c>0$ then (see \cite{[FOT]}, Theorem 1.6.4)
 %$$
 %F_{2}\ \text{transient}\quad\Longrightarrow\quad F_{1}\ \text{transient}
 %$$
 %and
 %$$
 %F_{1}\ \text{recurrent}\quad\Longrightarrow\quad F_{2}\ \text{recurrent.}
 %$$
 By \cite{[CF]}, Theorems 2.1.5 and 3.5.2, if $A$ is 
transient then 
\begin{equation}
{\mathsf P}_{\! x}(\zeta=+\infty\,,\ \lim_{t\uparrow\infty}\Z_{t}=\partial)=
{\mathsf P}_{\! x}(\zeta=+\infty)\,,\quad \text{for q.e. $x\in X$}\,.
\end{equation}
%In particular, when $A$ is both conservative and transient, 
%the paths wanders out to infinity, i.e.
%\begin{equation}
%P_{x}(\lim_{t\uparrow\infty}Z_{t}=\partial)=1
%\,,\quad \text{for q.e. $x\in X$}\,.
%\end{equation}
A Markov operator $A$ is said to be {\it irreducible} if 
\be
\forall u\in L^2(X)\,,\forall t>0\,,\ \, e^{tA}(1_{\mathcal B}u)=1_{\mathcal B}e^{tA}u
\ \Longrightarrow\ m(\mathcal B)m(\mathcal B^{c})=0\,.
\ee
If ${A}$ is irreducible and if the bottom of its spectrum is an eigenvalue, then such an eigenvalue is simple and the corresponding eigenfunction is strictly positive $m$-a.e.  (see e.g. \cite{[Dv]}, Proposition 1.4.3).\par
By \cite{[FOT]}, Lemma 1.6.4, 
\begin{equation}\label{irreduc1}
 {A}\ \text{irreducible}\quad\Longrightarrow\quad {A}\ \text{either recurrent or transient}.
 \end{equation}
 By \cite{[CF]}, Theorem 2.1.11, if $A$ is recurrent then 
 \begin{equation}\label{irreduc2}
 {A}\ \text{irreducible}\ \iff\ \text{$u$ is $m$-a.e. 
constant whenever $F_{A}(u)=0$  
 .}
 \end{equation}
 By \cite{[CF]}, Theorem 3.5.6, if ${A}$ is irreducible then
 \be
 {\mathsf P}_{\! x}(\sigma_{\mathcal B}<+\infty)>0\,,\quad\text{for q.e. $x\in X$}
 \ee
 and if ${A}$ is irreducible and recurrent then
 \be
 {\mathsf P}_{\! x}(\sigma_{\mathcal B}\circ\theta_{n}<+\infty \ \text{for every $n\ge 0$})=1\,,\quad\text{for q.e. $x\in X$}\,.
 \ee
 Here $\mathcal B$ is any Borel set with strictly positive $F_{A}$-capacity, $\sigma_{\mathcal B}$ denotes the first hitting time of $\mathcal B$, i.e. 
 $\sigma_{\mathcal B}:=\inf\{t>0 : \Z_{t}\in {\mathcal B}\}$,
 and $\theta_{n}$ is the time shift $\Z_{t}\circ\theta_{s}=\Z_{t+s}$.
 %\end{subsection}
%\subsection{Hunt processes}
%The Dirichlet form $F$ is said to be regular if $\D(F)\cap 
%C_c(X)$ is a core for $F$ and it is $L^\infty(\Omega)$-dense 
%in $C_c(\Omega)$.
%%\end{subsection}
%\end{section}

\section{Markovian extensions - a toy example}
Let $S=A_\min$ be the negative, symmetric linear operator given by the minimal realization of 
$\frac{d^2}{dx^2}$ on the finite interval $(0,\ell)$:
\be
A_\min:C^\infty_c(0,\ell)\subseteq L^2(0,\ell)\to
L^2(0,\ell)\,,\qquad A_\min u=\frac{d^2u}{dx^2}\,.
\ee
Its Friedrichs' extension is given by  
\be
A_D:H^2(0,\ell)\cap H^1_0(0,\ell)\subseteq L^2(0,\ell)\to
L^2(0,\ell)\,,\quad A_D u=\frac{d^2u}{dx^2}\,,
\ee
with corresponding bilinear  form 
\be
F_D:H^1_0(0,\ell)\times H^1_0(0,\ell)\subseteq L^2(0,\ell)\times L^2(0,\ell)\to \RE\,,\quad
\ee
\be
F_D(u,v)=\left\langle \frac{du}{dx},\frac{dv}{dx}\right\rangle\equiv \int_0^\ell \frac{du}{dx}\frac{dv}{dx}\,dx\,.
\ee
Here the index $D$ stands for Dirichlet boundary conditions. $H^n(0,\ell)$ denotes the usual Sobolev-Hilbert space of square integrable functions with square integrable distributional derivatives 
up to the $n$-th order and 
\be
H^1_0(0,\ell):=\left\{u\in H^1(0,\ell)\,:\, \gamma_0 u=0\right\}\,,
\ee
\be
\gamma_0:H^1(0,\ell)\to\RE^2\,,\qquad \gamma_0 u:=\left(u(0),u(\ell)\right)\,,
\ee
By Sobolev embedding theorems one has 
$H^n(0,\ell)\subset C^{n-1}[0,\ell]$. \par
The closure of $A_\min$ is given by 
\be
A_\min^{**}:H^2_0(0,\ell)\subseteq L^2(0,\ell)\to
L^2(0,\ell)\,,\qquad A_{\min}^{**}u=\frac{d^2u}{dx^2}\,,
\ee 
where
\be
H^2_0(0,\ell):=\left\{u\in H^2(0,\ell)\,:\, \gamma_0u=\gamma_1u=0\right\}\,,
\ee
\be
\gamma_1:H^2(0,\ell)\to\RE^2\,,\qquad \gamma_1 u:=\left(\frac{du}{dx}(0),-\frac{du}{dx}(\ell)\right)\,,
\ee
and in order to find all self-adjoint extensions of $A_\min$ 
together with
the corresponding bilinear  forms
we can apply 
Theorems \ref{estensioni} and \ref{qf} with 
\be
A_0=A_D\,,\quad F_0=F_D\,,\quad \fh=\RE^2\,,\quad 
\tau=\gamma_1|H^2(0,\ell)\cap H^1_0(0,\ell)\,.
\ee
By  
\begin{align}
&R^{D}_{0}u(x)\equiv\left(-A_D\right)^{-1}u(x)=
\frac{\ell-x}{\ell}\int_0^x y\,u(y)\,dy+
\frac{x}{\ell}\int_x^\ell(\ell-y)\,u(y)\,dy\,,
\end{align}
and, for $\lambda>0$,
\begin{align}
&R^{D}_{\lambda}u(x)\equiv\left(-A_{D}+\lambda\right)^{-1}u(x)\\
=&
\frac{\sinh(\sqrt \lambda\,(\ell-x))}{\sqrt {\lambda}\sinh(\sqrt {\lambda}\,\ell)}
\int_0^x{\sinh(\sqrt {\lambda}\,y)}\,u(y)\,dy
%&\\
+\frac{\sinh(\sqrt {\lambda}\,x)}{\sqrt {\lambda}\sinh(\sqrt {\lambda}\,\ell)}
\int_x^\ell{\sinh(\sqrt {\lambda}\,(\ell-y))}\,u(y)\,dy\,,
\end{align}
one obtains (here $\xi\equiv(\xi_1,\xi_2)$)
\begin{equation}
%\label{G0}
G_0:\RE^2\to L^2(0,\ell)\,,\quad G_0\xi(x)=
\frac{x}{\ell}\,(\xi_2-\xi_1)+\xi_1\,,
\end{equation}
and, for $\lambda>0$,
\begin{equation}
G_\lambda:\RE^2\to L^2(0,\ell)\,,\quad
G_\lambda\xi(x)=\frac{\sinh(\sqrt
    {\lambda}\,(\ell-x))}{\sinh(\sqrt {\lambda}\,\ell)}\ \xi_1+
\frac{\sinh(\sqrt{\lambda}\,x)}{\sinh(\sqrt {\lambda}\,\ell)}\ \xi_2 \,.
\end{equation}
Since $G_0\xi\notin H^1_0(0,\ell)$ for any $\xi\not=0$, hypothesis \eqref{hypp} holds. Thus, by Theorem \ref{qf}, the bilinear  forms corresponding to the self-adjoint extensions of $A_\min$ are of the kind
\be
F_{(\Pi,\Theta)}:\D(F_{(\Pi,\Theta)})\times\D(F_{(\Pi,\Theta)})\subseteq
L^2(0,\ell)\times L^2(0,\ell)\to\RE\,,
\ee
\be
F_{(\Pi,\Theta)}(u,v)=\left\langle \frac{du_0}{dx},\frac{dv_0}{dx}\right\rangle +f_{\Pi,\Theta}(\xi_u,\xi_v,)\,,
\ee
\be
\D(F_{(\Pi,\Theta)})=\{u\in H^1(0,\ell)\,:\,u=u_0+G_0\xi_u\,,\ u_0\in H_0^1(0,\ell)\,,\ \xi_u\in\R(\Pi)\}\,,
\ee
where
\be
f_{\Pi,\Theta}:\R(\Pi)\times\R(\Pi)\to\RE\,,
\ee
is the bilinear  form on $\R(\Pi)$ corresponding to $\Theta$ and $(\Pi,\Theta)\in\E(\RE^2)$. By straightforward calculations, integrating by parts, by $\frac{d^2}{dx^2}(G_0\xi)=0$ and by 
\be
\gamma_0u=\gamma_0 (u_0+G_0\xi_u)=\gamma_0G_0\xi_u=\xi_u\,,
\ee 
one obtains
\begin{align}
\left\langle \frac{du_0}{dx},\frac{dv_0}{dx}\right\rangle=
%&\left\langle \frac{d}{dx}(u-G_0\xi_u),\frac{d}{dx}(v-G_0\xi_v)\right\rangle\\
%=&\left\langle \frac{du}{dx},\frac{dv}{dx}\right\rangle+
%\left\langle \frac{d}{dx}(G_0\xi_u),\frac{d}{dx}(G_0\xi_v)\right \rangle\\
%-&\left\langle \frac{du}{dx},\frac{d}{dx}(G_0\xi_v)\right \rangle-
%\left\langle \frac{d}{dx}(G_0\xi_u),\frac{dv}{dx} \right\rangle\\
%=&\left\langle \frac{du}{dx},\frac{dv}{dx}\right\rangle-\left\langle \frac{d}{dx}(G_0\xi_u),\frac{d}{dx}(G_0\xi_v)\right \rangle\\
%=&
\left\langle \frac{du}{dx},\frac{dv}{dx}\right\rangle+ P_{0}\xi_u\ccdot\xi_v\,,
\end{align}
where 
\be
P_{\lambda}:\RE^2\to\RE^2\,,\quad\lambda\ge 0\,,
\ee
is the Dirichlet-to-Neumann operator 
\be
P_{\lambda}\xi:=\gamma_1 u_\xi\,,\quad
\begin{cases}
\frac{d^{2}}{dx^{2}}\,u_\xi=\lambda u_{\xi}&\\
\gamma_0u_\xi=\xi\,,&
\end{cases}
\ee
i.e.
\be
P_{0}\equiv\frac{1}{\ell}
\left(\begin{matrix}-1&{\ \ }1\\
{\ \ }1&-1
\end{matrix}\right)\,,
\ee
\be
P_{\lambda}\equiv\frac{\sqrt\lambda}{\sinh\sqrt\lambda\,\ell}\left(\begin{matrix}-\cosh\sqrt\lambda\,\ell&{\ \ }1\\
{\ \ }1&-\cosh\sqrt\lambda\,\ell
\end{matrix}\right)\,,\quad\lambda>0\,.
\ee
Noticing that any $u\in H^1(0,\ell)$ can be decomposed as $u=u_0+G_0\gamma_0u$, where $u_0\in H^1_0(0,\ell)$ is defined by $u_0:=
u-G_0\gamma_0u$, in conclusion one has that, for any $(\Pi,\Theta)\in \E(\RE^2)$,
\be
F_{(\Pi,\Theta)}(u,v)=F_N(u,v) + f_{\Pi,B_\Theta}(\gamma_0u,\gamma_0v)\,,
\ee
\be
\D(F_{(\Pi,\Theta)})=\{u\in H^1(0,\ell)\,:\gamma_0u\in\R(\Pi)\}\,,
\ee
where
\be
F_N:H^1(0,\ell)\times H^1(0,\ell)\subseteq L^2(0,\ell)\times L^2(0,\ell)\to\RE\,,\quad
%\ee
%\be
F_N(u,v)=\left\langle \frac{du}{dx},\frac{du}{dx}\right\rangle\,,
\ee
and 
\be
B_\Theta:\R(\Pi)\to\R(\Pi)\,,\quad B_\Theta:=\Pi P_{0}\Pi+\Theta\,.
\ee
Here the index N stands for Neumann, since the self-adjoint operator $A_N$ associated  with  $F_N$ corresponds to Neumann boundary conditions.
\par
Then, posing 
\be
\t A_{(\Pi,B)}:=
A_{(\Pi,\Theta_B)}\,,\qquad\Theta_B:=B-\Pi P_{0}\Pi\,,
\ee
by Theorem \ref{estensioni} one obtains (see  Example 5.1  in \cite{[P08]}) the following 
result, which is nothing but our version of results that be extracted from the theory of self-adjoint extension of symmetric Sturm-Liouville operators (see e.g. \cite{[Weid]}, Section 4):  
\begin{theorem} $\Ex (A_\min)=\{\t A_{(\Pi,B)}\,,\ (\Pi,B)\in \E(\RE^2)\}$, where 
\be
\t A_{(\Pi,B)}:\D(\t A_{(\Pi,B)})\subseteq L^2(0,\ell)\to
L^2(0,\ell)\,,\qquad \t A_{(\Pi,B)}u=\frac{d^{2}u}{dx^2}\,,
\ee
\be
\D(\t A_{(\Pi,B)})=\left\{u\in
H^2(0,\ell)\,:\,\gamma_0u\in\Ran(\Pi)\,,\quad \Pi\gamma_1 u=
B \gamma_0 u
\right\}\,.
\ee
\end{theorem}
\begin{remark}
Notice that the case $\Pi=0$ corresponds to Dirichlet boundary conditions, the case $\Pi=\uno$, posing
\be
B=\left(\begin{matrix}b_{11}&b_{12}\\
b_{12}&b_{22}
\end{matrix}\right)\,,\quad b_{ij}\in\RE\,,
\ee
corresponds to the boundary conditions 
\begin{align}\label{bc11}
\frac{du}{dx}(0)&=b_{11}\,u(0)+b_{12}\,u(\ell)\,, 
\qquad -\frac{du}{dx}(\ell)=b_{12}\,u(0)+b_{22}\,u(\ell)\,,
\end{align}
and the case $\Pi=v\otimes v$,
$v\equiv(v_1,v_2)\in\RE^2$, $v_1^2+v_2^2=1$, $B=b\in\RE$, 
corresponds to the boundary conditions
\begin{align}\label{bc12}
v_2\,u(0)&=v_1\,u(\ell)\,,\qquad
v_1\,\left(\frac{du}{dx}(0)-b\,u(0)\right)
=v_2\,\left(\frac{du}{dx}(\ell)+b\,u(\ell)\right)\,.
\end{align}
The boundary conditions \eqref{bc12} can be re-written (when $bv_1v_2\not=0$) as
\begin{align}\label{bc12feller}
\frac{b}{v_1}\,u(0)=\frac{b}{v_2}\,u(\ell)=
v_1\frac{du}{dx}(0)-v_2\frac{du}{dx}(\ell)\,.
\end{align}
\end{remark}
\begin{remark}
The boundary conditions \eqref{bc11} 
and \eqref{bc12feller} 
coincide with the ones obtained (in the case of more general second order differential operators) by Feller in \cite{[Fel]}, Theorem 10.1. 
\end{remark}

\begin{lemma}\label{Diff} 1. In the case $\Pi=\uno$ 
\be
f_B:\RE^2\times\RE^2\to\RE\,,\qquad
f_{B}(\xi,\zeta):=b_{11}\xi_1\zeta_1+b_{12}(\xi_1\zeta_2+\zeta_{1}\xi_{2})+b_{22}
\xi_2\zeta_2
\ee
is Dirichlet form on $\RE^2$ if and only if
\be
b_{11}+b_{12}\ge 0\,,\quad b_{12}+b_{22}\ge 0\,,\quad b_{12}\le 0\,.
\ee
2. In the case $\Pi=v\otimes v$ 
\be
f_{\Pi,B}:\R(\Pi)\times\R(\Pi)\subset\RE^2\times\RE^2\to\RE\,,\qquad
f_{\Pi,B}(\xi v,\zeta v):=b\,\xi\zeta
\ee
is a Dirichlet form  on $\RE^2$ with domain $\R(\Pi)$ if and only if $b\ge 0$ and $v$ is one of the following unit vectors:
\be
\e_0\equiv\left(\frac{1}{\sqrt 2}\,,\frac{1}{\sqrt 2}\right)\,,\quad 
\e_1\equiv(1,0)\,,\quad \e_2\equiv(0,1)\,.
\ee
\end{lemma}
\begin{proof}
1. Let $\Pi=\uno$. The thesis follows by Theorem \ref{dirichlet} and 
by the well known characterization of 
generators of Markovian semigroups on a finite (or countable) set (see e.g. \cite{[Dv2]}, Theorem 12.3.2).\par
2. Let $\Pi=v\otimes v$, $v\equiv(v_1,v_2)\in\RE^2$, $v_1^2+v^2_2=1$. If $v=\e_i$, $i=0,1,2$ then 
$(cv)_\#\in\R(\Pi)$ for any $c\in\RE$. Conversely suppose 
that $v\notin\R(\e_i\otimes \e_i)$, $i=0,1,2$, so that $v_1v_2\not=0$ and
$v_1\not=v_2$. If $v_1v_2<0$ then $v_\#\notin\R(\Pi)$; if $v_1v_2>0$ 
then, posing $c=\text{sign}(v_1)/(|v_1|\wedge |v_2|)$, $(cv)_\#\equiv(1,1)\notin\R(\Pi)$. Finally notice that 
$f_{\Pi,B}\ge 0\iff b\ge 0$.
\end{proof}
\begin{theorem}\label{lemmaintervallo} Let $\t F_{(\Pi,B)}$ be the bilinear  form associated  with  $-\t A_{(\Pi,B)}$. Then $\t F_{(\Pi,B)}$ is a Dirichlet form if and only if $f_{\Pi,B}$ is a Dirichlet form  on $\RE^2$. 
\end{theorem}
\begin{proof}
At first notice that the bilinear  form $F_N$ is a Dirichlet form. Thus if $f_{\Pi,B}$ is a Dirichlet form then, since $(\gamma_0u)_\#=
\gamma_0u_\# $ by $H^1(0,\ell)\subset C[0,\ell]$, one has 
\be
u\in\D(\t F_{(\Pi,B)})\quad\Longrightarrow\quad u_\#\in\D(\t 
F_{(\Pi,B)})
\ee
and
\be
\t F_{(\Pi,B)}(u_\#)=F_N(u_\#)+f_{\Pi,B}((\gamma_0 u)_\#)\le 
F_N(u)+ f_{\Pi,B}({\gamma_0 u})=\t F_{(\Pi,B)}(u)\,.
\ee
Therefore 
$\t F_{(\Pi,B)}$ is a Dirichlet form. \par 
Suppose now that 
$f_{\Pi,B}$ is not a Dirichlet form.\par
Case $\Pi=\uno$.  Let 
$u=G_0\xi$, with $\xi\equiv(1+\epsilon,1)$, $\epsilon>0$. Then 
\be
f_{B}(\xi_\#)-f_{B}(\xi)
=-2(b_{11}+b_{12})\epsilon-b_{11}\epsilon^2
\ee
and
\be
F_N(u)-F_N(u_\#)=F_N(u)=\frac{\epsilon^2}{\ell}\,.
\ee
Thus if $b_{11}+b_{12}<0$ then $\t F_{(\Pi,B)}(u_\#)>\t 
F_{(\Pi,B)}(u)$ by taking $\epsilon$ sufficiently small. The same kind of reasoning holds in the case $b_{12}+b_{22}<0$ taking 
$\xi\equiv(1,1+\epsilon)$. Let 
\be
u(x)=\begin{cases}1-\frac{2}{\ell}\,x\,,&0<x\le 
\frac{\ell}{2}\\
\epsilon\,\left(1-\frac{2}{\ell}\,x\right)\,,&
\frac{\ell}{2}<x< \ell\,.\end{cases}
\ee
Then $\gamma_0u\equiv(1,-\epsilon)$,    
\be
f_{B}(\gamma_0u_\#)-f_{B}(\gamma_0u)=2b_{12}\epsilon-b_{22}\epsilon^2
\ee
and
\be
F_N(u)-F_N(u_\#)=\frac{2}{\ell}\,\epsilon^2\,.
\ee
Thus if $b_{12}>0$ then $\t F_{(\Pi,B)}(u_\#)>\t F_{(\Pi,B)}(u)$ by taking $\epsilon$ sufficiently small. \par
Case $\Pi=v\otimes v$. Suppose $\xi\not= \e_i$, $i=0,1,2$. Then 
taking $u=G_0\xi$, one has $\gamma_0 u_\#=\xi\not\in\R(\Pi)$, and 
so $u_\#\notin \D(\t F_{(\Pi,B)})$. 
\par
Suppose $u=G_0\e_0$ and $b<0$. Then
\be
\t F_{(\Pi,B)}(u)=f_{\Pi,B}(e_0)=b<0\,.
\ee
Let 
\be
u(x)=\begin{cases}\frac{2}{\ell}\,x\,,&0<x\le 
\frac{\ell}{2}\\
\epsilon\,\left(\frac{2}{\ell}\,x-1\right)+1\,,&
\frac{\ell}{2}<x< \ell\,.\end{cases}
\ee
Then $\gamma_0u=(1+\epsilon)\e_2$, $u\in\D(\t F_{(\Pi,B)})$,   
and
\be
f_{\Pi,B}(\gamma_0u_\#)-f_{\Pi,B}(\gamma_0u)=-b(2+\epsilon)\epsilon
\ee
\be
F_N(u)-F_N(u_\#)=\frac{2}{\ell}\,\epsilon^2\,.
\ee
Thus if $b<0$ then $\t F_{(\Pi,B)}(u_\#)>\t F_{(\Pi,B)}(u)$ by taking
$\epsilon$ sufficiently small. 
The case in which $v=\e_1$ is treated similarly.
\end{proof}
\begin{remark}\label{toyregular}
If $\Pi=\zero$ then $\t F_{(\zero)}=F_{D}$ and such a Dirichlet form is regular on $X=(0,\ell)$. When $\Pi\not=\zero$, by the continuous embedding $H^{1}(0,\ell)\subset C[0,\ell]$ and by 
\be
\D(\t F_{(\Pi,B)})=\{u=u_{0}+G_{0}\xi_{u}\,,\ u_{0}\in H^{1}_{0}(0,\ell)\,,\ \xi_{u}\in\R(\Pi)\}\,,
\ee
\be
\t F_{(\Pi,B)}(u,v)=F_{N}(u,v)+B\xi_{u}\!\cdot\!\xi_{v}\,,
\ee
it is immediate to check that  $\t F_{(\Pi,B)}$ is always regular on $X=[0,\ell]$. It suffices to approximate (with respect to $H^{1}(0,\ell)$-convergence) the component $u_{0}$ by a sequence in $C^{\infty}_{0}(0,\ell)$.  
\end{remark}
\begin{corollary}\label{toy2} A bilinear form $F$ on $L^{2}(0,\ell)$ is the bilinear form 
of a Markovian self-adjoint extension of $A_{\min}$  if and only if $\D(F)\subseteq H^{1}(0,\ell)$ and there exists a Dirichlet form $f_{b}$ on $\RE^{2}$ such that 
\be
\D(F)=\{u\in H^{1}(0,\ell):\gamma_{0}u\in\D(f_{b})\}
\ee 
and 
\be 
F(u,v)=F_{N}(u,v)+f_{b}(\gamma_{0}u,\gamma_{0}v)\,.
\ee
\end{corollary}
By Kre\u\i n's resolvent formula \eqref{R} and by 
\begin{align}
\Theta_{B}+\lambda \Pi G_{0}^{*}G_{\lambda}\Pi=\Theta_{B}+\Pi\gamma_{1}(G_{0}-G_{\lambda})\Pi
=\Theta_{B}+\Pi(P_{0}-P_{\lambda})\Pi=B-\Pi P_{\lambda}\Pi\,,
\end{align}
Corollary \ref{toy2} has an equivalent version in terms of resolvents:
\begin{corollary}\label{toyresolvent}
$R_{\lambda}\in \B(L^{2}(0,\ell))$, $\lambda> 0$, is the resolvent of a Markovian extension of $A_{\min}$ if and only if  
\be
R_{\lambda}=R^{D}_{\lambda}+G_{\lambda}\Pi(B-\Pi P_{\lambda}\Pi)^{-1}\Pi G^{*}_{\lambda}\,,
\ee
where $f_{\Pi,B}$ is a Dirichlet form on $\RE^{2}$.
\end{corollary}
By combining Theorem \ref{dirichlet} with Lemma \ref{Diff} one obtains the following
\begin{theorem}\label{teointervallo} $A\in \ExM(A_\min)$ if and only if 
$A=\t A_{(\Pi,B)}$, with $(\Pi,B)\in\E(\RE^2)$ 
satisfying one of the
  following conditions:
\begin{enumerate}
\item $\Pi=0$; 
\item 
$\Pi=\uno\,,\quad B=\left(\begin{matrix}b_{12}&b_{12}\\
b_{12}&b_{22}
\end{matrix}\right),\ 
b_{11}+b_{12}\ge 0\,,\ b_{12}+b_{22}\ge 0\,,\ b_{12}\le 0\,;
$
\item
$\Pi=v\otimes v\,,\quad v\in\{\e_0\,,\, \e_1\,,\, \e_2\}\,,\quad B=b\ge 0\,. $ 
\end{enumerate}
\end{theorem}
\begin{remark}
Notice that $A_N$ is the maximal element of $\ExM(A_\min)$ and that it does not coincide with the Kre\u\i n's extension $A_K$.
\end{remark}
\begin{remark}
The above theorem reproduces, if $\Pi=\uno$, the results obtained (in the case of more
general second order differential operators) by Feller in \cite{[Fel]}, Theorem 10.2. If $\Pi=v\otimes v$ our results differ from the ones 
stated by Feller in the same theorem. However looking at the proof of Feller's theorem, the arguments there seems to lead to our results. A thorough study of the diffusion processes related to Dirichlet forms on an interval by means of boundary conditions is given in the recent paper \cite{[Fu13]}.
\end{remark}
\begin{remark}
Notice that the cases $\Pi=\e_i\otimes\e_i$, $B=0$, $i=1,2$, 
give mixed (Dirichlet-Neumann) boundary conditions, while the case
$\Pi=\e_0\otimes\e_0$, $B=0$ gives the self-adjoint extension 
corresponding to the Laplacian on the circle of radius $\ell/2\pi$. 
\end{remark}
\begin{remark}
By \eqref{conserviff2}, Lemma \ref{Diff} and Theorem \ref{lemmaintervallo} one has that 
$\t A_{(\Pi,B)}$ is conservative (equivalently recurrent) if and only if either $(\Pi,B)=(\uno\,,-bP_0)$, $b\ge 0$, or $(\Pi,B)=(\e_{0}\otimes\e_{0},0)$.
\end{remark}
\begin{remark}
Notice that while 
\be
f_{\Pi,B}\quad \text{Dirichlet}\quad 
\Longrightarrow\quad f_{\Pi,\Theta_B} \quad\text{Dirichlet},
\ee
the converse is not true: for example (here for simplicity we pose $\ell=1$), taking $\Pi=\uno$ and 
\be
\Theta=\left(\begin{matrix}2&-1/2\\
-1/2&2\end{matrix}\right)\,, 
\ee 
$f_\Theta$ is a Dirichlet form whereas $f_{B_\Theta}$ is positive but is not a Dirichlet form. 
%$$
%B_\Theta=\left(\begin{matrix}1&1/2\\
%1/2&1
%\end{matrix}\right)\,,$$
Hence $f_{\Pi,\Theta}$ 
Dirichlet is a necessary but not sufficient condition for $A_{(\Pi,\Theta)}$ being Markovian. This elementary example, 
together with Theorem \ref{lemmaintervallo}, shows that the correct description of the set 
of Markovian extensions $\ExM(A_\min)$ in terms of the parametrizing couples $(\Pi,\Theta)$ is the following:
\end{remark}
\begin{theorem}\label{toy} $A_{(\Pi,\Theta)}\in \ExM(A_\min)$ if and only if  
\be f_{\Pi,\Theta}:\R(\Pi)\times\R(\Pi)\to\RE
\ee
is a Dirichlet form on $\RE^2$ which admits the decomposition  
\be
f_{\Pi,\Theta}(\xi_1,\xi_2)=f_{\Pi,B}(\xi_1,\xi_2)-P_0\xi_1\ccdot\xi_2\,,
\ee
where 
\be
f_{\Pi,B}:\R(\Pi)\times\R(\Pi)\to\RE
\ee
is a Dirichlet form on $\RE^2$.
\end{theorem}
%one which uses the operators $\t A_{(\Pi,B)}$ and not the one with uses the $A_{(\Pi,\Theta)}$'s. However, by the correspondence $\t A_{(\Pi,B)}=A_{(\Pi,\Theta_B)}$, Krein's formula \eqref{R} can be applied to  Markovian extensions too. 
%\end{section}
\section{Markovian extensions of elliptic operators}
Let $\Omega\subset\RE^n$, $n>1$,  a bounded open set 
with a smooth
boundary $\Gamma$. We suppose that $\Omega$ is connected, otherwise we work on each connected component separately.
\par 
Given the differential expression
\be
\nabla\!\cdot\! a\nabla\equiv\sum_{1\le i,j\le n}\frac{\partial\,}{\partial x_{i}}\left(a_{ij}\frac{\partial\,}{\partial x_j}\right)\,.
\ee
we suppose that the real-valued matrix $a(x)\equiv(a_{ij}(x))$ 
is symmetric, that $a_{ij}\in C^\infty(\Omega)$ and that there exist $0<\mu_0\le \mu_{1}<+\infty$  such that
\be
\forall x\in\Omega\,,\ \forall\xi\in\RE^n\,,\qquad  \mu_{0}\|\xi\|^{2}\le\sum_{1\le i,j\le n}a_{ij}(x)\xi_i\xi_j\le \mu_1\|\xi\|^2\,.
\ee  
We denote by $H^k(\Omega)$ the Sobolev-Hilbert  space given
by closure of  $C^{\infty}(\bar\Omega)$ with respect to the norm
\be
\|u\|_{H^k(\Omega)}^2=\sum_{0\le \alpha_1+\dots+\alpha_n\le k} \left\|
\frac{\partial^{\alpha_1}}{\partial x_1^{\alpha_1}}\dots \frac{\partial^{\alpha_n}u}{\partial  x_n^{\alpha_n}}\right\|^2_{L^2(\Omega)}\,.
\ee
Then the spaces $H^s(\Omega)$, $s\ge 0$ real, can be  defined by interpolation as $H^s(\Omega):=[H^k(\Omega),L^2(\Omega)]_\theta$, 
$(1-\theta)k=s$, $0<\theta<1$. Alternatively  $H^s(\Omega)$ can be defined as the space of restrictions to $\Omega$ of the elements of $H^s(\RE^n)$, the latter defined by Fourier transform.
% (for details see e.g. \cite{[LM]}, section 9, chapter 1).
\par
$H^s_0(\Omega)$, $s>0$, denotes  the closure
of $C_c^{\infty}(\Omega)$ with respect to the $H^s(\Omega)$-norm. If $0<s\le 1/2$ then $H^s_0(\Omega)=H^s(\Omega)$, otherwise 
$H^s_0(\Omega)\subsetneq H^s(\Omega)$. 
Here we use also the alternative definitions 
\be
H^s_0(\Omega):=\{u\in H^s(\Omega)\,:\, \gamma_k u=0\,,\ 0\le k<s-1/2\}
\ee
where 
\be
\gamma_k:H^s(\Omega)\to H^{s-k-\frac12}(\Gamma )\,,\quad s> k+1/2\,,
\ee
is defined as the 
unique continuous and
surjective linear maps
 such that
\be
\gamma_k u\,(x)=\frac{\partial^k u}{\partial \nu^k_a}\,(x)
\,,\qquad u\in C^\infty(\bar\Omega)\,,\quad x\in\Gamma\,.
\ee  
Here $\frac{\partial^k\ }{\partial \nu^k_a}$ denotes the $k$-th order 
directional derivative along  the vector $\nu_a:=a\nu$, where $\nu$ is the inward normal vector on
$\Gamma $. \par
The vector spaces $H^s(\Gamma )$, $s\in\RE$, are the Sobolev-Hilbert spaces, defined, since $\Gamma $ can be made a smooth compact Riemannian
manifold, as the
completion of $C^\infty(\Gamma )$ with respect of the scalar
product 
\be
\langle h_1,h_2\rangle_{H^s(\Gamma )}
:=\langle h_1,(-\Delta_{LB}+1)^{s}h_2\rangle_
{L^2(\Gamma )}\,.
\ee
Here the self-adjoint
operator $\Delta_{LB}$ is the Laplace-Beltrami operator in 
$L^2(\Gamma )$; $(-\Delta_{LB}+1)^{s}$ 
can be extended to a unitary map, 
which we denote by the same
symbol, 
\be
(-\Delta_{LB}+1)^{s}:H^{r}(\Gamma )
\to H^{r-2s}(\Gamma )\,.
\ee
For successive notational convenience we pose
\be
\Lambda:=(-\Delta_{LB}+1)^{\frac14}:H^{s}(\Gamma )
\to H^{s-\frac12}(\Gamma )\,,\quad \Sigma:=\Lambda^{-1}
\ee
and we denote by $(\cdot,\cdot)_{-s,s}$ the duality between $H^{-s}(\Gamma)$ and $H^s(\Gamma)$, i.e 
\be
(h_1,h_2)_{-s,s}=\langle\Sigma^{2s}h_1,\Lambda^{2s}h_2\rangle_{L^2(\Gamma)}\,.
\ee
\vskip 10pt
\noindent
{\bf Warning.} Notice that in \cite{[P08]} and \cite{[P10]}  
$\Lambda$ has been defined as 
$(-\Delta_{LB}+1)^{\frac12}$. Such a change in notation is due to a different choice of the space $\fh$: 
$L^2(\Gamma )$ instead of $H^{\frac12}(\Gamma )$.  
\vskip 10pt
\noindent
\begin{remark}\label{norm1/2}
In the following we use also the equivalent Besov-type norm  on $H^s(\Gamma)$, $0<s<1$, defined by
\be
\|h\|^2_{H^s(\Gamma)}=
\|h\|^2_{L^2(\Gamma)}+\int_{\Gamma\times\Gamma}\frac{|h(x)-h(y)|^2}{\|x-y\|^{n+2s-1}}\,
d\sigma(x)d\sigma(y)\,,
\ee
where $\sigma$ denotes surface measure. By 
$|\,|a|-|b|\,|\le |a-b|$, one immediately gets 
$\|\,|h|\,\|_{H^s(\Gamma)}\le \|h\|_{H^s(\Gamma)}$. By such an inequality, $H^{s}(\Gamma)$ is a Dirichlet space for any  $0<s< 1$. By $\|\,|h|\,\|_{H^1(\Gamma)}= 
\|h\|_{H^1(\Gamma)}$, also $H^{1}(\Gamma)$ is a Dirichlet space.
\end{remark}
The symmetric operator $S=-A_\min$,
\be
A_\min :C_c^\infty(\Omega)\subseteq L^2(\Omega)\to L^2(\Omega)\,,\quad A_\min u:=\nabla\!\cdot\! a\nabla u\,, 
\ee
is positive and its Friedrichs' extensions $A_D$ (here the index $D$ stands for Dirichlet boundary conditions) is 
given by
\be
A_D: H^2(\Omega)\cap H^1_0(\Omega)\subseteq L^2(\Omega)\to L^2(\Omega)\,,
\qquad A_Du=\nabla\!\cdot\! a\nabla u
\ee
with corresponding bilinear  form
\be
F_D:H^1_0(\Omega)\times H^1_0(\Omega)\subseteq 
L^2(\Omega)\times L^2(\Omega)\to\RE\,,
\ee
\be
F_D(u,v):= \langle\nabla u,a\nabla v\rangle_{L^2(\Omega)}
\equiv\sum_{1\le i,j\le n}\int_\Omega a_{ij}\,\frac{\partial u}{\partial x_i}\,
\frac{\partial u}{\partial x_{j}}\,dx\,.
\ee
$A_D$ has a compact resolvent and its spectrum
consists of an infinite sequence of negative eigenvalues, each having finite multiplicity.\par
The closure of $A_\min$ is given by 
\be
A_\min^{**}:H^2_0(\Omega)\subseteq L^2(\Omega)\to
L^2(\Omega)\,,\qquad A_\min^{**}u=\nabla\!\cdot\! a\nabla u\,,
\ee 
and in order to find all self-adjoint extensions of $A_\min$ 
we can apply 
Theorems \ref{estensioni} and \ref{qf} with 
\be
A_0=A_D\,,\quad F_0=F_D\,,\quad \fh=L^2(\Gamma )
\,,\quad \tau=\Lambda\gamma_1|(H^2(\Omega)\cap H^1_0(\Omega))\,.
\ee
Notice that $\K(\tau)=H^2_0(\Omega)$ and 
that $\tau$ is surjective by the
surjectivity of 
\be
\gamma: H^{2}(\Omega)\to H^{\frac32}(\Gamma )\times 
H^{\frac12}(\Gamma )\,,\quad \gamma u:=(\gamma_0u,\gamma_1u)\,.
\ee
In order to write down 
the extensions of $A_\min$ together
with their resolvents, we make explicit the operator $G_0$. 
One has $A_\min^*=A_\max$, where $A_{\max}$, the maximal realization of $\diff$, is defined by
\be
A_{\max}:\D(A_\max)\subseteq L^2(\Omega)\to L^2(\Omega)\,,\qquad
A_{\max}u:=\diff u\,,
\ee
\be
\D(A_\max):=\{u\in L^2(\Omega)\,:\, \diff u\in L^2(\Omega)\}\,.
\ee
The maps $\gamma_0$ and $\gamma_1$ can be extended to (see \cite{[LM]}, Chapter 2, Section 6.5)
\be
 \hat\gamma_0:\D(A_{\max})\to H^{-\frac12}(\Gamma )\,,
\ee
\be
\hat\gamma_1:\D(A_{\max})\to H^{-\frac32}(\Gamma )\,,
\ee 
and Green's formula can be
extended to the case in which $u\in \D(A_\max)$, $v\in H^2(\Omega)\cap H^1_0(\Omega)$:
\begin{equation}\label{Green}
\langle A_\max u,v\rangle_{L^2(\Omega)}= \langle
u,A_Dv\rangle_{L^2(\Omega)} +(\hat\gamma_0
u,\gamma_1v)_{-\frac12,\frac12}\,.
\end{equation}
Moreover for any $u\in \D(A_\max)\cap H^1(\Omega)$ one has $\hat\gamma_1
u\in H^{-\frac12}(\Gamma)$ and then for any $v\in H^1(\Omega)$ the "half" Green's formula holds (see e.g. \cite{[McL]}, Theorem 4.4):
\begin{equation}\label{halfGreen}
\langle -A_\max u,v\rangle_{L^2(\Omega)}= \langle
\nabla u,a\nabla v\rangle_{L^2(\Omega)} +(\hat\gamma_1
u,\gamma_0v)_{-\frac12,\frac12}\,.
\end{equation}
By Remark \ref{aggiunto}, since $A_\max=A_\min^*$, we have 
$A_\max G_0h=0$ and so by (\ref{Green}) there follows,
for all $h\in L^{2}(\Gamma )$ and for all $u\in\D(A_D)$,
\be
\langle G_0h,A_Du\rangle_{L^2(\Omega)}= -(\hat \gamma_0
G_0h,\gamma_1u)_{-\frac12,\frac12}\,.
\ee
Since, by (\ref{Green0}),
\begin{align} 
\langle G_0h,A_Du\rangle_{L^2(\Omega)}= 
\langle G_0h,A_\max u\rangle_{L^2(\Omega)}=
\langle G_0h,A_\min^*u\rangle_{L^2(\Omega)}=
-\langle
h,\Lambda\gamma_1 u\rangle_{L^{2}(\Gamma )}\,,
\end{align} 
one obtains $\hat \gamma_0G_0h=\Lambda h$. Thus $G_0h$ is the unique solution of
the Dirichlet boundary value problem
\begin{equation}\label{dbvp}
\begin{cases}
A_\max G_0h=0\,,\\ \hat\gamma_0\, G_0h=\Lambda h\,,
\end{cases}
\end{equation}
i.e. 
\be
G_0\Sigma=K_0\,,
\ee 
where 
\be
K_\lambda :H^{-\frac12}(\Gamma )\to\D(A_\max)\,,\qquad \lambda\ge 0\,,
\ee 
denotes the Poisson operator which provides the unique solution of the Dirichlet problem with boundary data in $H^{-\frac12}(\Gamma )$ (see \cite{[LM]}, Chapter 2, Section 6):
\begin{equation}\label{Klambda}
\begin{cases}
A_\max K_{\lambda}h=\lambda K_{\lambda}h\,,\\ 
\hat\gamma_0\, K_{\lambda}h=h\,.
\end{cases}
\end{equation}
Posing 
\be
R_{\lambda}^{D}:=(-A_{D}+\lambda)^{-1}\,,
\ee
by \eqref{1.2} and by 
\begin{equation}\label{KL}
K_{\lambda}=K_{0}-\lambda R^{D}_{\lambda}K_{0}\,,
\end{equation}
one has 
\be
G_\lambda\Sigma=K_\lambda\,.
\ee
%Notice
%that $G_0h$ is uniquely defined as the solution of
%(\ref{dbvp}): for any other solution $u$ one has
%$u-G_0h\in\K(A_D)=\{0\}$.  
\begin{remark}\label{regularity}
By elliptic regularity  (see e.g. \cite{[G68]}, Proposition III 5.2), $K_\lambda$ is a topological isomorphism from $H^s(\Gamma)$ onto $H^{s+\frac12}(\Omega)$ for any  
$s\ge-\frac12$, 
so that, for all $s\ge 0$, 
\be
G_\lambda h\in H^s(\Omega)\iff h\in H^s(\Gamma)\,.
\ee
Since $\R(G_0)=\K(A_\min^*)=\K(A_\max)$, $G_0h\in H^1_0(\Omega)$
implies $G_0h\in H^2(\Omega)\cap H^1_0(\Omega)=\D(A_D)$. However by
(\ref{dbvp}) this implies $h=0$. Thus (\ref{hyp}) always holds true in
this case and we can apply Theorem \ref{qf} to $S=A_\min$.
\end{remark}
By $\Sigma\hat\gamma_0\, G_0h=h$, for any $(\Pi,\Theta)\in\E(L^2(\Gamma
))$ one has (this is our version of Theorem 2.2 in \cite{[G70]})

\begin{theorem}\label{lemmaforma} Let 
$A_{(\Pi,\Theta)}\in\Ex(A_\min)$ and let $F_{(\Pi,\Theta)}$ the 
symmetric bilinear form associated  with  $-A_{(\Pi,\Theta)}$. Then
\be
F_{(\Pi,\Theta)}:\D(F_{(\Pi,\Theta)})\times \D(F_{(\Pi,\Theta)})\subseteq
L^2(\Omega)\times L^2(\Omega)\to\RE\,,
\ee
\begin{align}
\D(F_{(\Pi,\Theta)})
=\{u\in L^2(\Omega):
u=u_0+K_0\t\gamma_0u\,,\ u_0\in H^1_0(\Omega)\,,\  \t\gamma_0u\in\Lambda\D(f_{\Pi,\Theta})\}\,,
\end{align}
\be
F_{(\Pi,\Theta)}(u,v)=F_D(u_0,v_0)+
f_{\Pi,\Theta}(\Sigma\t\gamma_0u,\Sigma\t\gamma_0v)\,,
\ee
where
\be
\t\gamma_0:H^1_0(\Omega)+\R(K_0)\to H^{-\frac12}(\Gamma )\,,\quad
%\ee
%\be
\t\gamma_0u\equiv\t\gamma_0(u_0+K_0h):=
\gamma_0u_0+\hat\gamma_0K_0h=h\,.
\ee
\end{theorem}
\begin{remark}\label{remarkforma} Notice that, for any $s\in(0,1]$,
\be
\D(F_{(\Pi,\Theta)})\subseteq H^s(\Omega)\quad\iff\quad 
\D(f_{\Pi,\Theta})\subseteq H^s(\Gamma )\,.
\ee
Indeed, 
since $u=u_0+G_0h$ with $u_0\in H^1_0(\Omega)$ and $h\in \D(f_{\Pi,\Theta})$, one has that $u\in H^s(\Omega)$ if and only if 
$G_0 h\in H^s(\Omega)$. By Remark \ref{regularity} 
$G_0 h\in H^s(\Omega)$ if and only if  $h\in H^{s}(\Gamma )$.
\end{remark}
By Theorem \ref{lemmaforma} and Remark \ref{remarkforma} one immediately obtains the following
\begin{corollary}\label{lemmaforma1} If $A_{(\Pi,\Theta)}\in \Ex_{0}(A_{\min})$ and if $\D(F_{(\Pi,\Theta)})\subseteq H^1(\Omega)$ then
\begin{align}
&\D(F_{(\Pi,\Theta)})=\{u\in H^1(\Omega):\gamma_0u\in\D(f_{\Pi_\Lambda,\Theta_{\Sigma}})\}\\
=&\{u\in L^2(\Omega):
u=u_0+K_0\gamma_0u\,,\ u_0\in H^1_0(\Omega)\,,\  \gamma_0u\in\D(f_{\Pi_\Lambda,\Theta_{\Sigma}})\}\,,
\end{align}
\be
F_{(\Pi,\Theta)}(u,v)=F_D(u_0,v_0)+
f_{\Pi_\Lambda,\Theta_{\Sigma}}(\gamma_0u,\gamma_0v)\,.
\ee
Here $\Pi_\Lambda$ denotes the orthogonal projection onto the $L^2(\Gamma)$-closure of 
$\Lambda\D(f_{\Pi,\Theta})$ and $\Theta_{\Sigma}$ is the positive self-adjoint operator in $\R(\Pi_\Lambda)$ associated with the closed, densely defined, positive symmetric bilinear form 
\be
f_{\Pi_\Lambda,\Theta_{\Sigma}}:\D(f_{\Pi_\Lambda,\Theta_{\Sigma}})\times\D(f_{\Pi_\Lambda,\Theta_{\Sigma}})\subseteq \R(\Pi_\Lambda)\times \R(\Pi_\Lambda)\to\RE\,,
\ee
\be
\D(f_{\Pi_\Lambda,\Theta_{\Sigma}}):=\Lambda\D(f_{\Pi,\Theta})\,,\qquad f_{\Pi_\Lambda,\Theta_{\Sigma}}(h_1,h_2):=f_{\Pi,\Theta}(\Sigma h_1,\Sigma h_2)
\ee
\end{corollary} 
Let us now define the bounded linear map 
\be
\hat\Pi_{\Lambda}:H^{-\frac12}(\Gamma)\to H^{-\frac12}(\Gamma)\,,\quad \hat\Pi_{\Lambda}:=\Lambda\Pi\Sigma\,.
\ee
It is immediate to check that $\hat \Pi_{\Lambda}$ is the orthogonal projector in the Hilbert space $H^{-\frac12}(\Gamma)$ such that $\R(\hat \Pi_{\Lambda})=\Lambda\R(\Pi)$. Since $\R(\Pi_{\Lambda})\subseteq\Lambda\R(\Pi)$, we can define the injection
\be
\t\Pi_{\Lambda}:\R(\Pi_{\Lambda})\to H^{-\frac12}(\Gamma)\,,\quad 
\t\Pi_{\Lambda}:=\hat\Pi_{\Lambda}|\R(\Pi_{\Lambda})
\equiv\Lambda\Pi\Sigma|\R(\Pi_{\Lambda})
\ee 
and, by using the duality $(\cdot,\cdot)_{-\frac12,\frac12}$, its dual
\be
\t \Pi_{\Lambda}^{\star}:H^{\frac12}(\Gamma)\to \R(\Pi_{\Lambda})
\ee
By using the same duality we denote by 
\be
K^{\star}_{\lambda}:L^{2}(\Omega)\to H^{\frac12}(\Gamma)
\ee
the dual of the linear operator
\be
K_{\lambda}:H^{-\frac12}(\Gamma)\to L^{2}(\Omega)\,.
\ee
Notice that, by the definition of $G_{\lambda}$ and by the relation $G_{\lambda}\Sigma=  K_{\lambda}$, one has
\begin{equation}\label{K*}
K_{\lambda}^{\star}=\Sigma G_{\lambda}^{*}=\gamma_{1}R^{D}_{\lambda}\,.
\end{equation}
Having introduced these notation, we can state the following result, which provides an alternative Kre\u\i n's formula (of the kind provided in \cite{[BGW]}) for the resolvent of  $A_{(\Pi,\Theta)}$ in the case $\D(F_{(\Pi,\Theta)})  \subseteq H^1(\Omega)$: 
  
\begin{lemma}\label{lemmarisolvente} If $\D(F_{(\Pi,\Theta)})\subseteq H^1(\Omega)$ then  
the resolvent $R_{\lambda}^{(\Pi,\Theta)}$ of $A_{(\Pi,\Theta)}$ in $\Ex_{0}(A_{\min})$ is given by 
\be
R_{\lambda}^{(\Pi,\Theta)}= R_{\lambda}^{D}+K_{\lambda}\t\Pi_{\Lambda}(\Theta_{\Sigma}+\lambda \t\Pi^{\star}_{\Lambda}K^{\star}_{0}K_{\lambda}\t\Pi_{\Lambda})^{-1}\t\Pi^{\star}_{\Lambda}K^{\star}_{\lambda}\,.
\ee
In particular, taking  $\Pi=\uno$, the resolvent $R^{(\Theta)}_{\lambda}$ of $A_{(\Theta)}$
%\in \Ex_{0}(A_{\min})$ with $\D(F_{(\Theta)})\subseteq H^1(\Omega)$ 
is given by 
\be
R_{\lambda}^{(\Theta)}= R_{\lambda}^{D}+K_{\lambda}(\Theta_{\Sigma}+\lambda K^{\star}_{0}K_{\lambda})^{-1}K^{\star}_{\lambda}\,.
\ee
\end{lemma}
\begin{proof} The thesis is consequence of formula \eqref{R}, of the relation $G_{\lambda}\Sigma=K_{\lambda}$ and of the definition of $\hat\Pi_{\Lambda}$, by noticing that, for any 
bounded linear operator $M$ such that $0\in\rho(\Theta+M)$, one has
\be
(\Theta+M)^{-1}=\Sigma (\Theta_{\Sigma}+\Sigma M\Sigma)^{-1}\Sigma\,.
\ee
\end{proof}
Let us now denote by $A_N$ the
self-adjoint extension corresponding to Neumann boundary condition, and by 
$F_N$ the symmetric bilinear form associated  with  $-A_N$, i.e.  
\be
F_N:H^1(\Omega)\times H^1(\Omega)\subseteq L^2(\Omega)\times
L^2(\Omega)\to\RE\,,\quad
F_N(u,v):=\langle \nabla u,a\nabla v\rangle_{L^2(\Omega)}\,.
\ee
Decomposing any $u\in H^1(\Omega)$ 
as $u=u_0+K_0\gamma_0 u$, $u_0\in H^1_0(\Omega)$, one gets, by \eqref{halfGreen}, 
%by Theorem \ref{qf}, Lemma \ref{lemmaforma1} and \eqref{halfGreen}
%\begin{align}\label{DFN}
%&\D(F_N)=H^1(\Omega)\nonumber\\
%=&\{u\in L^2(\Omega)\,:\,u=u_0+G_0h\,,\ u_0\in H^1_0(\Omega)\,,\ h\in H^1(\Gamma)\}\\
%=&\{u\in L^2(\Omega)\,:\,u=u_0+K_0\gamma_0u\,,\ u_0\in H^1_0(\Omega)\}\nonumber\,,
%\end{align}
\begin{align}\label{FNFD}
F_N(u,v)
=F_D(u_0,v_0)
-(P_{0}\gamma_0u,\gamma_0v)_{-\frac12,\frac12}
%\\=&F_D(u_0,v_0)-f_{ P_{0}}(\gamma_0u,\gamma_0v)
\,,
\end{align}
where $P_{0}$ denotes the Dirichlet-to-Neumann operator over $\Gamma $ defined by
\be 
P_{0}:H^{s}(\Gamma )\to H^{s-1}(\Gamma )\,,\quad s\ge-\frac12\,,\quad
P_{0}:=\hat\gamma_1\, K_0\,.
\ee
Notice that (see e.g. \cite{[G68]}, Theorem III 1.1)
\be
\forall s\ge-\frac{1}{2}\,,\quad 
P_{0}\in \B(H^s(\Gamma ),H^{s-1}(\Gamma ))\,.
\ee
Moreover $P_{0}$ is $L^2(\Gamma )$-symmetric 
(by Green's formula) and 
\be
P_{0}:H^1(\Gamma)\subseteq L^2(\Gamma)\to L^2(\Gamma)
\ee
is a negative self-adjoint operator. By Corollary \ref{lemmaforma1}, $A_N=A_{(\uno,\Theta)}$ with 
\be
\Theta=-\Lambda P_{0}\Lambda :H^2(\Gamma)\subseteq L^2(\Gamma)\to L^2(\Gamma)
\ee
and 
\begin{align}
f_{\Pi_\Lambda,\Theta_\Sigma}(h_1,h_2)=-f_{P_{0}}(h_1,h_2)
=
-(P_{0}h_1,h_2)_{-\frac12,\frac12}=F_N(K_0h_1,K_0h_2)
\end{align}
with $\D(f_{P_{0}})=H^\frac12(\Gamma)$.\par
\begin{remark}\label{regular}
It is easy to check that both $F_{D}$ and $F_{N}$ are Dirichlet forms on $L^{2}(\Omega)$, see e.g. \cite{[Fu]}, Examples 1.2.1 and 1.2.3. Both are local and irreducible  (see e.g. next Corollary \ref{irr}),  $F_{D}$ 
is transient since $1\notin \D(F_{D})=H_{0}^{1}(\Omega)$ while $F_{N}$ is recurrent since 
$1\in \D(F_{N})=H^{1}(\Omega)$ and $F_{N}(1)=0$. Moreover $F_{D}$ is regular on $\Omega$, while $F_{N}$ is regular on $\bar\Omega$. The corresponding  diffusions are, in the case $A_{\min}=\Delta|C^{\infty}_{c}(\Omega)$,  the absorbing Brownian motion on $\Omega$ and the reflecting Brownian motion on $\bar\Omega$ respectively (see e.g. Example 3.5.9 in \cite{[CF]}). 
%\par From now on, when speaking about regular Dirichlet forms, we always suppose $X=\bar \Omega$.
\end{remark}
By the following result $A_N$ is a fundamental object  as regards our purposes (for the proof see \cite{[Fu69]}, Theorem 5.1, \cite{[Fu]}, Theorem 2.3.1 and \cite{[FOT]}, Theorem 3.3.1):
\begin{theorem}\label{maximum} $A_N$ is the maximal element of $\ExM(A_\min)$.
\end{theorem}
Since  $A_{N}=A_{(\uno,-\Lambda P_{0}\Lambda)}$, $A_{K}=A_{(\uno,\zero)}$ and 
 $(\uno,-\Lambda P_{0}\Lambda)\prec(\uno,\zero)$, one has  
 \begin{corollary} \label{CK} The Kre\u\i n extension $A_{K}$ of $A_{\min}$ is never Markovian.
\end{corollary}
By the definition of logarithmic Sobolev inequality (see Subsection 2.9) as immediate consequence of Theorem \ref{maximum} one has 
\begin{corollary}\label{ultra}
If $A\in \ExM(A_\min)$ then $F_{A}$ satisfies a logarithmic Sobolev inequality. Hence the semigroup $e^{tA}$ is ultracontractive.
%$$
%\forall t>0\,,\ \forall u\in L^{2}(\Omega)\,,\quad
%\|e^{tA}\|_{L^{\infty}(\Omega)}\le c\,\exp\left(\frac n2\,\frac 1t\int_{0}^{t\wedge 1}\log\frac 1\epsilon\,d\epsilon\right)\,\|u\|_{L^{2}(\Omega)}\,.
%$$  
\end{corollary}
\begin{proof}
By \cite{[Dv]}, Theorem 3.2.9, for the heat kernel $\kappa_{N}$ of $A_{N}$ one has
\be
\kappa_{{N}}(t,x,y)\le c\,\left({t^{-\frac n2}}\vee 1\right)\,.
\ee
Hence, see Subsection 2.9, $F_{N}$ satisfies a logarithmic Sobolev inequality. Since 
$A\preceq A_{N}$, $F_{A}$ satisfies a logarithmic Sobolev inequality (with the same function) and so $e^{tA}$ is ultracontractive.
\end{proof}
Theorem \ref{maximum} also gives  heat kernel estimates for any Markovian extension 
(upper Gaussian and lower bounds on $\kappa_{D}$ and $\kappa_{N}$ can be found in \cite{[Dv]} and references therein):
\begin{corollary}\label{gaussian} If $A\in \ExM(A_\min)$ then 
\be
\kappa_{D}\le \kappa_{A}\le\kappa_{N}\,,
\ee
where $\kappa_{D}$ and $\kappa_{N}$ denote the heat kernels of $A_{D}$ and $A_{N}$ 
respectively. 
\end{corollary} 
\begin{proof} Here we follow the same kind of reasonings as in \cite{[GMN]}. 
By $A_{D}\preceq A\preceq A_{N}$ one gets (see \cite{[GMN]}, Theorem 2.12) that both
$e^{tA}-e^{tA_{D}}$ and $e^{tA_{N}}-e^{tA}$ are positivity preserving. By Corollary \ref{ultra} and by \eqref{trace}, $e^{tA_{D}}$, $e^{tA_{N}}$ and $e^{tA}$ are trace-class operators and hence they are integral operators. The proof is then concluded by noticing that a positive preserving integral operator has a positive kernel (see \cite{[GMN]}, Theorem 2.3).
\end{proof}
Since $\kappa_{D}(t,x,y)>0$  for all $t>0$ and for all $x,y$ in compact subsets of $\Omega$ (see \cite{[Dv]}, Theorem 3.3.5), by Corollary \ref{gaussian} the same is true for $\kappa_{A}$ and so one gets the following 
\begin{corollary}\label{irr}
If $A\in \ExM(A_\min)$ then the semigroup $e^{tA}$ is irreducible. Hence $A$ is either recurrent or 
transient.
\end{corollary} 
\begin{remark}\label{RS} By \cite{[RS]}, Theorem 1.1, Theorem \ref{maximum} hold true on 
arbitrary open bounded set. Since the results we used in Corollaries  \ref{CK}-\ref{irr} also hold under more general hypothesis, these results remains true without the smoothness hypothesis on $\Gamma$. In particular Corollaries \ref{CK} and \ref{irr} remain true in the case of $\Omega$ open and bounded while Corollaries \ref{ultra} and  \ref{gaussian} holds for any open bounded 
$\Omega$ which has the extension property (e.g. $\Omega$ has a Lipschitz boundary, in particular $\Omega$ is convex). For example, by using the upper Gaussian bound for $\kappa_{N}$ 
given in \cite{[Dv]}, Theorem 3.2.9, if $\Omega$ has the extension property  one gets, for any $A\in \ExM(A_{\min})$, the estimate 
(here $1<c_{1}<2$ and the constant $c_{\circ}$ depends on $\Omega$, $c_{1}$ and $\mu_{0}$)
\be
\kappa_{A}(t,x,y)\le c_{\circ}\left(\frac 1{t^{n/2}}\vee 1\right)
\exp\left(-\frac{\|x-y\|^{2}}{4c_{1}\mu_{1}t }\right)\,.
\ee
%(see \cite{[GMN]} for such a kind of estimates for elliptic operators with Robin-type boundary conditions on Lipschitz domains).
\end{remark}
\vskip10pt\noindent
Theorem \ref{maximum} suggests us to introduce the set 
\be
\ExN(A_\min):=
\{A\in \Ex(A_\min)\,:\,A_D\preceq A \preceq A_N\}\,,
\ee
so that 
\be
\ExM(A_\min)\subseteq\ExN(A_\min)\subseteq\ExP(A_\min)\,.
\ee
%and
%$$
%A\in\ExN(A_\min)\Longrightarrow H^1_0(\Omega)\subseteq\D(F_A)\subseteq H^1(\Omega)\,.
%$$
%This suggests to introduce the set of operators 
%$A=\t A_{(\Pi,B)}$ as defined in the next
%Our next aim is the characterization of  $\ExN(A_\min)$. \par
We also introduce a convenient subset of $\E(L^2(\Gamma))$:
\begin{align}
\t\E(L^2(\Gamma)):=
\{(\Pi,B)\in\E(L^2(\Gamma)): 
\D(f_{\Pi,B})\cap H^{\frac12}(\Gamma)\ \text{is a core of $f_{\Pi,B}\ge 0$}\}.
\end{align} 
Then, for any $(\Pi,B)\in\t\E(L^2(\Gamma))$, let us define the 
positive, symmetric, densely defined bilinear  form  $\t F_{(\Pi,B)}$ by
\be
\t F_{(\Pi,B)}:\D(\t F_{(\Pi,B)})\times \D(\t F_{(\Pi,B)})\subseteq
L^2(\Omega)\times L^2(\Omega)\to\RE\,,
\ee
\begin{align}\label{FT}
\t F_{(\Pi,B)}(u,v):=F_N(u,v)+f_{\Pi,B}(\gamma_0u,\gamma_0v)\,,
\end{align}
\be
\D(\t F_{(\Pi,B)}):=\{u\in H^1(\Omega)\,:\,\gamma_0u\in\D( f_{\Pi,B})\}\,.
\ee
\begin{remark}\label{dominioforma}
Notice that 
\be
\t F_{(\Pi,B)}=F_D\iff \D(f_{\Pi,B})\cap H^{\frac12}(\Gamma)=\{0\}
\ee  
and this, by our hypothesis on the core of $f_{\Pi,B}$, 
implies $\Pi=0$. This is consistent with Corollary \ref{lemmaforma1} which says that $F_D$ corresponds to $\Pi=0$. Indeed the core hypothesis  
%(which plays no role in the proof of Lemma \ref{sub}) 
was introduced in order to have 
$\t F_{(\Pi,B)}$ uniquely defined by $f_{\Pi,B}$ and hence by $(\Pi,B)$.
\end{remark}
One has the following 
\begin{theorem}\label{sub} 1. The symmetric bilinear form $\t F_{(\Pi,B)}$ is closed and, denoting by $-\t A_{(\Pi,B)}$ 
the self-adjoint operator associated  with $\t F_{(\Pi,B)}$, one has  
\be
\t A_{(\Pi,B)}\in \ExN(A_\min)\,.
\ee
2. Let $\Pi_\Sigma$ denote the orthogonal projector onto  the $L^2(\Gamma)$-closure of $\Sigma\R(\Pi)$.
Then
\be
f_{\Pi_\Sigma,\Theta_B}:\D(f_{\Pi_\Sigma,\Theta_B})\times \D(f_{\Pi_\Sigma,\Theta_B})
\subseteq \R(\Pi_\Sigma)\times 
\R(\Pi_\Sigma)\to \RE\,,
\ee
\be
\D(f_{\Pi_\Sigma,\Theta_B}):=\Sigma\D( f_{\Pi,B})\cap  
H^{1}(\Gamma)\,,
\ee
\be
f_{\Pi_\Sigma,\Theta_B}(h_1,h_2):=f_{\Pi,B}(\Lambda h_1,\Lambda h_2)-
(P_{0}\Lambda h_1,\Lambda h_2)_{-\frac12,\frac12}
\ee
is a symmetric, closed, densely defined, positive 
bilinear form. Denoting by $\Theta_B$  the positive
self-adjoint operator in $\R(\Pi_\Sigma)$ associated  with  $ f_{\Pi_\Sigma,\Theta_B}$ 
one has
\be
\t A_{(\Pi,B)}=A_{(\Pi_\Sigma,\Theta_B)}\,.
\ee
%where $A_{(\Pi_\Sigma,\Theta_B)}$ is defined by Theorem \ref{estensioni}.
\par\noindent
3. If $f_{\Pi,B}$ is a Dirichlet form on $L^2(\Gamma)$ then 
\be
\t A_{(\Pi,B)}\in \ExM(A_\min)\,.
\ee
Moreover $\t A_{(\Pi,B)}$ is recurrent (equivalently conservative) if and only if  
\be
1\in\D(f_{\Pi,B})\quad \text{and}\quad  f_{\Pi,B}(1)=0 \,.
\ee
\end{theorem}
\begin{proof} 1. Since both $F_N$ and $f_{\Pi,B}$ are positive and $F_N$ is closed,  
\be
\|u_n-u\|_{L^2(\Gamma)}\to 0\quad \text{\rm and}\quad 
\t F_{(\Pi,B)}(u_n-u_m)\to 0
\ee
imply $u\in H^1(\Omega)$ and 
$\|u_n-u\|_{H^1(\Omega)}\to 0$. 
Since $\gamma_0:H^1(\Omega)\to L^{2}(\Gamma)$ is continuous and $\R(\Pi)\subseteq L^{2}(\Gamma)$ is closed, one has 
$\|\gamma_0u_n-\gamma_0u\|_{L^{2}(\Gamma)}\to 0$ and 
$\gamma_0u\in \R(\Pi)$. Thus, by $f_{\Pi,B}(\gamma_0 u_n-\gamma_0 u_m)\to 0$, since $f_{\Pi,B}$ is closed, one gets $\gamma_0u\in\D(f_{\Pi,B})$ 
and $f_{\Pi,B}(\gamma_0u_n-\gamma_0u)\to 0$. Thus $\t F_{(\Pi,B)}$ is closed. \par Since, for any $u_0\in C^\infty_c(\Omega)$ and for any 
$v_0\in H^1_0(\Omega)$, 
\be
\langle A_\min u_0,v_0\rangle_{L^2(\Omega)}=
F_D(u_0,v_0)=\t F_{(\Pi,B)}(u_0,v_0)=\langle \t A_{(\Pi,B)} u_0,v_0\rangle_{L^2(\Omega)}\,,
\ee
one has $\t A_{(\Pi,B)}|C^\infty_c(\Omega)=A_\min$ 
and so $\t A_{(\Pi,B)}\in\Ex(A_\min)$. Since $\t F_{(\Pi,B)}\ge 0$, 
$\t A_{(\Pi,B)}\in\ExP(A_\min)$ and so $A_D\preceq\t A_{(\Pi,B)}$. Then $\t A_{(\Pi,B)} \preceq A_N$ is consequence of the definition of $\t A_{(\Pi,B)}$.
\par\noindent
2. By Theorem \ref{lemmaforma} and Remark \ref{remarkforma}, one has
\begin{align}
&\D(\t F_{(\Pi,B)})=\{u=u_0+G_0h\,,
\ u_0\in H^1_0(\Omega)\,,\ 
h\in \Sigma (\D( f_{\Pi,B})\cap H^{\frac12}(\Gamma))\}\,.
\end{align}
Since $\D(f_{\Pi,B})\cap H^{\frac12}(\Gamma)$ is a core of 
$f_{\Pi,B}$, $\Sigma (\D( f_{\Pi,B})\cap H^{\frac12}(\Gamma))$
is dense in
$\R(\Pi_\Sigma)$. By \eqref{FNFD}, 
\be
\t F_{(\Pi,B)}(u,v)=F_D(u_0,v_0)+
f_{\Pi_\Sigma,\Theta_B}(\Sigma\gamma_0 u,\Sigma\gamma_0 v)
\ee
and hence $f_{\Pi_\Sigma,\Theta_B}(h)=\t F_{(\Pi,B)}(G_0h)\ge 0$. Let 
$\{h_n\}_1^\infty\subset \Sigma (\D( f_{\Pi,B})\cap
H^{\frac12}(\Gamma))$ such that $\|h_n-h\|_{L^2(\Gamma)}\to 0$ and 
$f_{\Pi_\Sigma,\Theta_B}(h_n-h_m)\to 0$. Since $G_0:L^2(\Gamma)\to 
L^2(\Omega)$ is continuous and 
$\t F_{(\Pi,B)}$ is closed, one has that 
$G_0h\in\D(\t F_{(\Pi,B)})$, and $\t F_{(\Pi,B)}(G_0h_n-G_0h)\to 0$. Thus $f_{\Pi_\Sigma,\Theta_B}$ is closed. Finally let us apply 
Theorem \ref{qf}.
\par\noindent
3. By 
\be
\|\,|u|\,\|_{H^1(\Omega)}=\|u\|_{H^1(\Omega)}\,,\qquad\|\,|h|\,\|_{H^{\frac12}(\Gamma)}\le\|h\|_{H^{\frac12}(\Gamma)}\,,
\ee
and
\be
a\vee b=\frac12\,(a+b+|a-b|)\,,\qquad a\wedge b=\frac12\,(a+b-|a-b|)\,,
\ee
both $H^1(\Omega)$ and $H^{\frac12}(\Gamma)$ are Dirichlet spaces. Hence, by 
Theorem \ref{ancona}, 
%for any normal contraction $\Phi$
the maps $u\mapsto u_{\#}$ and $h\mapsto h_{\#}$ are $H^1(\Omega)$- and $H^{\frac12}(\Gamma)$-continuous respectively. Since 
$C^\infty(\bar\Omega)$ is dense in $H^1(\Omega)$,  and $\gamma_0:H^1(\Omega)\to H^\frac12(\Gamma)$ is continuous, taking $\{u_n\}_1^\infty\subset C^\infty(\bar\Omega)$ such that $\|u_n-u\|_{H^1(\Omega)}\to 0$, one has
\begin{equation}\label{gammacanc}
\gamma_0(u_{\#})= \lim_{n\to\infty}\gamma_0((u_n)_{\#})=\lim_{n\to\infty}
(\gamma_0u_n)_{\#}=
(\gamma_0 u)_{\#}\,.
\end{equation}
Let $f_{\Pi,B}$ be a Dirichlet form.  Then, by \eqref{gammacanc},  
\be u\in\D(\t F_{(\Pi,B)})\quad\Longrightarrow\quad 
u_{\#}\in\D(\t F_{(\Pi,B)})\,,
\ee
\begin{align}
\t F_{(\Pi,B)}(u_{\#})=
%&
F_N(u_{\#})+f_{\Pi,B}((\gamma_0 u)_{\#})
%\\
\le 
%&
F_N( u)+f_{\Pi,B}(\gamma_0u)
%\\
=
%&
\t F_{(\Pi,B)}(u)
\end{align}
and so $\t F_{(\Pi,B)}$ is a Dirichlet form. Finally the result about recurrence is an immediate 
consequence of \eqref{conserviff2} and the definition of $\t F_{(\Pi,B)}$.
\end{proof}

%Now let $\t\Pi$ denote the injection 
%$$\t\Pi : \R(\Pi) \to H^{-\frac 12}(\Gamma)\,,\quad \t\Pi h := h$$ 
%and let  $\t\Pi^{\star}$ be the surjection defined by its dual 
%$$\t\Pi^{\star}: H^{\frac 12}(\Gamma) \to \R(\Pi)$$ 
%with respect to the duality $(\cdot,\cdot)_{-\frac 12,\frac 12}$.\par
To state the next result we introduce the family of Dirichlet-to-Neumann operators 
\be
P_{\lambda}:H^{\frac 12}(\Gamma)\to H^{-\frac 12}(\Gamma)\,,
\quad P_{\lambda} :=
\hat\gamma_{1}K_{\lambda}\,,\quad\lambda\ge 0\,.
\ee
By \eqref{KL} and \eqref{K*} one has
\begin{equation}\label{PL}
P_{\lambda} = P_{0} - \lambda K_{0}^{\star}K_{\lambda}=
P_{0} - \lambda \gamma_{1}R^{D}_{0}K_{\lambda}
\,.
\end{equation}
Then, by Lemma \ref{lemmarisolvente}, one obtains a Kre\u\i n's formula for the resolvent of 
$\t A_{(\Pi,B)}$:
\begin{lemma}\label{LRT}
Let $(\Pi,B)\in\t\E(L^{2}(\Gamma))$ and let $B\dotplus  \Pi(- P_{\lambda})\Pi $  be 
defined as a form-sum by the closed, densely defined positive bilinear form
\be
f_{\Pi,B,\lambda}:(\D(f_{\Pi,B})\cap H^{\frac 12}(\Gamma))\times 
(\D(f_{\Pi,B})\cap H^{\frac 12}(\Gamma))\subseteq
\R(\Pi)\times\R(\Pi)\to\RE\,,
\ee
\be
f_{\Pi,B,\lambda}(h_{1},h_{2}):=f_{\Pi,B}(h_{1},h_{2})-(P_{\lambda}h_{1},h_{2})_{-\frac 12,\frac 12}\,.
\ee
Then the resolvent $\t R^{(\Pi,B)}_{\lambda}$ of $\t A_{(\Pi,B)}$ is given  by
\begin{equation}\label{RT}
\t R^{(\Pi,B)}_{\lambda}=R^{D}_{\lambda}+K_{\lambda}\Pi(B\dotplus  \Pi(- P_{\lambda})\Pi )^{-1}\Pi K^{\star}_{\lambda}\,.
\end{equation}
\end{lemma}
\begin{proof} By point 2 in Theorem \ref{sub} we know that $\t A_{(\Pi,B)}=
A_{(\Pi_{\Sigma},\Theta_{B})}$. Since $\R((\Pi_{\Sigma})_{\Lambda})$ is, by definition, 
the $L^{2}(\Gamma)$-closure of $\Lambda(\Sigma\D(f_{\Pi,B})\cap H^{1}(\Gamma))=
\D(f_{\Pi,B})\cap H^{\frac 12}(\Gamma)$ which, by our hypothesis, is a core, one obtains 
$(\Pi_{\Sigma})_{\Lambda}=\Pi$.  Then $\widehat{(\Pi_{\Sigma})}_{\Lambda}$ is the orthogonal projector onto the $H^{-\frac 12}(\Gamma)$-closure of $\R(\Pi)$ and so 
$\widetilde{(\Pi_{\Sigma})}_{\Lambda}=\Pi|\R(\Pi)$, here considered as map from $\R(\Pi)$ into $H^{-\frac 12}(\Gamma)$. Hence $\widetilde{(\Pi_{\Sigma})}_{\Lambda}^{\star}=\Pi|
H^{\frac 12}(\Gamma)$.
Therefore, by Lemma \ref{lemmarisolvente}, noticing that  
the bilinear form associated with $((\Pi_{\Sigma})_{\Lambda},(\Theta_{B})_{\Sigma})$ is given by $f_{\Pi,B,0}$, one gets
\begin{equation}\label{RTT}
\t R_{\lambda}^{(\Pi,B)}= R_{\lambda}^{D}+K_{\lambda}\Pi(B\dotplus \Pi(- P_{0})\Pi+
\lambda \Pi K^{\star}_{0}K_{\lambda}\Pi)^{-1}\Pi K^{\star}_{\lambda}\,.
\end{equation}
The proof is then concluded by \eqref{PL}.
\end{proof}
\begin{remark}\label{regular1} Given $B\equiv(\uno,B)\in\t\E(L^{2}(\Gamma))$, suppose that $f_{B}$ is a regular Dirichlet form with Beurling-Deny decomposition 
\be
f_{B}=f_{B}^{(c)}+f_{B}^{(j)}+f_{B}^{(k)}\,.
\ee
Then the Dirichlet form $\t F_{(B)}\equiv F_{(\uno,B)}$ has the decomposition 
\be
\t F_{(B)}=\t F_{(B)}^{(c)}+\t F_{(B)}^{(j)}+\t F_{(B)}^{(k)}\,,
\ee
where the strongly local component $\t F_{(B)}^{(c)}$ is given by 
\be
\t F_{(B)}^{(c)}(u,v)=F_{N}(u,v)+f_{B}^{(c)}(\gamma_{0}u,\gamma_{0}v)\,,
\ee
and
\begin{align}
%&
\t F_{(B)}^{(j)}(u,v)=f_{B}^{(j)}(\gamma_{0}u,\gamma_{0}v)
%\\
=
%&
\int_{\Gamma\times\Gamma}
(\widetilde{\gamma_{0}u}(x)-\widetilde{\gamma_{0}u}(y))(\widetilde{\gamma_{0}v}(x)-\widetilde{\gamma_{0}v}(y))\, dJ(x,y)\,,
\end{align}
\be
\t F_{(B)}^{(k)}(u,v)=f_{B}^{(k)}(\gamma_{0}u,\gamma_{0}v)=
\int_{\Gamma}
\widetilde{\gamma_{0}u}(x)\widetilde{\gamma_{0}v}(x)\, d\kappa(x)\,.
\ee
Hence the Dirichlet form $\t F_{(B)}$ is strongly local whenever $f_{B}$ is strongly local (i.e. 
$f_{B}^{(j)}=f_{B}^{(k)}=0$) and, 
in the case $\t F_{(B)}$ is regular (see next lemma for a criterion), $\Z_{\t A_{(B)}}$ is a Diffusion whenever $f_{B}^{(j)}=0$. 
\end{remark}

\begin{lemma}\label{Lregular}
Let $B\equiv(\uno,B)\in \t\E(L^2(\Gamma))$ such that either
\be
H^{\frac12}(\Gamma)\subseteq \D(f_{B})\quad\text{and}\quad 
f_{B}(h)\le c\,\|(-\Delta_{LB})^{\frac 14}h\|^{2}_{L^{2}(\Gamma)}+c_{0}\|h\|^{2}_{L^{2}(\Gamma)}
\,,
\ee
or 
\be
 \D(f_{B})\subseteq H^{\frac12}(\Gamma)\quad\text{and}\quad f_{B}(h)\ge c\,
 \|(-\Delta_{LB})^{\frac 14}h\|^{2}_{L^{2}(\Gamma)}
\,,
\ee
where $c>0$, $c_{0}\ge 0$. Then 
\be
\text{$f_{B}$ regular on $\Gamma$ $\quad\Longrightarrow\quad$ $\t F_{(B)}$ regular on $\bar\Omega$.}
\ee
\end{lemma}
\begin{proof}  Let us at first show that 
\begin{align}
%&
\D(\t F_{(B)})\cap C(\bar\Omega)
%\\
\equiv
%&
\{u=u_{0}+K_{0}h\,,\ u_{0}\in H^{1}_{0}(\Omega)\cap C_{0}(\Omega)\,,\ h\in\D(f_{B})\cap C(\Gamma)\}
\end{align}
is $L^{\infty}(\Omega)$-dense in $C(\bar\Omega)$. By Stone-Weirerstrass theorem this is equivalent to show that $\D(\t F_{(B)})\cap C(\bar\Omega)$ separates the points of $\bar\Omega$ (see \cite{[CF]}, Remark 1.3.11). If $x,y\in\Omega$ then one takes $h=0$ and $u_{0}\in C^{\infty}_{0}(\Omega)$ such that $u_{0}(x)\not=u_{0}(y)$. If $x,y\in\Gamma$ then, since $f_{B}$ is regular, there exists $h\in \D(f_{B})\cap C(\Gamma)$ such that $h(x)\not=h(y)$  and so $u(x)\not=u(y)$ by posing $u=K_{0}h$.  Suppose now that $x\in \Omega$ and $y\in\Gamma$. Then, given $h\in \D(f_{B})\cap C(\Gamma)$, it suffices to take $u=u_{0}+K_{0}h$, where $u_{0}\in C^{\infty}_{0}(\Omega)$ is such that $u_{0}(x)\not=h(y)-K_{0}h(x)$.\par
If $H^{\frac12}(\Gamma)\subseteq \D(f_{B})$  then $\D(\t F_{(B)})=H^{1}(\Omega)$. Taking a sequence $\{u_{n}\}_1^{\infty}\subset  C^{\infty}(\bar\Omega)$ converging in $H^{1}(\Omega)$ to 
$u$,  since  $\gamma_{0}u_{n}$ converges to $\gamma_{0}u$ in $H^{\frac12}(\Gamma)$ 
and  
\begin{align}
(\t F_{(B)}+1)(u_n-u)
\le (F_{N}+1)(u_{n}-u)+(c\vee c_{0})\,\|\gamma_{0}u_{n}-\gamma_{0}u\|^{2}_{H^{\frac12}(\Gamma)}\,,
\end{align} 
one gets $(\t F_{(B)}+1)(u_n-u)\to 0$.\par
Since $f_{B}$ is regular, for any $h\in \D(f_{B})$ there exists a sequence $\{h_{n}\}_{1}^{\infty}\subset \D(f_{B})\cap C(\Gamma)$ such that $(f_{B}+1)(h_{n}-h)\to 0$. If 
$f_{B}(h)\ge c\,\|(-\Delta_{LB})^{\frac14}h\|^{2}_{L^{2}(\Gamma)}$, then $\|h_{n}-h\|_{H^{\frac 12}(\Gamma)}\to 0$ and so $K_{0}h_{n}$ converges to $K_{0}h$ in $H^{1}(\Omega)$. Given $u=u_{0}+K_{0}h\in \D(\t F_{(B)})$, let $\{u_{0,n}\}_1^{\infty}\subset  C^{\infty}_{0}(\Omega)$ converge in $H^{1}(\Omega)$ to 
$u_{0}$. Then $u_{n}=u_{0,n}+K_{0}h_{n}\in \D(\t F_{(B)})\cap C(\bar\Omega)$,  converges in $H^{1}(\Omega)$ to 
$u$ and so $(\t F_{(B)}+1)(u_n-u)=(F_{N}+1)(u_{n}-u)+f_{B}(h_{n}-h)\to 0$.
\end{proof}
\begin{remark} A resolvent formula for $A_{(B)}$, where $R^{D}_{\lambda}$ is substituted by $R^{N}_{\lambda}:=(-A_{N}+\lambda)^{-1}$, can be given under the hypothesis $\D(A_{N})\subseteq \D(\t F_{(B)})$ (for example this holds when 
$H^{\frac12}(\Gamma)\subseteq \D(f_{B})$, so that $\D(\t F_{(B)})=H^{1}(\Omega)=\D(F_{N})$): defining the closed, densely defined  operator in $H^{1}(\Omega)$
\be
J_{B}:\D(\t F_{(B)})\subseteq H^{1}(\Omega)\to L^{2}(\Gamma)\,,\quad J_{B}u:=\sqrt B\,\gamma_{0}u\,.
\ee
one has 
\be
\t F_{(B)}(u,v)=F_{N}(u,v)+\langle J_{B}u,J_{B}v\rangle_{L^{2}(\Gamma)}
\ee
and so, by Lemma 3 in \cite{[Br]}, one gets
\be
\t R^{(B)}_{\lambda}=R^{N}_{\lambda}+(J_{B}R^{N}_{\lambda})^{*}
(\uno+J_{B}J_{B}^{*})^{-1}
J_{B}R^{N}_{\lambda}\,.
\ee
\end{remark}
Before stating the converse of Theorem \ref{sub} we give the following result concerning 
the Dirichlet-to-Neumann operator:  
\begin{lemma}\label{Plim}
For all $h\in L^{\infty}(\Gamma)\cap H^{\frac12}(\Gamma)$ one has
\begin{equation}
(P_{0}h,h)_{-\frac12,\frac12}=\lim_{\alpha\uparrow\infty}\alpha\left(\langle K_{0}h,
K_{\alpha}h\rangle_{L^2(\Omega)}-\langle{1},K_{\alpha}h^{2}\rangle_{L^2(\Omega)}\,\right)
\end{equation}
\end{lemma}
\begin{proof}
By Theorem 2 in \cite{[Fu64]} (also use Theorem 5.5.9 in \cite{[CF]}),
\be
F_{N}(K_{0}h)=\frac12\,\lim_{\alpha\uparrow\infty}\int_{\Gamma\times\Gamma}
(h(x)-h(y))^{2}U_{\alpha}(x,y)\,d\sigma(x)d\sigma(y)
\,,
\ee
where the kernel $U_{\alpha}$ is defined by 
\be
U_{\alpha}(x,y):=\alpha\int_{\Omega}\frac{\partial g_{\alpha}}{\partial \nu_{a}}(x,z)
\,\frac{\partial g_{0}}{\partial \nu_{a}}(z,y)\,dz\,,
\ee
and $g_{\alpha}$ denotes the kernel of $(-A_{D}+\alpha)^{-1}$. The proof is then
concluded by $F_{N}(K_{0}h)=-(P_{0}h,h)_{-\frac12,\frac12}$, by $K_{0}1=1$ and by
\begin{align}
&\alpha\langle K_{0}h,
K_{\alpha}h\rangle_{L^2(\Omega)}=\int_{\Gamma\times\Gamma}
h(x)h(y)U_{\alpha}(x,y)\,d\sigma(x)d\sigma(y)\\
=&\alpha\langle {1},
K_{\alpha}h^{2}\rangle_{L^2(\Omega)}
-\frac12
\int_{\Gamma\times\Gamma}
(h(x)-h(y))^{2}U_{\alpha}(x,y)\,d\sigma(x)d\sigma(y)\,.
\end{align}

\end{proof}
Theorem \ref{sub} has the following converse:
\begin{theorem}\label{converse} Let 
$A_{(\Pi,\Theta)}\in\ExN(A_\min)$ 
%and let  $\Pi_\Lambda$ be the orthogonal projector 
%onto the $L^2(\Gamma)$-closure of $\Lambda\D(f_{\Pi,\Theta})$
. Then:
\par\noindent
1. \be
\D(F_{(\Pi,\Theta)})
=\{u\in H^1(\Omega)\,:\,\gamma_0u\in\D(f_{\Pi_{\Lambda},\Theta_{\Sigma}}) \}\,,
\ee
%If $A_{(\Pi,\Theta)}\in\ExN(A_\min)$ then 
\be 
F_{(\Pi,\Theta)}(u,v)=F_N(u,v)+\t f^\Theta_{\Pi_\Lambda}(\gamma_0u,\gamma_0v)\,,
\ee
where the positive symmetric bilinear form $\t f^\Theta_{\Pi_\Lambda}$ is defined by 
\be
\t f^\Theta_{\Pi_\Lambda}:\D(f_{\Pi_\Lambda,\Theta_\Sigma})\times
\D(f_{\Pi_\Lambda,\Theta_\Sigma})\subseteq \R(\Pi_\Lambda)\times \R(\Pi_\Lambda)\to\RE\,,
\ee
\begin{align}
&\t f^\Theta_{\Pi_\Lambda}(h_1,h_2):= f_{\Pi_\Lambda,\Theta_\Sigma}(h_1, h_2)
+(P_{0}h_1,h_2)_{-\frac12,\frac12}\,.
\end{align}
2. If $\t f^\Theta_{\Pi_\Lambda}$ is closable with closure $f_{\Pi_\Lambda,B_\Theta}$ then 
$(\Pi_\Lambda,B_\Theta)\in\t\E(L^2(\Gamma))$ and 
%\be
$A_{(\Pi,\Theta)}=\t A_{\Pi_\Lambda,B_\Theta}$,
%\ee
i.e.
\be
F_{(\Pi,\Theta)}(u,v)=F_{N}(u,v)+f_{\Pi_{\Lambda},B_{\Theta}}(\gamma_{0}u,\gamma_{0}v)
\ee
\par\noindent
3. If $A_{(\Pi,\Theta)}\in\ExM(A_\min)$ then: 
%\par\noindent
\item
3.1. $f_{\Pi_\Lambda,\Theta_\Sigma}$ is a Dirichlet form on $L^2(\Gamma)$;
%\par\noindent
\item 3.2. $\t f^\Theta_{\Pi_\Lambda}$ is a Markovian form on $L^2(\Gamma)$;
%\par\noindent 
\item3.3 $f_{\Pi_\Lambda,B_\Theta}$ is a 
Dirichlet form on $L^2(\Gamma)$ whenever $\t f^\Theta_{\Pi_\Lambda}$ is closable.
\par\noindent
\item Moreover $A_{(\Pi,\Theta)}$ is recurrent (equivalently conservative) if and only if  
\be
1\in\D(\t f^\Theta_{\Pi_\Lambda})\quad \text{and}\quad \t f^\Theta_{\Pi_\Lambda}(1)=
0\,.
\ee
\end{theorem}
\begin{proof} 1. By Corollary \ref{lemmaforma1}, by Remark \ref{remarkforma}, 
and by \eqref{FNFD}, if $A_{(\Pi,\Theta)}\in \ExN(A_\min)$ then
\begin{align}
&\D(F_{(\Pi,\Theta)})
=\{u\in H^1(\Omega)\,:\,\gamma_0u\in\Lambda\D(f_{\Pi,\Theta}) \}
\\
=
&
\{u=u_0+K_0 h\,,u_0\in H_0^1(\Omega)\,,\ h\in\Lambda\D(f_{\Pi,\Theta})\}
\,,
\end{align}
\be
F_{(\Pi,\Theta)}(u,v)=F_N(u,v)+\t f^\Theta_{\Pi_\Lambda}(\gamma_0 u,\gamma_0 v)\,.
\ee
Even if $P_{0}$ is negative, $\t f^\Theta_{\Pi_\Lambda}$ is positive: since  $A_{(\Pi,\Theta)}\preceq A_N$, one has
\be
\t f^\Theta_{\Pi_\Lambda}(h)=
F_{(\Pi,\Theta)}(K_0h)-F_N(K_0h)
\ge 0\,.
\ee
2. Since $A_{(\Pi,\Theta)}\preceq A_N$, by Remark \ref{remarkforma} $\Lambda\D(f_{\Pi,\Theta})\subseteq H^\frac12(\Gamma)$. Hence $\D(f_{\Pi_\Lambda,B_\Theta})\cap H^\frac12(\Gamma)$ contains $\D(f^\Theta_{\Pi_\Lambda})$ and so is a core of $f_{\Pi_\Lambda,B_\Theta}$. Thus  $(\Pi_\Lambda,B_\Theta)\in\t\E(L^2(\Gamma))$ 
and $F_{(\Pi,\Theta)}=\t F_{(\Pi_\Lambda,B_\Theta)}$.\par\noindent
3.1. By Corollary \ref{lemmaforma1}, 
$f_{\Pi_\Lambda,\Theta_\Sigma}$ is a closed, positive bilinear form on $\R(\Pi_\Lambda)$ such that 
\be
f_{\Pi_\Lambda,\Theta_\Sigma}(h_1,h_2)=F_{(\Pi,\Theta)}(K_0h_1,K_0h_2)
\ee
for any $h_1,h_2\in\D(f_{\Pi_\Lambda,\Theta_\Sigma})$. 
For any $\alpha\ge 0$, let us define the closed, positive bilinear form on $\R(\Pi_\Lambda)$ 
with domain $\D(f_{\Pi_\Lambda,\Theta_\Sigma})$
\begin{equation}\label{falpha}
f^{(\alpha)}_{\Pi_\Lambda,\Theta_\Sigma}(h_1,h_2):=
f_{\Pi_\Lambda,\Theta_\Sigma}(h_1,h_2)+\langle M_\alpha\Sigma h_1,\Sigma h_2\rangle_{L^2(\Gamma)}\,.
\end{equation}
% where $M_\alpha:=\alpha G_\alpha^*G_0=\alpha G_0^*G_\alpha$.
By Remark \ref{traslazioneforma} 
\begin{equation}
f^{(\alpha)}_{\Pi_\Lambda,\Theta_\Sigma}(h)=
\left(F_{(\Pi,\Theta)}+\alpha\right)(K_\alpha h)\,,\qquad K_\alpha:=G_\alpha\Sigma\,.
\end{equation}
Given $h\in \D(f_{\Pi_\Lambda,\Theta_\Sigma})$, suppose that  $h_{\#}\notin\D(f_{\Pi_\Lambda,\Theta_\Sigma})$. 
%for some normal contraction $\t\Phi$
By \eqref{gammacanc} $\gamma_0(K_0h)_{\#}=(\gamma_0K_0h)_{\#}=h_{\#}$ and so, by Corollary \ref{lemmaforma1}, 
$(K_0h)_{\#}\notin\D(F_{(\Pi,\Theta)})$ which is impossible since $K_0h\in\D(F_{(\Pi,\Theta)})$ and $F_{(\Pi,\Theta)}$ is a Dirichlet form by hypothesis. 
Thus 
%for any normal contraction $\Phi$ one has
\be
h\in\D(f_{\Pi_\Lambda,\Theta_\Sigma})\quad\Longrightarrow\quad h_{\#}\in
\D(f_{\Pi_\Lambda,\Theta_\Sigma})\,.
\ee
Moreover, since $F_{(\Pi,\Theta)}+\alpha$ is a Dirichlet form and 
$u_\alpha:=(K_\alpha h)_{\#}-K_\alpha h_{\#}\in H^1_0(\Omega)$, by Corollary \ref{lemmaforma1} and Remark \ref{traslazioneforma} one gets
\begin{align}
&f^{(\alpha)}_{\Pi_\Lambda,\Theta_\Sigma}(h_{\#})\le \left(F_D+\alpha\right)(u_\alpha)+f^{(\alpha)}_{\Pi_\Lambda,\Theta_\Sigma}(h_{\#})
\\=&
\left(F_{(\Pi,\Theta)}+\alpha\right)((K_\alpha h)_{\#})
\le \left(F_{(\Pi,\Theta)}+\alpha\right)(K_\alpha h)\\
=&f^{(\alpha)}_{\Pi_\Lambda,\Theta_\Sigma}(h)\,.
\end{align}
Hence $f^{(\alpha)}_{\Pi_\Lambda,\Theta_\Sigma}$ is a Dirichlet form for any $\alpha\ge 0$.\par\noindent
3.2. Denoting by $f^{\lambda}_{\Pi_\Lambda,\Theta_\Sigma}$ and 
$f^{(\alpha),\lambda}_{\Pi_\Lambda,\Theta_\Sigma}$ the Yosida approximations 
of $f_{\Pi_\Lambda,\Theta_\Sigma}$ and $f^{(\alpha)}_{\Pi_\Lambda,\Theta_\Sigma}$ 
respectively, and posing $h_{n}:=((-n)\vee h)\wedge n$,  
by Theorem \ref{Res-Form}, by \eqref{falpha} and by Lemma \ref{Plim} one has (here we follow the same strategy as in the proof of point (i) in \cite{[Fu69]}, Lemma 5.4)
\begin{align}
&f_{\Pi_\Lambda,\Theta_\Sigma}(h)\\
=&\lim_{n\uparrow\infty}\lim_{\alpha\uparrow\infty}\lim_{\lambda\uparrow\infty}
\left(
\breve f^{(\alpha),\lambda}_{\Pi_\Lambda,\Theta_\Sigma}(h_{n})
+\check f^{(\alpha),\lambda}_{\Pi_\Lambda,\Theta_\Sigma}(h_{n})
-\langle M_\alpha\Sigma h_n,\Sigma h_n\rangle_{L^2(\Gamma)}\right)\\
=&\lim_{n\uparrow\infty}\lim_{\alpha\uparrow\infty}\lim_{\lambda\uparrow\infty}
\left(
\breve f^{(\alpha),\lambda}_{\Pi_\Lambda,\Theta_\Sigma}(h_{n})
+\check f^{\lambda}_{\Pi_\Lambda,\Theta_\Sigma}(h_{n})\right)\\
+&\lim_{n\uparrow\infty}\lim_{\alpha\uparrow\infty}\left(\langle M_\alpha\Sigma 1,\Sigma h^2_n\rangle_{L^2(\Gamma)}
-\langle M_\alpha\Sigma h_n,\Sigma h_n\rangle_{L^2(\Gamma)}\,\right)\\
=&\lim_{n\uparrow\infty}\lim_{\alpha\uparrow\infty}\lim_{\lambda\uparrow\infty}
\left(
\breve f^{(\alpha),\lambda}_{\Pi_\Lambda,\Theta_\Sigma}(h_{n})
+\check f^{\lambda}_{\Pi_\Lambda,\Theta_\Sigma}(h_{n})\right)\\
+&\lim_{n\uparrow\infty}\lim_{\alpha\uparrow\infty}
\left(\langle 1,K_{\alpha} h^2_n\rangle_{L^2(\Omega)}
 -\langle K_{0} h_n,K_{\alpha}h_n\rangle_{L^2(\Omega)}\,\right)\\
=&\lim_{n\uparrow\infty}\lim_{\alpha\uparrow\infty}\lim_{\lambda\uparrow\infty}
\left(
\breve f^{(\alpha),\lambda}_{\Pi_\Lambda,\Theta_\Sigma}(h_{n})
+\check f^{\lambda}_{\Pi_\Lambda,\Theta_\Sigma}(h_{n})\right)-(P_{0}h,h)_{-\frac12,\frac12}\,.
\end{align}
Hence
\be
\t f_{\Pi_\Lambda}^\Theta(h)=\lim_{n\uparrow\infty}\lim_{\alpha\uparrow\infty}\lim_{\lambda\uparrow\infty}
\left(
\breve f^{(\alpha),\lambda}_{\Pi_\Lambda,\Theta_\Sigma}(h_{n})
+\check f^{\lambda}_{\Pi_\Lambda,\Theta_\Sigma}(h_{n})\right)
\ee
and so $\t f_{\Pi_\Lambda}^\Theta$ is a Markovian form since both $\breve f^{(\alpha),\lambda}_{\Pi_\Lambda,\Theta_\Sigma}$ and 
$\check f^{\lambda}_{\Pi_\Lambda,\Theta_\Sigma}$ are Markovian forms.
\par\noindent
3.3. If $\t f_{\Pi_\Lambda}^\Theta$ is closable then $f_{\Pi_\Lambda,B_\Theta}$ is a Dirichlet form by Theorem \ref{closure}.  \par
The result on recurrence is an immediate consequence of 
\eqref{conserviff2}.
\end{proof}
\begin{corollary}\label{tilde} $\t F_{(\Pi,B)}$ in \eqref{FT} is a Dirichlet form in $L^{2}(\Omega)$ if and only if
  $f_{\Pi,B}$ is a Dirichlet form in $L^{2}(\Gamma)$. Equivalently $\t R_{\lambda}^{(\Pi,B)}$ in \eqref{RT} is a Markovian resolvent in $L^{2}(\Omega)$, i.e. is the resolvent of a Markovian self-adjoint extension of $A_{\min}$, if and only if  $f_{\Pi,B}$ is a Dirichlet form in $L^{2}(\Gamma)$.  \end{corollary}
By combining Theorem \ref{sub} with Theorem \ref{converse} one gets the analogue of 
Theorem \ref{toy}:
\begin{theorem}\label{finale} Let $A\in  \Ex(A_{\min})$. Then $A$ belongs to $\ExM(A_{\min})$ if and only if $\D(F_{A})\subseteq H^{1}(\Omega)$ and the bilinear form 
\be
f_{[A]}:(\gamma_{0}\D(F_{A}))\times (\gamma_{0}\D(F_{A}))\subseteq L^{2}(\Gamma)\times L^{2}(\Gamma)\to\RE\,,
\ee
\be
f_{[A]}(h_{1},h_{2}):=F_{A}(K_{0}h_{1},K_{0}h_{2})
\ee
is a Dirichlet form which admits the decomposition
\be
f_{[A]}(h_1,h_2)=
f_b(h_1,h_2)-(P_{0}h_1,h_2)_{-\frac12,\frac12}\,,
\ee
where 
\be
f_b:(\gamma_{0}\D(F_{A}))\times (\gamma_{0}\D(F_{A}))\subseteq L^{2}(\Gamma)\times L^{2}(\Gamma)\to\RE
\ee
is a Markovian form. $A$ is recurrent (equivalently conservative) if and only if  
\be
1\in  \gamma_{0}\D(F_{A})\quad\text{and}\quad  f_b(1)\equiv f_{[A]}(1)=
0\,.
\ee
\end{theorem}
\begin{proof} By Theorem \ref{estensioni}, we know that   $A=A_{(\Pi,\Theta)}$ for some $(\Pi,\Theta)$ belonging to $\E(L^{2}(\Gamma))$. \par 
If $A_{(\Pi,\Theta)}\in \ExM(A_\min)$ then $\D(F_{A})\subseteq H^{1}(\Omega)$ by Theorem \ref{maximum} and  then  $f_{[A]}(h_{1},h_{2})=f_{\Pi_{\Lambda},\Theta_{\Sigma}}(h_{1},h_{2})$ by Corollary \ref{lemmaforma1}. The thesis is then a consequence of Theorem \ref{converse} by posing $f_b=\t f^\Theta_{\Pi_\Lambda}$. \par
Conversely, if  $\D(F_{A})\subseteq H^{1}(\Omega)$ then,  by Corollary \ref{lemmaforma1}, 
one gets  $f_{[A]}(h_{1},h_{2})=f_{\Pi_{\Lambda},\Theta_{\Sigma}}(h_{1},h_{2})$ and, proceeding as in the proof of point 1 of Theorem \ref{converse}, by $\D(f_{\Pi,\Theta})\subseteq  H^1(\Gamma) $, by Corollary \ref{lemmaforma1}, by Remark \ref{remarkforma}, 
and by \eqref{FNFD} one obtains
\be
F_{(\Pi,\Theta)}(u,v)=F_N(u,v)+\t f^\Theta_{\Pi_\Lambda}(\gamma_0 u,\gamma_0 v)\,.
\ee
By posing again $f_b=\t f^\Theta_{\Pi_\Lambda}$ and supposing 
$f_b$ is Markovian, one gets, as in the proof of point 3 in Theorem \ref{sub},
\begin{align}
F_{(\Pi,\Theta)}(u_{\#})=
%&
F_N(u_{\#})+
f_b((\gamma_0 u)_{\#})
%\\
\le 
%&
F_N( u)+f_b(\gamma_0u)
%\\
=
%&
F_{(\Pi,\Theta)}(u)\,.
\end{align}
Thus $F_{(\Pi,\Theta)}$ is a Dirichlet form and $A_{(\Pi,\Theta)}\in \ExM(A_\min)$.
\end{proof}
Since in the proof of Theorem \ref{sub} the hypothesis requiring $f_{\Pi,B}$ closed was used there only 
to show that $\t F_{(\Pi,B)}$ is closed, Theorem \ref{finale} can be re-phrased in the following form:
\begin{theorem}\label{finale2}  
Let $F$ be a closed bilinear form on $L^{2}(\Omega)$. Then 
$F=F_{A}$ with  $A\in\ExM(A_{\min})$ if and only if $\D(F)\subseteq H^{1}(\Omega)$ and there exists a Markovian form 
\be 
f_{b}:\D(f_{b})\times\D(f_{b})\subseteq  L^{2}(\Gamma)\times L^{2}(\Gamma)\to\RE
\ee 
such that 
\be
\D(F)=\{u\in H^{1}(\Omega):\gamma_{0}u\in\D(f_{b})\}
\ee
 and
\be
 F(u,v)=F_{N}(u,v)+f_{b}(\gamma_{0}u,\gamma_{0}u)\,.
 \ee
$A$ is recurrent (equivalently conservative) if and only if  
\be
1\in  \D(f_{b})\quad\text{and}\quad  f_b(1)=
0\,.
\ee
\end{theorem}
\begin{remark} If there exists $u\in \D(F_{A})$ such that $F_{A}(u)=F_{N}(u)+f_{b}(\gamma_{0}u)=0$ then $u=c=\text{const.}$ and $f_{b}(c)$=0, so that $A$ is recurrent. Therefore
\be
\text{$A\in \ExM(A_\min)$ is transient $ \iff$ $A<0$.}
\ee 
\end{remark}
Next we give an equivalent version of Theorem \ref{finale} in terms of resolvents. \par A family $R_{\lambda}\in \B(L^{2}(\Omega))$, $\lambda>0$, of bounded symmetric linear operators is said to be a {\it resolvent family} in $L^{2}(\Omega)$ if it satisfies the resolvent identity
\be
R_{\lambda}-R_{\mu}=(\mu-\lambda)R_{\lambda}R_{\mu}
\ee
and the bounds
\be
%\forall u\in L^{2}(\Omega)\,,\qquad 
\|R_{\lambda}u\|_{L^{2}(\Omega)}\le \frac 1\lambda\,\|u\|_{L^{2}(\Omega)}\,.
\ee
\begin{lemma}\label{LR} 
If $R_{\lambda}$, $\lambda>0$, is a resolvent family, then  
the bilinear form
\be
\D(f_{[R]}):=\{h\in L^{2}(\Gamma):  \lim_{\lambda \uparrow\infty}\,
\langle G_{0}h,\lambda(\uno-\lambda R_{\lambda})G_{0}h\rangle_{L^{2}(\Omega)}<\infty\}\,,
\ee
\be
f_{[R]}(h_{1},h_{2}):=\lim_{\lambda\uparrow\infty}\,\langle G_{0}h_{1},\lambda(\uno-\lambda R_{\lambda})G_{0}h_{2}\rangle_{L^{2}(\Omega)}\,,
\ee
is symmetric, positive and  closed.
\end{lemma}
\begin{proof}
By \cite{[FOT]}, Theorem 1.3.2, $\lambda\mapsto \langle u,\lambda(\uno-\lambda R_{\lambda})u\rangle_{L^{2}(\Omega)}$ is not decreasing and 
the bilinear form
\be
\D(F_{[R]}):=\{u\in L^{2}(\Omega): \lim_{\lambda \uparrow\infty}\,
\langle u,\lambda(\uno-\lambda R_{\lambda})u\rangle_{L^{2}(\Omega)}<\infty\}\,,
\ee
\be
F_{[R]}(u,v):=\lim_{\lambda\uparrow\infty}\,\langle u,\lambda(\uno-\lambda R_{\lambda})v\rangle_{L^{2}(\Omega)}\,,
\ee
is positive, symmetric and  closed. Hence 
\be
\D(f_{[R]})=\{h\in L^{2}(\Gamma):G_{0}h\in\D(F_{[R]})\}
\ee
and 
\be
f_{[R]}(h_{1},h_{2})=F_{[R]}(G_{0}h_{1},G_{0}h_{2})\,.
\ee
The thesis then follows by $G_{0}\in\B(L^{2}(\Omega),L^{2}(\Gamma))$.
\end{proof}
\begin{theorem}\label{finaleres} The resolvent family $R_{\lambda}\in \B(L^{2}(\Omega))$, $\lambda>0$, is the resolvent of  $A\in\ExM(A_{\min})$ if and only if the following 
conditions hold: 
\par\noindent
1. $\D(f_{[R]})\subseteq H^{1}(\Gamma)$ and the bilinear form
\be
\D(\t f_{[R]}):=\Lambda \D(f_{[R]})\,,\quad \t f_{[R]}(h_{1},h_{2}):=f_{[R]}(\Sigma h_{1},\Sigma{h_{2}})
\ee
is a Dirichlet form in $L^{2}(\Gamma)$ which admits the decomposition
\be
\t f_{[R]}(h_1,h_2)=f_b(h_1,h_2)-(P_{0}h_1,h_2)_{-\frac12,\frac12}\,,
\ee
where $f_{b}$, $\D(f_{b})\equiv\D(\t f_{[R]})$, is a Markovian form in $L^{2}(\Gamma)$;
\par\noindent
2. $R_{\lambda}$ admits the decomposition
\begin{equation}\label{RTTT}
R_{\lambda}= R_{\lambda}^{D}+G_{\lambda}\Pi_{[R]}(\Theta_{[R]}+
\lambda \Pi_{[R]} G^{*}_{0}G_{\lambda}\Pi_{[R]})^{-1}\Pi_{[R]} G^{*}_{\lambda}\,,
\end{equation}
where $\Pi_{[R]}$ denotes the orthogonal projector onto the $L^{2}(\Gamma)$-closure of $\D(f_{[R]})$ and $\Theta_{[R]}$ denotes the positive self-adjoint operator in $\R(\Pi_{[R]})$ associated to $f_{[R]}$. 
\end{theorem} 
\begin{proof} Suppose that  $R_{\lambda}=R^{A}_{\lambda}$ with 
$A\in\ExM(A_{\min})$. Then, by Theorem \ref{estensioni}, Theorem \ref{maximum} and Remark \ref{remarkforma}, $A=A_{(\Pi,\Theta)}$ for some $(\Pi,\Theta)\in \E(L^{2}(\Gamma))$ with $\D(f_{(\Pi,\Theta)})\subseteq H^{1}(\Gamma)$. Then, by Theorem \ref{qf} and Theorem \ref{Res-Form}, $f_{[R]}=f_{(\Pi,\Theta)}$. By \eqref{R}, $R_{\lambda}=R^{(\Pi,\Theta)}_{\lambda}$ and \eqref{RTTT} holds with $\Theta_{[R]}=\Theta$ and $\Pi_{[R]}=\Pi$. 
By Corollary \ref{lemmaforma1}, $\t f_{[R]}=f_{\Pi_{\Lambda},\Theta_{\Sigma}}$ and the decomposition $\t f_{[R]}(h_1,h_2)=f_b(h_1,h_2)-(P_{0}h_1,h_2)_{-\frac12,\frac12}$ follows by 
Theorem \ref{converse}.
\par
Conversely suppose that 1 and 2 hold true. Then by \eqref{RTTT} and by Theorem \ref{estensioni}, $R_{\lambda}$ is the resolvent of $A_{(\Pi_{[R]},\Theta_{[R]})}\in \Ex(A_{\min})$. By $\D(f_{[R]})\subseteq H^{1}(\Gamma)$, by Remark \ref{remarkforma} 
and by Corollary \ref{lemmaforma1}, one obtains
\be
F_{(\Pi_{[R]},\Theta_{[R]})}(u,v)=F_{D}(u_{0},v_{0})+
\t f_{[R]}(\gamma_{0}u,\gamma_{0}v)\,.
\ee
Thus 
\be
F_{(\Pi_{[R]},\Theta_{[R]})}(u,v)=F_{N}(u,v)+
f_{b}(\gamma_{0}u,\gamma_{0}v)
\ee
and so  $A_{(\Pi_{[R]},\Theta_{[R]})}\in \ExM(A_{\min})$ by  Theorem \ref{finale2}.
\end{proof}
%Since $\t A_{(\Pi,B)}\in \ExN(A_\min)$, by Theorem 
%\ref{estensioni}
%and Remark 
%\ref{remarkforma} there exist 
%$(\Pi_\Sigma,\Theta_B)\in \E(L^2(\Gamma))$, $\Theta_B\ge 0$, 
%$\D(f_{\Pi_\Sigma,\Theta_B})\subseteq H^1(\Gamma)$, such that 
%$\t A_{(\Pi,B)}=A_{(\Pi_\Sigma,\Theta_B)}$. This provides a Kre\u\i n's resolvent formula for  $\t A_{(\Pi,B)}$. Next lemma makes explicit the relation 
%between the bilinear  forms $f_{\Pi,B}$ and 
%$f_{\Pi_\Sigma,\Theta_B}$ (compare with \eqref{thetaB}):
\begin{remark}\label{4.25}
Suppose that $R_{\lambda}=R^{A}_{\lambda}$ with $A\in\ExM(A_{\min})$. 
Then, by Lemma \ref{lemmarisolvente}, \eqref{RTTT} can be rewritten as
\be
R_{\lambda}=R^{D}_{\lambda}+K_{\lambda}\t\Pi_{\Lambda}(\t\Theta_{[R]}+\lambda\t\Pi_{\Lambda}K_0^{\star}K_{\lambda}\t\Pi_{\Lambda})^{-1}\t\Pi_{\Lambda}K_{\lambda}^{\star}\,,
\ee
where $\t\Theta_{[R]}$ is the Markovian operator associated with the Dirichlet form $\t f_{[R]}$. 
This resolvent formula is our version of Fukushima's \eqref{FR}.
\end{remark}
\begin{remark}\label{closed} By Corollary \ref{tilde}, in order to get the analogue of 
Theorem \ref{teointervallo} 
one should prove that if $A_{(\Pi,\Theta)}\in\ExM(A_{\min})$ then $\t f^{\Theta}_{\Pi_{\Lambda}}$ is always closable. In this case
\be
\ExM(A_{\min})=\{\t A_{(\Pi,B)}\,:\, (\Pi,B)\in\t\E(L^{2}(\Gamma)\,,\ f_{\Pi,B}\ \text{is a Dirichlet form}\,\}
\ee
and 
\be
\t R^{(\Pi,B)}_{\lambda}=R^{D}_{\lambda}+K_{\lambda}\Pi(-\Pi P_{\lambda}\Pi\dotplus B)^{-1}\Pi K_{\lambda}^{\star}\,.
\ee
Equivalently in Theorem \ref{finale2} one could avoid the hypothesis ``$F$ closed'' and  
substitute ``$f_{b}$ Markovian form'' with ``$f_{b}$ Dirichlet form''. However proving that $\t f^{\Theta}_{\Pi_{\Lambda}}$ is closable is a not trivial problem since $\Theta$ is positive and always unbounded (by $\D(\Theta)\subseteq H^{1}(\Gamma)$)  while $P_{0}$ is negative and unbounded. 
%Should one be able to show that  
%$\t f^{\Theta}_{\Pi_{\Lambda}}$ is closable then Theorem \ref{sub} would provide a complete classification of Markovian extensions of $A_{\min}$ in terms of Dirichlet forms on $L^{2}(\Gamma)$ of the kind 
%$f_{\Pi,B}$ with $(\Pi,B)\in\t\E(L^{2}(\Gamma))$. %However notice that in any case Theorem \ref{sub} provides as many Markovian self-adjoint extensions as the Dirichlet forms of the kind $f_{\Pi,B}$, $(\Pi,B)\in \t \E(L^{2}(\Gamma))$. 
\end{remark}

%\end{section}
\section{Wentzell-type boundary conditions}
In order to describe the self-adjoint extension of $A_{\min}$ corresponding to the bilinear  form $F_{(\Pi,\Theta)}$ we introduce the extended trace operators $\hat \rho$ and $\hat \tau$ as in Remark \ref{aggiunto}. By $\hat\gamma_0\, G_0h=\Lambda h$ and 
by Remark \ref{aggiunto} one has
\be
\hat \rho:\D(A_\max)\to L^{2}(\Gamma )\,,\quad \hat\rho 
u:=\Sigma\hat\gamma_0 u
\ee
and
\begin{equation}\label{tau2}
\hat \tau:\D(A_\max)\to L^{2}(\Gamma )\,,\quad
\hat\tau u:=\tau(u-G_0\Sigma\hat\gamma_0 u)
\equiv \Lambda(\hat\gamma_1
u-P_{0}\hat\gamma_0 u)\,.
\end{equation} 
Notice that even if both $\hat\gamma_1
u$ and $P_{0}\hat\gamma_0 u$ are in $H^{-\frac12}(\Gamma)$, by \eqref{tau2} one gets
\be
(\hat\gamma_1
u-P_{0}\hat\gamma_0 u)\in H^{\frac12}(\Gamma)\,.
\ee
Then, by \eqref{DA1}, one
has the following (this is our version of Theorem 4.1 in 
\cite{[G68]}, Chapter III)
\begin{theorem}\label{elliptic} 
$\Ex (A_\min)=\{A_{(\Pi,\Theta)}\,,\ (\Pi,\Theta)\in 
\E(L^{2}(\Gamma ))\}$, where 
\be
A_{(\Pi,\Theta)}:\D(A_{(\Pi,\Theta)})
\subseteq L^2(\Omega)\to L^2(\Omega)\,,\quad 
A_{(\Pi,\Theta)}u=A_\max u\,,
\ee
\begin{align}
%&
\D(A_{(\Pi,\Theta)})
%\\
=
%&
\{
u\in\D(A_{\max})\,:\, \Sigma\hat\gamma_0
u\in\D(\Theta)\,,\ \Pi\Lambda(\hat\gamma_1
u-P_{0}\hat\gamma_0 u)=\Theta\Sigma\hat\gamma_{0}u\}\,.
\end{align}
\end{theorem}  
\begin{remark} \label{remarkop}
Notice that, for all $s\in(0,2]$ 
\be
\D(A_{(\Pi,\Theta)})\subseteq H^s(\Omega)\quad\iff\quad 
\D(\Theta)\subseteq H^s(\Gamma )\,.
\ee
Indeed, 
since $u=u_0+G_0h$ with $u_0\in H^2(\Omega)\cap H^1_0(\Omega)$ 
and $h\in \D(\Theta)$, one has that $u\in H^s(\Omega)$ 
if and only if 
$G_0 h\in H^s(\Omega)$. By Remark \ref{regularity} 
$G_0 h\in H^s(\Omega)$ if and only if $h\in H^{s}(\Gamma )$.
\end{remark}
In the case $A\in \ExM (A_\min)$, the boundary conditions characterizing its domain are said  {\it Wentzell-type boundary conditions}. Such boundary conditions can be expressed by the 
boundary form $f_{b}$ appearing in Theorem \ref{finale}:
\begin{theorem}\label{tww} Let $A\in\widetilde{\Ex} (A_\min)$. Then $u\in \D(A)$ if and only if $u\in \D(A_{\max})\cap H^{1}(\Omega)$, $\gamma_{0}u\in \D(f_{b})$ and  
\begin{equation}\label{ww}
\forall h\in \D(f_{b})\,,\quad f_{b}(\gamma_{0}u,h)=(\hat\gamma_{1}u,h)_{-\frac12,\frac12}\,.
\end{equation}
The boundary conditions \eqref{ww} define a Markovian extension of $A_{\min}$, i.e. they 
are Wentzell-type, if and only if $f_{b}$ is a Markovian form on $L^{2}(\Gamma)$.
\end{theorem}
\begin{proof} By Theorem \ref{estensioni}, $A=A_{(\Pi,\Theta)}$ for some  $(\Pi,\Theta)\in \E(L^{2}(\Gamma))$. \par 
%By  Remark \ref{remarkforma} and Remark \ref{remarkop},  
If $A_{(\Pi,\Theta)}\in\widetilde{\Ex} (A_\min)$ then $\D(A_{(\Pi,\Theta)})\subseteq H^{1}(\Omega)$ and so $(\gamma_{0}u,\hat\gamma_{1}u)\in H^{\frac12}(\Gamma)\times H^{-\frac12}(\Gamma)$ for any $u\in \D(A_{(\Pi,\Theta)})$. \par 
By Theorem \ref{elliptic}, for all $u\in  \D(A_{(\Pi,\Theta)})$ and for all $h\in \D(f_{\Pi_{\Lambda},\Theta_{\Sigma}})\equiv\D(f_{b})$, 
\begin{align}
\langle(\hat\gamma_1
u-P_{0}\gamma_0 u),h\rangle_{L^{2}(\Gamma)}=&
\,(\hat\gamma_1
u,h)_{-\frac12,\frac12}-(P_{0}\gamma_0u ,h)_{-\frac12,\frac12}
%\\
=\langle\Theta \Sigma\gamma_0 u,\Sigma h\rangle_{L^{2}(\Gamma)}
\\
=&\,f_{\Pi_{\Lambda},\Theta_{\Sigma}}(\gamma_0 u,h)
\,.
\end{align}
Relations \eqref{ww} are then consequence of the definition of $f_{b}\equiv \t f^{\Theta}_{\Pi_{\Lambda}}$ (see Theorem \ref{converse}).\par
Conversely, if \eqref{ww} hold true, by $f_{b}\equiv \t f^{\Theta}_{\Pi_{\Lambda}}$ and the definition of $\t f^{\Theta}_{\Pi_{\Lambda}}$ in Theorem \ref{converse}, for all $h\in \Sigma \D(f_{b})=\Sigma\D(f_{\Pi_{\Lambda},\Theta_{\Sigma}})=\D(f_{\Pi,\Theta})$ one has 
\begin{align}
f_{\Pi,\Theta}(\Sigma\gamma_{0}u,h)=&f_{\Pi_{\Lambda},\Theta_{\Sigma}}(\gamma_{0}u, \Lambda h)=\langle (\hat\gamma_{1}u-P_{0}\gamma_{0}u),\Lambda h\rangle_{L^{2}(\Gamma)}\\
=&\langle \Lambda(\hat\gamma_{1}u-P_{0}\gamma_{0}u),h\rangle_{L^{2}(\Gamma)}\,.
\end{align}
This gives 
\be
\Sigma\gamma_0
u\in\D(\Theta)\,,\quad  \Pi\Lambda(\hat\gamma_1
u-P_{0}\gamma_0 u)=\Theta\Sigma\gamma_{0}u\,.
\ee
Finally, by Theorem \ref{finale}, $A\in \ExM(A_{\min})$, i.e. the boundary conditions are Wenzell-type,  if and only if  $f_{b}$ is Markovian.
\end{proof}
\begin{corollary} \label{cww} Let $(\Pi,B)\in
\t\E(L^{2}(\Gamma))$. Then $u\in D(\t A_{(\Pi,B)})$ if and only if $u\in \D(A_{\max})\cap H^{1}(\Omega)$, $\gamma_{0}u\in \D(f_{\Pi,B})$ and  
\begin{equation}\label{www}
\forall h\in \D(f_{\Pi,B})\cap H^{\frac12}(\Gamma)\,,\quad f_{\Pi,B}(\gamma_{0}u,h)=(\hat\gamma_{1}u,h)_{-\frac12,\frac12}\,.
\end{equation}
 The boundary conditions \eqref{www} define a Markovian extensions of $A_{\min}$, i.e they are Wentzell-type, if and only if $f_{\Pi,B}$ is a Dirichlet form on $L^{2}(\Gamma)$.
\end{corollary} 
\begin{remark}\label{rww} By Beurling-Deny decomposition, if $f_{B}$, $B\equiv(\uno,B)\in\t \E(L^{2}(\Gamma))$, is a regular Dirichlet form then the corresponding Wentzell-type boundary conditions are
\begin{align}
%\label{WBC}
&
f^{(c)}_{B}(\gamma_{0}u,h)
+\int_{\Gamma\times\Gamma}(\widetilde{\gamma_{0}u}(x)-
\widetilde{\gamma_{0}u}(y))(\t h(x)-\t h(y))\,dJ(x,y)
%\\
+
%&
\int_{\Gamma}\widetilde{\gamma_{0}u}(x)\t h(x)\,d\kappa(x)
\\
=
&
(\hat\gamma_{1}u,h)_{-\frac12,\frac12}\,.\nonumber
\end{align}  
These boundary conditions are similar (in weak form) to the ones appearing in Wentzell's 
seminal paper \cite{[we]} (also compare with the results given in \cite{[CF]}, Theorem 7.3.5).
%\par
%Notice that if $\kappa=0$ then the corresponding Markovian extension is recurrent (equivalently conservative).
\end{remark}
\begin{remark}
In the case $\Pi=\uno$ and $B$ is such that the form sum appearing in Lemma \ref{LRT} can be improved to an operator sum (this holds under the 
hypotheses given in the successive Remark \ref{tetaB}, see \cite{[P08]}, Example 5.5, \cite{[G68]}, Chapter III, Section 6) then the boundary conditions defining the domain of the extension $\t A_{(B)}\equiv \t A_{(\uno, B)}$ become
\be
\hat \gamma_{1} u=B\gamma_{0}u 
\ee
and the resolvent $\t R_{\lambda}^{(B)}$ of $A_{(B)}$ is given by the formula (see Lemma \ref{LRT} and \eqref{K*})
\be
\t R^{(B)}_{\lambda}=(\uno+K_{\lambda}(B-P_{\lambda} 
)^{-1}\gamma_{1} )R^{D}_{\lambda
}\,.
\ee
\end{remark}

\begin{remark}\label{tetaB} Suppose $B\in\B(H^{s}(\Gamma),H^{s-\alpha}(\Gamma))$ for all $s\ge-\frac 12$ and for some $\alpha>0$, and that
\be
B : H^{\alpha}(\Gamma)\subseteq L^{2}(\Gamma)\to L^{2}(\Gamma)
\ee
is self-adjoint and positive. Then $B\equiv(\uno,B)\in \t \E(L^{2}(\Gamma))$ and
\be
\D(f_{B})=H^{\frac{\alpha} 2}(\Gamma)\,,\quad f_{B}(h_{1},h_{2})=(Bh_{1},h_{2})_{-\frac{\alpha}2,\frac{\alpha} 2} \,.
\ee
Thus, by Corollary \ref{cww}, 
\be
\D(\t A_{(B)})=\{u\in\D(A_{\max})\cap H^{1}(\Omega):\gamma_{0}u\in H^{\frac{\alpha} 2}(\Gamma)\,,\ B\gamma_{0}u=\hat\gamma_{1}u\}\,.
\ee
If $\alpha<1$ then, by elliptic regularity, $u\in H^{\frac 32-\frac\alpha 2}(\Omega)$ and so in this case 
\be
\D(\t A_{(B)})=\{u\in\D(A_{\max})\cap H^{\frac 32-\frac\alpha 2}(\Omega): 
B\gamma_{0}u=\hat\gamma_{1}u\}\,.
\ee
If $\alpha\ge 1$ and supposing 
\be
\{h\in H^{-\frac 12}(\Gamma )
\,:\, Bh\in H^s(\Gamma )\}\subseteq 
H^{s+\alpha}(\Gamma )\,,
\ee
(this holds if $B$ is an elliptic pseudo-differential operator of order $\alpha$), then 
$\gamma_{0}u\in H^{-\frac12+\alpha}(\Omega)$ and so in this case 
\be
\D(\t A_{(B)})=\{u\in\D(A_{\max})\cap H^{2\wedge\alpha}(\Omega): 
B\gamma_{0}u=\hat\gamma_{1}u\}\,.
\ee
In particular, when $\alpha\ge 2$, one has 
\be
\D(\t A_{(B)})=\{u\in H^{2}(\Omega): 
B\gamma_{0}u=\gamma_{1}u\}\,.
\ee
By Theorem \ref{sub}, if  moreover $B$ is Markovian then $\t A_{(B)}$ is a Markovian extension of 
$A_{\min}$ and the boundary conditions $B\gamma_{0}u=\hat\gamma_{1}u$ are Wentzell-type. The corresponding Dirichlet form $\t F_{(B)}$ is given by 
\be
\D(\t F_{(B)} )=H^{1}(\Omega)\quad\text{whenever $\alpha\le 1$,}
\ee
\be
\D(\t F_{(B)} )=\{u\in H^{1}(\Omega):\gamma_{0}u\in H^{\frac\alpha 2}(\Gamma)\}\,,
\quad\text{whenever $\alpha> 1$,}
\ee
\be
\t F_{B}(u,v)=F_{N}(u,v)+(B\gamma_{0}u,\gamma_{0}v)_{-\frac{\alpha}2,\frac{\alpha} 2}\,. 
\ee
\end{remark}
\begin{remark}\label{rem5.8}
Let $\Pi:L^2(\Gamma)\to L^{2}(\Gamma)$ be an orthogonal projector  such that 
$\R(\Pi)\cap H^{\frac12}(\Gamma)$ is $L^2(\Gamma)$-dense in $\R(\Pi)$. Then for any positive $B\in\B(\R(\Pi))$ one has $(\Pi,B)\in \t\E(L^2(\Gamma))$. 
Since $\R(\Pi)\cap H^{\frac 12}(\Gamma)$ is a closed subspace of $H^{\frac 12}(\Gamma)$, 
we have the continuous projection $\Pi_{\circ}:H^{\frac 12}(\Gamma)\to H^{\frac 12}(\Gamma)$ such that $\R(\Pi_{\circ})=\R(\Pi)\cap H^{\frac 12}(\Gamma)$ and $\Pi_{\circ}|\R(\Pi)\cap H^{\frac 12}(\Gamma)=\Pi|\R(\Pi)\cap H^{\frac 12}(\Gamma)$. Denoting by $\Pi_{\circ}^{\star}:H^{-\frac 12}(\Gamma)\to H^{-\frac 12}(\Gamma)$ the dual
of $\Pi_{\circ}$ with respect to the duality $(\cdot,\cdot)_{-\frac 12,\frac 12}$, \eqref{www} gives $\Pi_{\circ}^{\star}\hat \gamma_1u\in \R(\Pi)$ and 
 the boundary conditions
\be
\Pi_{\circ}^{\star}\hat \gamma_1u=B\gamma_0u\,.
\ee
Hence in this case the operator domain of the corresponding self-adjoint extension is given by
\be
\D(\t A_{(\Pi,B)})=\{u\in \D(A_{\max})\cap H^{1}(\Omega)\,:\,\gamma_0 u\in \R(\Pi)\,,
\ \Pi_{\circ}^{\star}\hat \gamma_1u=B\gamma_0u\}\,.
\ee
In particular, in the case $\Pi=\uno$, one obtains, by elliptic regularity,
\be
\D(\t A_{(B)})=\{u\in \D(A_{\max})\cap H^{\frac 32}(\Omega)\,:\,\hat \gamma_1u=B\gamma_0u\}\,.
\ee
\end{remark}
\begin{example} Since $H^{s}(\Gamma)$ is a Dirichlet space for any $s\in[0,1]$, by Remark \ref{tetaB} one gets a Markovian 
extension $\t A_{(B)}$ by taking 
\be
B=-b_{1}\,\Delta_{LB}+b_{s} (-\Delta_{LB})^{s}+b_{0}\,,\quad 0<s<1\,,
\ b_{1},b_{s},b_{0}\ge 0\,.
\ee 
To such an extension correspond the Wentzell-type boundary conditions
\begin{equation}\label{WBC1}
b_{1}\,\Delta_{LB}\,\gamma_0 u-b_{s} (-\Delta_{LB})^{s}\gamma_0 u-b_{0}\gamma_0 u+\gamma_1 u=0\,.
\end{equation}
For the corresponding Markovian extension $\t A_{(B)}$ on has
\be
\D(\t A_{(B)})\subseteq H^{2}(\Omega)\quad\text{whenever  $b_{1}\not=0$}
\,,
\ee
\be
\D(\t A_{(B)})\subseteq H^{2s}(\Omega)\quad\text{whenever $b_{1}=0$, $b_{s}\not=0$ and 
$1/2\le s<1$}\,,
\ee
\be
\D(\t A_{(B)})\subseteq H^{\frac 32-s}(\Omega)\quad\text{whenever $b_{1}=0$, $b_{s}\not=0$ and 
$0<s\le 1/2$}\,.
\ee
Since $f_{B}$ is regular and satisfies the hypothesis given in Lemma \ref{Lregular} for any $0<s<1$ and $b_1,b_{s},b_{0}\ge 0$,  $\t F_{(B)}$ is a regular Dirichlet form. Moreover,
since the Beurling-Deny decomposition of the Dirichlet form $f_{B}$ is
\be
f^{(c)}_{B}(h_{1},h_{2})=b_{1}(-\Delta_{LB}h_{1},h_{2})_{-1,1}\,,\ 
\ee
\be
f^{(j)}_{B}(h_{1},h_{2})=b_{s}((-\Delta_{LB})^{s}h_{1},h_{2})_{-s,s}\,,\ 
\ee
\be
f^{(k)}_{B}(h_{1},h_{2})=b_{0}\langle h_{1},h_{2}\rangle_{L^{2}(\Gamma)}\,,
\ee
the Beurling-Deny decomposition of the Dirichlet form $\t F_{(B)}$ is (see Remark \ref{regular})
\be
\t F_{(B)}^{(c)}(u,v)=F_{N}(u,v)+b_{1}(-\Delta_{LB}\gamma_{0}u,\gamma_{0}v)_{-1,1}\,,
\ee
\be
\t F_{(B)}^{(j)}(u,v)=b_{s}((-\Delta_{LB})^{s}\gamma_{0}u,\gamma_{0}v)_{-s,s}\,,
\ee
\be
\t F_{(B)}^{(k)}(u,v)=b_{0}\langle \gamma_{0}u,\gamma_{0}v\rangle_{L^{2}(\Gamma)}\,.
\ee
Hence $\t F_{(B)}$ has the local property (and so $\Z_{\t A_{(B)}}$ is a Diffusion) if and only if $b_{s}=0$ and is strongly local  if and only 
if $b_{s}=b_{0}=0$.  The corresponding Markovian extension $\t A_{(B)}$ is recurrent if $b_{0}=0$ and transient otherwise.\par 
\end{example}
\begin{example} Let $\Omega\subset\RE^{n}$, $n\ge 2$, be open and bounded and such that $\Gamma$ is a smooth, compact, $n-1$ dimensional Lie group (for example  this is true if $\Omega$ is a solid torus, or a planar disc with $N\ge 0$ circular holes, or a four-dimensional ball).
%with unit $e$, $e  x=x$, $x x^{-1}=x^{-1} x=e$. 
Let $e$ denote the unit element and let $L_{1},\dots, L_{n-1}$ be a basis of left-invariant vector field in the corresponding 
Lie algebra.
% and pose $C_{l}^{2}(\Gamma)=\{h\in C^{2}(\Gamma)\,:\, X_{i}h\in C(\Gamma)\,,\ X_{i}X_{j}h\in C(\Gamma)\,,\ 1\le i,j<n\}$. 
Then there exist functions $\zeta_{i}\in C^{2}(\Gamma)$, 
$1\le i<n$, such that $\zeta_{i}(e)=0$, $L_{i}\zeta_{j}(e)=\delta_{ij}$ and 
$\zeta_{i}(x^{-1})=-\zeta_{i}(x)$. A measure $\nu$ on the Borel $\sigma$-algebra of $\Gamma$ 
is said to be a {\it L\'evy measure} whenever $\nu(\{e\})=0$ and 
$\int_{\Gamma}((\sum_{i=1}^{n-1}|\zeta_{i}(x)|^{2})\wedge 1)\,d\nu(x)<+\infty$, and is said to be symmetric if $\nu(E)=\nu(E^{-1})$ for any measurable $E$.\par
By Hunt's theorem (see \cite{[Hunt]}, \cite{[Liao]}; here we use the version provided in  \cite{[App]}, Theorem 2.1), any symmetric convolution semigroup of measures in $\Gamma$ has a generators given by a Markovian self-adjoint operator $B$ on $L^{2}(\Gamma)$ such that $C^{2}(\Gamma)\subset \D(B)$ and, for any $h\in C^{2}(\Gamma)$,
\begin{align}\label{hunt}
Bh(x)=\sum_{i,j=1}^{n-1}c_{ij}L_{i}L_{j}h(x)
+\frac12\int_{\Gamma}\left(h(xy)-2h(x)+h(xy^{-1})\right)d\nu(y)\,,
\end{align} 
where $c\equiv(c_{ij})$ is a constant, real-valued, not-negative-definite matrix and $\nu$ is a symmetric L\'evy measure. \par 
By Theorem \ref{sub} and Corollary \ref{cww}, the corresponding  Wentzell-type boundary conditions 
$B\gamma_{0}u=\hat\gamma_{1}u$ produce a Markovian extension $\t A_{(B)}$ which, since
$f_{B}(1)=0$, is recurrent. \par 
%Notice that, if $\int_{\Gamma}
%|\zeta_{i}(x)|\,d\nu(x)<+\infty$, for example if $\nu$ is finite, then
%, since 
%\begin{align}
%\int_{\Gamma\backslash\{e\}}
%\zeta_{i}(y)\,d\nu(y)=-\int_{\Gamma\backslash\{e\}}
%\zeta_{i}(y^{-1})\,d\nu(y)=-\int_{\Gamma\backslash\{e\}}
%\zeta_{i}(y)\,d\nu(y)=0\,,
%\end{align} 
%formula \eqref{hunt} simplifies to:
%\begin{align}
%Bh(x)=\sum_{i,j=1}^{n-1}c_{ij}X_{i}X_{j}h(x)
%+\int_{\Gamma}\left(h(xy)-h(x)\right)\,d\nu(y)\,.
%\end{align} 
By the Beurling-Deny decomposition of $f_{B}$ (see \cite{[App]}, Theorem 2.4) and by Remark \ref{regular1}, the Beurling-Deny decomposition of the Dirichlet form $\t F_{(B)}$ is 
\be
\t F^{(c)}_{(B)}(u,v)=F_{N}(u,v)+\sum_{i,j=1}^{n-1}c_{ij}\int_{\Gamma}L_{i}\gamma_{0}u(x)L_{j}\gamma_{0}v(x)\,d\sigma(x)\,,
\ee
\be
\t F^{(j)}_{(B)}(u,v)=\int_{\Gamma\times \Gamma}\left(\gamma_{0}u(x)-\gamma_{0}u(y)\right)\left(\gamma_{0}v(x)-\gamma_{0}v(y)\right)dJ(x,y)\,,
\ee
\be
\t F^{(k)}_{(B)}(u,v)=0\,,
\ee
where the measure $J$ is defined by 
\be
J(E_{1}\times E_{2}):=\int_{E_{2}}\nu(y^{-1}E_{1})\,d\sigma(y)\,.
\ee
\end{example}
\begin{example} Given the decomposition  $\Gamma=\Gamma_0\cup 
\Gamma_1$, $\Gamma_0$ open, let $\Pi:L^{2}(\Gamma)\to L^{2}(\Gamma)$ be the orthogonal projector $\Pi\,h:=1_{\Gamma_{0}}h$ and let $B=0$. Then 
\be
\R(\Pi)=\{h\in L^2(\Gamma): \text{supp}(h)\subseteq \bar\Gamma_{0}\}\simeq L^2(\Gamma_0)\,.
\ee
Since $C^\infty_c(\Gamma_0)$ is dense in $L^2(\Gamma_0)$ and 
\be
C^\infty_c(\Gamma_0)\simeq\{h\in C^{\infty}(\Gamma): \text{supp}(h)\subset \Gamma_{0}\}\subseteq \R(\Pi)\cap H^{\frac12}(\Gamma)\,,
\ee 
$\R(\Pi)\cap H^{\frac12}(\Gamma)$ is dense in $\R(\Pi)$. Then, by Remark \ref{rem5.8}, one gets the self-adjoint extension $A_{DN}$ with domain
\be
\D(A_{DN})=\{u\in\D(A_{\max})\cap H^{1}(\Omega):\text{supp$(\gamma_0 u)\subseteq \bar\Gamma_0$,\ supp$(\hat\gamma_1u)
\subseteq\Gamma_1$}\}\,.
\ee
The conditions appearing in $\D(A_{DN})$ are a weak form of the 
mixed Dirichlet-Neumann boundary conditions
\be
\begin{cases}
\gamma_{0}u=0\,,& \text{on $\Gamma_{1}$,}\\ 
\gamma_{1}u=0\,,& \text{on $\Gamma_{0}$.}
\end{cases}
\ee
Since $(1_{\Gamma_{0}}h)_\#=1_{\Gamma_{0}}h_{\#}$, by Theorem \ref{sub} and Corollary \ref{cww} such boundary conditions are Wentzell-type and define a (transient) Markovian extension with 
associated  Dirichlet form $F_{DN}$ given by
\be
\D(F_{DN})=\{u\in H^{1}(\Omega)\,:\, \text{supp$(\gamma_0 u)\subseteq \bar\Gamma_0$}\}\,,\quad F_{DN}(u,v)=F_{N}(u,v)\,.
\ee
Even if $F_{DN}$ is not regular on $\bar\Omega$, it has an associated Markov process: it is the part process 
of the reflecting diffusion associated to $F_{N}$ killed upon hitting $\Gamma_{1}$ (see Example 6.6.12 (ii) in \cite{[CF]}). 
\end{example}
\begin{example} Let $\R(\Pi)\subset L^{2}(\Gamma)$ be the one-dimensional subspace corresponding to the orthogonal projector $\Pi=(|\Gamma|^{-\frac12}1)\otimes (|\Gamma|^{-\frac12}1)$, where $|\Gamma|$ denotes the volume of the boundary $\Gamma$, and let $B=b\,\uno:\R(\Pi)\to\R(\Pi)$, $b\ge 0$. Then, by Remark \ref{rem5.8},  denoting by $\langle h\rangle$ the mean of $h\in H^{-\frac 12}(\Gamma)$ over $\Gamma$, i.e. $\langle h\rangle:=|\Gamma|^{-1}(h,1)_{-\frac 12, \frac 12}$,  one gets 
\begin{align}
\D(\t A_{(\Pi,B)})=&\{u\in \D(A_{\max})\cap H^{1}(\Omega):b\,\gamma_{0}u=\langle
\hat \gamma_{1}u\rangle\}\\
=&\{u\in H^{2}(\Omega):b\,\gamma_{0}u=\langle\gamma_{1}u\rangle\}\,.
\end{align}
The corresponding bilinear form is
\be
\D(\t F_{(\Pi,B)})=\{u\in H^{1}(\Omega):\gamma_{0}u=\text{const.}\}\,.
\ee
\be
\t F_{(\Pi,B)}(u,v)=F_{N}(u,v)+b\,\langle\gamma_{0}u\rangle\langle\gamma_{0}v\rangle\,.
\ee
Since $f_{\Pi,B}$ is obviously a Dirichlet form, $\t F_{(\Pi,B)}$ is a (not regular) Dirichlet form 
and so $\t A_{(\Pi,B)}$ is a Markovian extension and the boundary conditions $b\,\gamma_{0}u=\langle\gamma_{1}u\rangle$ are Wentzell-type. $\t A_{(\Pi,B)}$ is recurrent whenever $b=0$, otherwise it is transient.
\end{example}
%\end{section}
%\vfill\eject
\vskip10pt\noindent
\section*{Acknowlegments}
%\phantomsection
%\addcontentsline{toc}{section}{Acknowlegments}%
%\markboth{Acknowlegments}{Acknowlegments}%

The author is indebted to anonymous referees for useful remarks.

\end{document}